\documentclass[12pt]{amsart}

\usepackage{hyperref}
\usepackage{cleveref}
\usepackage{amssymb}
\usepackage{amsmath}
\usepackage{amsthm}
\usepackage{graphicx}
\usepackage{enumerate}
\usepackage{tikz}
\usetikzlibrary{arrows}
\usetikzlibrary{arrows.meta}
\tikzset{>={Latex[width=1.2mm,length=1.7mm]}}
\usepackage{tabularx}
\usepackage{mathrsfs}

\usepackage{amssymb, latexsym}
\usepackage{graphicx}
\usepackage{xfrac}
\usepackage{verbatim}
\usepackage{mathtools}
\usepackage{listings}
\usepackage{xurl}
\DeclarePairedDelimiter\ceil{\lceil}{\rceil}
\DeclarePairedDelimiter\floor{\lfloor}{\rfloor}

\DeclareSymbolFont{boldoperators}{OT1}{cmr}{bx}{n}
\SetSymbolFont{boldoperators}{bold}{OT1}{cmr}{bx}{n}
\edef\bar{\unexpanded{\protect\mathaccentV{bar}}\number\symboldoperators16}

\newtheorem{thm}{Theorem}[section]
\newtheorem{prop}[thm]{Proposition}
\newtheorem{cor}[thm]{Corollary}
\newtheorem{lem}[thm]{Lemma}
\newtheorem{obs}[thm]{Observation}
\newtheorem{conj}[thm]{Conjecture}
\newtheorem{prob}[thm]{Problem}
\theoremstyle{definition}
\newtheorem{defn}[thm]{Definition}%[section]
\newtheorem{alg}[thm]{Algorithm}

\numberwithin{equation}{section}

\newcommand{\sn}{\mathfrak{S}_n}

\newcommand{\mfs}[1]{\mathfrak{S}_{#1}}

\newcommand{\ol}[1]{\overline{#1}}

\newcommand{\A}{\mathcal{A}}
\newcommand{\D}{\mathcal{D}}
\newcommand{\U}{\mathcal{U}}

\newcommand{\imm}[1]{\mathrm{Imm}_{#1}}

\newcommand{\inv}{\textsc{inv}}

\newcommand{\defeq}{:=}

\newcommand{\sgn}{\mathrm{sgn}}

\newcommand{\snk}{\mathrm{snk}}

\newcommand{\ins}[1]{\mathrm{ins}_{#1}}
\newcommand{\inss}[1]{\mathrm{inss}_{#1}}

\newcommand{\tr}{{\negthickspace \top \negthickspace}}
\newcommand{\ntnsp}{\negthinspace}
\newcommand{\ntksp}{\negthickspace}

\newcommand{\bp}{\begin{prob}}
\newcommand{\ep}{\end{prob}}

\newcommand{\mat}[1]{\mathrm{Mat}_{#1 \times #1}}

\DeclareMathOperator{\perm}{per}

\newcommand{\sort}{\mathrm{sort}}

\newcommand{\permmon}[2]{#1_{1,#2_1} \ntnsp\cdots #1_{n,#2_n}}

\newcommand{\ssm}{\smallsetminus}

\newcommand{\precdot}{\prec \negthickspace \negthinspace \cdot\ }
\newcommand{\net}[1]{\mathcal F_{#1}}

\newcommand{\vertex}[1]{#1}

\newcommand{\F}{\mathcal F}

\begin{document}
\author{Sihong Pan, Mark Skandera and Jiayuan Wang}
\title[Permanental Inequalities, Unit Interval Orders, and Bruhat-Increasing Maps]{From Unit Interval Orders to Totally Nonnegative Matrices: Permanental Inequalities and Bruhat-Increasing Bijections}

\bibliographystyle{dart}

\date{\today}

\begin{abstract}
Given an $n\times n$ matrix $A=(a_{i,j})$ and subsets $I,J$ of
$[1,n]\defeq\{1,\dotsc,n\}$, let
$A_{I,J}=(a_{i,j})_{i\in I,j\in J}$ denote the $(I,J)$-submatrix of $A$.
We prove that, for every totally nonnegative matrix $A$ and every
$h=1,\dotsc,n-1$,
\begin{equation}
\perm(A_{[1,h],[1,h]})
\perm(A_{[h+1,n],[h+1,n]})
\;\geq\;
\perm(A_{[1,n]\cap 2\mathbb Z,[1,n]\cap 2\mathbb Z})
\perm(A_{[1,n]\ssm 2\mathbb Z,[1,n]\ssm 2\mathbb Z}).
\tag{$*$}\label{eq:absmain}
\end{equation}
In particular, these inequalities hold for the $0$-$1$ totally
nonnegative matrices arising as antiadjacency matrices of unit interval
orders.

Our proof is based on explicit Bruhat-order comparisons between two
distinguished families of permutations. Let
$\mathfrak A_n$ be the Young subgroup
\begin{equation*}
\mathfrak S_{1,\dotsc,\floor{\frac{n}{2}}}
\times
\mathfrak S_{\floor{\frac{n}{2}}+1,\dotsc,n},
\end{equation*}
and let $\mathfrak B_n$ be the set of parity-alternating permutations
\begin{equation*}
\{w=w_1\cdots w_n\in\mathfrak S_n
\mid i \text{ and } w_i \text{ have the same parity}\}.
\end{equation*}
For the balanced case $h=\floor*{\frac{n}{2}}$, we construct a
bijection
$f_n:\mathfrak A_n\rightarrow\mathfrak B_n$
satisfying
$w\leq f_n(w)$ in Bruhat order.
We then extend this construction to arbitrary $h$ by means of maps
$g_{n,h}$ satisfying the corresponding Bruhat-order comparison.
These maps yield the inequalities~\eqref{eq:absmain} for all totally
nonnegative matrices.
\end{abstract}

\maketitle

\section{Introduction}\label{s:intro}

Given an $n \times n$ matrix $A = (a_{i,j})$
and subsets $I,J \subseteq [n] \defeq \{1,\dotsc,n\}$,
define the submatrix $A_{I,J} = (a_{i,j})_{i \in I, j \in J}$ of $A$,
and call $\det(A_{I,J})$ a {\em minor} of $A$.
We call $A \in \mat n(\mathbb R)$ {\em totally nonnegative}
if each of its minors is nonnegative.

It is known that the defining inequalities $\det(A_{I,J}) \geq 0$
of totally nonnegative matrices imply many other polynomial
inequalities in the matrix entries.
Some of these involve {\em immanants} of $A$ and its submatrices,
i.e., linear combinations of the $n!$ monomials
$\{ \permmon aw \,|\, w \in \sn \}$.
In particular for any
function $f: \sn \rightarrow \mathbb R$, we define
the {\em $f$-immanant of $A$}~\cite{StanPos} to be the polynomial
\begin{equation}\label{eq:immdef}
  \imm f(A) \defeq \sum_{w \in \sn} f(w) \permmon aw
\end{equation}
generalizing the permanent and determinant
\begin{equation*}
  \perm(A) = \sum_{w \in \sn} \permmon aw,
  \qquad
  \det(A) = \sum_{w \in \sn} \sgn(w) \permmon aw.
  \end{equation*}

In particular, the study of polynomial inequalities satisfied
by the entries of Hermitian positive semidefinite matrices 
led to the following results for each totally nonnegative
$n \times n$ matrix $A$.
(See \cite[\S 1]{SkanSoskinBJArx} for more details.)
Koteljanskii~\cite{KotelProp} and Marcus~\cite{marcus1963perm}
showed that
we have 
  \begin{align}
  \det(A) &\leq a_{1,1}\cdots a_{n,n},\label{eq:kotel1}\\
  \perm(A) &\geq a_{1,1}\cdots a_{n,n}.\label{eq:marcus}
  \end{align}
  Fan~(see \cite{carlson1967}) and Lieb~\cite{LiebPerm}
  showed that for $I\subset [n]$ and its complement $\Bar I \defeq [n] \ssm I$
  we have
  \begin{align}
    \det(A) &\;\leq\; \det(A_{I,I})\det(A_{\Bar I,\Bar I}), \label{eq:carlson}\\
    \perm(A) &\;\geq\; \perm(A_{I,I})\perm(A_{\Bar I,\Bar I}).\label{eq:lieb1} 
\end{align}  
  Stembridge~\cite{StemImm} and Johnson~\cite[p.\,1088]{StemConj}
  showed that for each $\sn$-character $\chi$ and the identity element
  $e \in \sn$
  we have
  \begin{align}
    \det(A) &\;\leq\; \frac{\imm\chi(A)}{\chi(e)},\label{eq:stemchar}\\
\perm(A) &\;\geq\; \frac{\imm\chi(A)}{\chi(e)}.\label{eq:johnsonchar}
\end{align}

Inequalities (\ref{eq:kotel1}) -- (\ref{eq:johnsonchar}) suggest that
determinantal and permanental inequalities for totally nonnegative matrices
correspond bijectively,
with inequality directions reversed.
However, this is not in general true.
For instance,
Koteljanskii \cite{KotelProp}
showed that for all $I,J \subset [n]$ and
for each totally nonnegative matrix $A$, we have 
\begin{equation}\label{eq:kotel2}
  \det(A_{I\cup J, I\cup J})\det(A_{I\cap J, I\cap J}) \;\leq\;
  \det(A_{I,I}) \det(A_{J,J}),
\end{equation}
but the second author and Soskin~\cite[\S 6]{SkanSoskinBJArx}
showed the corresponding (reversed) inequality of permanents
not to hold in general.

It would therefore be interesting to understand which inequalities
change direction when determinants are replaced by permanents.
A natural starting point for such an investigation
is the second author's result~\cite[Thm.\,3.2]{SkanIneq}
which extends 
the inequalities (\ref{eq:carlson}) by providing an {\em upper bound}
for all products $\det(A_{I,I}) \det(A_{\Bar I, \Bar I})$:
  \begin{equation*}
    \det(A_{I,I}) \det(A_{\Bar I,\Bar I}) \;\leq\;
    \det(A_{[n] \cap 2\mathbb Z, [n] \cap 2\mathbb Z})
    \det(A_{[n] \ssm 2\mathbb Z, [n] \ssm 2\mathbb Z})
  \end{equation*}
  whenever $A$ is TNN. 
  It is natural to ask whether we can similarly extend the
  inequalities (\ref{eq:lieb1}) with a {\em lower bound} for all
  products $\perm(A_{I,I}) \perm(A_{\Bar I, \Bar I})$:
\begin{equation}\label{eq:evenoddintro}
    \perm(A_{I,I}) \perm(A_{\Bar I, \Bar I}) \;\geq\;
    \perm(A_{[n] \cap 2\mathbb Z, [n] \cap 2\mathbb Z})
    \perm(A_{[n] \ssm 2\mathbb Z, [n] \ssm 2\mathbb Z}).
  \end{equation}

In this paper, we prove that~\eqref{eq:evenoddintro} holds for every
totally nonnegative matrix $A$ whenever $I$ is an initial interval
$[h]$, with $h=1,\dotsc,n-1$. Thus, for every totally nonnegative
$n\times n$ matrix $A$ and every $h\in[n-1]$, we have
\begin{equation}\label{eq:intromain}
  \perm(A_{[1,h],[1,h]})
  \perm(A_{[h+1,n],[h+1,n]})
  \;\geq\;
  \perm(A_{[n]\cap 2\mathbb Z,[n]\cap 2\mathbb Z})
  \perm(A_{[n]\ssm 2\mathbb Z,[n]\ssm 2\mathbb Z}).
\end{equation}
On the other hand, we show that~\eqref{eq:evenoddintro} does not hold
for arbitrary subsets $I\subseteq[n]$.

In Section~\ref{s:fourcatalan}, we discuss families of combinatorial
objects called unit interval orders and descending star networks,
together with the totally nonnegative matrices associated to them.
In Section~\ref{s:permpathfam}, we give combinatorial interpretations
of permanents of these matrices.
In Section~\ref{s:merge}, we define a family of maps on descending
star networks and describe their effects on the permanents of the
corresponding matrices.
Using these constructions, we prove~\eqref{eq:intromain} for all
$h\in[n-1]$ when $A$ is a $0$-$1$ totally nonnegative matrix arising
from a unit interval order.
In Section~\ref{s:A gen on J}, we give an extension of this result.

The arguments in Sections~\ref{s:A gen on all}--\ref{s:all h} are largely independent of the unit interval order constructions developed in the preceding sections and may be read separately. In Section~\ref{s:A gen on all}, we reformulate the problem for
arbitrary totally nonnegative matrices in terms of Bruhat-order
comparisons between appropriate families of permutations.
In Section~\ref{s:proof of balance case}, we prove the balanced case
$h=\floor*{\frac{n}{2}}$ for all totally nonnegative matrices
\footnote{During the final preparation of this paper, we became aware
of an
\href{https://the-omega-institute.github.io/trureturing-mdbook/Problems/psw-bruhat-increasing-bijection.html}
{independent proof} formalized in Lean.
We thank Wenlin Zhang for bringing this work to our attention.}.
In Section~\ref{s:all h}, we extend this result to arbitrary
$h\in[n-1]$, thereby establishing~\eqref{eq:intromain} in full generality.
We present further open questions in Section~\ref{s:open}.

\section{Four families of sets counted by Catalan numbers}\label{s:fourcatalan}

A simple subset of $n \times n$ totally nonnegative matrices is the set of
$0$-$1$
matrices which are weakly increasing in columns, weakly decreasing in rows,
and in which all entries on or below the diagonal are $1$.
Call this set $\A_n$.
For example, $\A_4$ consists of
the fourteen matrices
\begin{equation}\label{eq:14aamatrices}
  \begin{gathered}
  \begin{bmatrix}
    1 & 1 & 1 & 1\\
    1 & 1 & 1 & 1\\
    1 & 1 & 1 & 1\\
    1 & 1 & 1 & 1
  \end{bmatrix}\ntksp,
  \quad
  \begin{bmatrix}
    1 & 1 & 1 & 0\\
    1 & 1 & 1 & 1\\
    1 & 1 & 1 & 1\\
    1 & 1 & 1 & 1
  \end{bmatrix}\ntksp,
  \quad
  \begin{bmatrix}
    1 & 1 & 0 & 0\\
    1 & 1 & 1 & 1\\
    1 & 1 & 1 & 1\\
    1 & 1 & 1 & 1
  \end{bmatrix}\ntksp,
  \quad
  \begin{bmatrix}
    1 & 1 & 1 & 0\\
    1 & 1 & 1 & 0\\
    1 & 1 & 1 & 1\\
    1 & 1 & 1 & 1
  \end{bmatrix}\ntksp,
  \quad
  \begin{bmatrix}
    1 & 0 & 0 & 0\\
    1 & 1 & 1 & 1\\
    1 & 1 & 1 & 1\\
    1 & 1 & 1 & 1
  \end{bmatrix}\ntksp,
  \\
  \begin{bmatrix}
    1 & 1 & 1 & 0\\
    1 & 1 & 1 & 0\\
    1 & 1 & 1 & 0\\
    1 & 1 & 1 & 1
  \end{bmatrix}\ntksp,
  \quad
  \begin{bmatrix}
    1 & 1 & 0 & 0\\
    1 & 1 & 1 & 0\\
    1 & 1 & 1 & 1\\
    1 & 1 & 1 & 1
  \end{bmatrix}\ntksp,
  \quad
  \begin{bmatrix}
    1 & 0 & 0 & 0\\
    1 & 1 & 0 & 0\\
    1 & 1 & 1 & 1\\
    1 & 1 & 1 & 1
  \end{bmatrix}\ntksp,
  \quad
  \begin{bmatrix}
    1 & 0 & 0 & 0\\
    1 & 1 & 1 & 0\\
    1 & 1 & 1 & 0\\
    1 & 1 & 1 & 1
  \end{bmatrix}\ntksp,
  \quad
  \begin{bmatrix}
    1 & 1 & 0 & 0\\
    1 & 1 & 0 & 0\\
    1 & 1 & 1 & 0\\
    1 & 1 & 1 & 1
  \end{bmatrix}\ntksp,
  \\
  \begin{bmatrix}
    1 & 1 & 0 & 0\\
    1 & 1 & 0 & 0\\
    1 & 1 & 1 & 1\\
    1 & 1 & 1 & 1
  \end{bmatrix}\ntksp,
  \quad
  \begin{bmatrix}
    1 & 0 & 0 & 0\\
    1 & 1 & 1 & 0\\
    1 & 1 & 1 & 1\\
    1 & 1 & 1 & 1
  \end{bmatrix}\ntksp,
  \quad
  \begin{bmatrix}
    1 & 1 & 0 & 0\\
    1 & 1 & 1 & 0\\
    1 & 1 & 1 & 0\\
    1 & 1 & 1 & 1
  \end{bmatrix}\ntksp,
  \quad
  \begin{bmatrix}
    1 & 0 & 0 & 0\\
    1 & 1 & 0 & 0\\
    1 & 1 & 1 & 0\\
    1 & 1 & 1 & 1
  \end{bmatrix}\ntksp.
  \end{gathered}
\end{equation}
For $A \in \A_n$, the boundary between entries of $A$ equal to $1$ and $0$
can be viewed as a $2n$-step {\em Dyck path},
consisting of $n$ horizontal
and $n$ vertical steps, none below the diagonal,
from the upper left corner of the
matrix to the lower right corner. 
For example, we can draw an $8$-step Dyck path on a matrix in $\A_4$. 

\begin{center}
\tikzset{every picture/.style={line width=0.75pt}} %set default line width to 0.75pt        

\begin{tikzpicture}[x=0.75pt,y=0.75pt,yscale=-.9,xscale=.9]

\draw  [dash pattern={on 0.84pt off 2.51pt}]  (280,80) -- (280,200) ;
\draw  [dash pattern={on 0.84pt off 2.51pt}][line width=0.75]  (120,40) -- (280,40) -- (280,200) -- (120,200) -- cycle ;
\draw [color={rgb, 255:red, 208; green, 2; blue, 27 }  ,draw opacity=1 ]   (120,40) -- (160,40) -- (160,80) -- (240,80) -- (240,160) -- (280,160) -- (280,200) ;

\draw (138,52.4) node [anchor=north west][inner sep=0.75pt]  [color={rgb, 255:red, 0; green, 0; blue, 0 }  ,opacity=1 ]  {$1$};
\draw (218,52.4) node [anchor=north west][inner sep=0.75pt]  [color={rgb, 255:red, 0; green, 0; blue, 0 }  ,opacity=1 ]  {$0$};
\draw (258,53.4) node [anchor=north west][inner sep=0.75pt]  [color={rgb, 255:red, 0; green, 0; blue, 0 }  ,opacity=1 ]  {$0$};
\draw (178,52.4) node [anchor=north west][inner sep=0.75pt]  [color={rgb, 255:red, 0; green, 0; blue, 0 }  ,opacity=1 ]  {$0$};
\draw (138,93.4) node [anchor=north west][inner sep=0.75pt]  [color={rgb, 255:red, 0; green, 0; blue, 0 }  ,opacity=1 ]  {$1$};
\draw (218,93.4) node [anchor=north west][inner sep=0.75pt]  [color={rgb, 255:red, 0; green, 0; blue, 0 }  ,opacity=1 ]  {$1$};
\draw (258,92.4) node [anchor=north west][inner sep=0.75pt]  [color={rgb, 255:red, 0; green, 0; blue, 0 }  ,opacity=1 ]  {$0$};
\draw (178,93.4) node [anchor=north west][inner sep=0.75pt]  [color={rgb, 255:red, 0; green, 0; blue, 0 }  ,opacity=1 ]  {$1$};
\draw (138,132.4) node [anchor=north west][inner sep=0.75pt]  [color={rgb, 255:red, 0; green, 0; blue, 0 }  ,opacity=1 ]  {$1$};
\draw (218,132.4) node [anchor=north west][inner sep=0.75pt]  [color={rgb, 255:red, 0; green, 0; blue, 0 }  ,opacity=1 ]  {$1$};
\draw (258,133.4) node [anchor=north west][inner sep=0.75pt]  [color={rgb, 255:red, 0; green, 0; blue, 0 }  ,opacity=1 ]  {$0$};
\draw (178,132.4) node [anchor=north west][inner sep=0.75pt]  [color={rgb, 255:red, 0; green, 0; blue, 0 }  ,opacity=1 ]  {$1$};
\draw (138,172.4) node [anchor=north west][inner sep=0.75pt]  [color={rgb, 255:red, 0; green, 0; blue, 0 }  ,opacity=1 ]  {$1$};
\draw (218,172.4) node [anchor=north west][inner sep=0.75pt]  [color={rgb, 255:red, 0; green, 0; blue, 0 }  ,opacity=1 ]  {$1$};
\draw (258,173.4) node [anchor=north west][inner sep=0.75pt]  [color={rgb, 255:red, 0; green, 0; blue, 0 }  ,opacity=1 ]  {$1$};
\draw (178,172.4) node [anchor=north west][inner sep=0.75pt]  [color={rgb, 255:red, 0; green, 0; blue, 0 }  ,opacity=1 ]  {$1$};

\end{tikzpicture}
\end{center}
The bijection between $\A_n$ and $2n$-step Dyck paths
implies that the cardinality of $\A_n$ is the $n$th Catalan number,
$\tfrac1{n+1}\tbinom{2n}{n}$.
(See \cite[Exercise 6.19 (i)]{StanEC2}.)

\begin{obs}\label{o:antiadjTNN}
  Each matrix in $\A_n$ 
  is totally nonnegative.
\end{obs}
\begin{proof}
  Fix $A \in \A_n$ and consider a square submatrix $B$ of $A$.
  If $B$ has a pair of equal columns or rows, then
  $\det(B) = 0$.  Otherwise, $B$ is unitriangular with
  $\det(B) = 1$.
\end{proof}

Elements of $\A_n$ correspond bijectively to $n$-element posets
called {\em unit interval orders}.
Call this set $\U_n$.
(See \cite[p.\,33]{Fish}, \cite[Exercise 6.19 (ddd)]{StanEC2}, \cite[\S 8.2]{Trott}.)
These are the posets in which no induced four-element subposet
is isomorphic to either of the posets
\begin{equation*}
  \mathbf 3 + \mathbf 1 \;=\;
\begin{tikzpicture}[scale=.50,baseline=-3]
\draw[fill] (0,1) circle (1.2mm); %\node at (0,1.5) {$\scriptstyle{4}$};
\draw[fill] (0,0) circle (1.2mm); %\node at (-.8,0) {$\scriptstyle{2}$};
\draw[fill] (.8,0) circle (1.2mm); %\node at (.8,0) {$\scriptstyle{3}$};
\draw[fill] (0,-1) circle (1.2mm); %\node at (0,-1.5) {$\scriptstyle{1}$};
\draw[-] (0,-1) -- (0,0) -- (0,1);
\end{tikzpicture}\; ,
  \qquad
  \mathbf 2 + \mathbf 2 \;=\;
  \begin{tikzpicture}[scale=.50,baseline=-3]
\draw[fill] (0,.5) circle (1.2mm); %\node at (0,1) {$\scriptstyle{3}$};
\draw[fill] (0,-.5) circle (1.2mm); %\node at (0,-1) {$\scriptstyle{1}$};
\draw[fill] (.8,-.5) circle (1.2mm); %\node at (0.8,-1) {$\scriptstyle{2}$};
\draw[fill] (.8,.5) circle (1.2mm); %\node at (.8,1) {$\scriptstyle{4}$};
\draw[-] (0,.5) -- (0,-.5);
\draw[-] (.8,.5) -- (.8,-.5);
\end{tikzpicture}\; .
\end{equation*}
The bijection $A \mapsto P(A)$ is given by
declaring $i \leq_{P(A)} j$ if and only if $a_{i,j} = 0$.
Each matrix $A$ is said to be the {\em antiadjacency matrix} of $P(A)$.
For example, the fourteen
matrices (\ref{eq:14aamatrices})
are antiadjacency matrices of the unit interval orders
\begin{equation}\label{eq:14uios}
\begin{tikzpicture}[scale=.50,baseline=0]
\draw[fill] (0,0) circle (1.2mm); \node at (0.,.5) {$\scriptstyle{1}$};
\draw[fill] (.6,0) circle (1.2mm); \node at (0.6,.5) {$\scriptstyle{2}$};
\draw[fill] (1.2,0) circle (1.2mm); \node at (1.2,.5) {$\scriptstyle{3}$};
\draw[fill] (1.8,0) circle (1.2mm); \node at (1.8,.5) {$\scriptstyle{4}$};
\end{tikzpicture}
\
\begin{tikzpicture}[scale=.50,baseline=0]
\draw[fill] (0,.5) circle (1.2mm); \node at (0,1) {$\scriptstyle{4}$};
\draw[fill] (0,-.5) circle (1.2mm); \node at (0,-1) {$\scriptstyle{1}$};
\draw[fill] (.6,0) circle (1.2mm); \node at (0.6,.5) {$\scriptstyle{2}$};
\draw[fill] (1.2,0) circle (1.2mm); \node at (1.2,.5) {$\scriptstyle{3}$};
\draw[-] (0,.5) -- (0,-.5);
\end{tikzpicture}
\
\begin{tikzpicture}[scale=.50,baseline=0]
\draw[fill] (-.3,.5) circle (1.2mm); \node at (-.3,1) {$\scriptstyle{3}$};
\draw[fill] (.3,.5) circle (1.2mm); \node at (.3,1) {$\scriptstyle{4}$};
\draw[fill] (0,-.5) circle (1.2mm); \node at (0,-1) {$\scriptstyle{1}$};
\draw[fill] (.8,0) circle (1.2mm); \node at (0.8,.5) {$\scriptstyle{2}$};
\draw[-] (-.3,.5) -- (0,-.5);
\draw[-] (.3,.5) -- (0,-.5);
\end{tikzpicture}
\
\begin{tikzpicture}[scale=.50,baseline=0]
\draw[fill] (-.3,-.5) circle (1.2mm); \node at (-.3,-1) {$\scriptstyle{1}$};
\draw[fill] (.3,-.5) circle (1.2mm); \node at (.3,-1) {$\scriptstyle{2}$};
\draw[fill] (0,.5) circle (1.2mm); \node at (0,1) {$\scriptstyle{4}$};
\draw[fill] (.8,0) circle (1.2mm); \node at (0.8,.5) {$\scriptstyle{3}$};
\draw[-] (-.3,-.5) -- (0,.5);
\draw[-] (.3,-.5) -- (0,.5);
\end{tikzpicture}
\
\begin{tikzpicture}[scale=.50,baseline=0]
\draw[fill] (-.5,.5) circle (1.2mm); \node at (-.5,1) {$\scriptstyle{2}$};
\draw[fill] (0,.5) circle (1.2mm); \node at (0,1) {$\scriptstyle{3}$};
\draw[fill] (.5,.5) circle (1.2mm); \node at (0.5,1) {$\scriptstyle{4}$};
\draw[fill] (0,-.5) circle (1.2mm); \node at (0,-1) {$\scriptstyle{1}$};
\draw[-] (0,-.5) -- (-.5,.5);
\draw[-] (0,-.5) -- (0,.5);
\draw[-] (0,-.5) -- (.5,.5);
\end{tikzpicture}
\
\begin{tikzpicture}[scale=.50,baseline=0]
\draw[fill] (-.5,-.5) circle (1.2mm); \node at (-.5,-1) {$\scriptstyle{1}$};
\draw[fill] (0,-.5) circle (1.2mm); \node at (0,-1) {$\scriptstyle{2}$};
\draw[fill] (.5,-.5) circle (1.2mm); \node at (0.5,-1) {$\scriptstyle{3}$};
\draw[fill] (0,.5) circle (1.2mm); \node at (0,1) {$\scriptstyle{4}$};
\draw[-] (0,.5) -- (-.5,-.5);
\draw[-] (0,.5) -- (0,-.5);
\draw[-] (0,.5) -- (.5,-.5);
\end{tikzpicture}
\ \,
\begin{tikzpicture}[scale=.50,baseline=0]
\draw[fill] (0,.5) circle (1.2mm); \node at (0,1) {$\scriptstyle{3}$};
\draw[fill] (0,-.5) circle (1.2mm); \node at (0,-1) {$\scriptstyle{1}$};
\draw[fill] (.8,-.5) circle (1.2mm); \node at (0.8,-1) {$\scriptstyle{2}$};
\draw[fill] (.8,.5) circle (1.2mm); \node at (.8,1) {$\scriptstyle{4}$};
\draw[-] (0,.5) -- (0,-.5);
\draw[-] (.8,.5) -- (.8,-.5);
\draw[-] (0,-.5) -- (.8,.5);
\end{tikzpicture}
\
\begin{tikzpicture}[scale=.50,baseline=0]
\draw[fill] (-.3,1) circle (1.2mm); \node at (-.3,1.5) {$\scriptstyle{3}$};
\draw[fill] (.3,1) circle (1.2mm); \node at (.3,1.5) {$\scriptstyle{4}$};
\draw[fill] (0,0) circle (1.2mm); \node at (-.5,0) {$\scriptstyle{2}$};
\draw[fill] (0,-1) circle (1.2mm); \node at (0,-1.5) {$\scriptstyle{1}$};
\draw[-] (0,-1) -- (0,0) -- (-.3,1);
\draw[-]           (0,0) -- (.3,1);
\end{tikzpicture}
\
\begin{tikzpicture}[scale=.50,baseline=0]
\draw[fill] (0,1) circle (1.2mm); \node at (0,1.5) {$\scriptstyle{4}$};
\draw[fill] (-.3,0) circle (1.2mm); \node at (-.8,0) {$\scriptstyle{2}$};
\draw[fill] (.3,0) circle (1.2mm); \node at (.8,0) {$\scriptstyle{3}$};
\draw[fill] (0,-1) circle (1.2mm); \node at (0,-1.5) {$\scriptstyle{1}$};
\draw[-] (0,-1) -- (-.3,0) -- (0,1);
\draw[-] (0,-1) -- (.3,0) -- (0,1);
\end{tikzpicture}
\
\begin{tikzpicture}[scale=.50,baseline=0]
\draw[fill] (-.3,-1) circle (1.2mm); \node at (-.3,-1.5) {$\scriptstyle{1}$};
\draw[fill] (.3,-1) circle (1.2mm); \node at (.3,-1.5) {$\scriptstyle{2}$};
\draw[fill] (0,0) circle (1.2mm); \node at (.5,0) {$\scriptstyle{3}$};
\draw[fill] (0,1) circle (1.2mm); \node at (0,1.5) {$\scriptstyle{4}$};
\draw[-] (0,1) -- (0,0) -- (-.3,-1);
\draw[-]           (0,0) -- (.3,-1);
\end{tikzpicture}
\
\begin{tikzpicture}[scale=.50,baseline=0]
\draw[fill] (0,.5) circle (1.2mm); \node at (0,1) {$\scriptstyle{3}$};
\draw[fill] (0,-.5) circle (1.2mm); \node at (0,-1) {$\scriptstyle{1}$};
\draw[fill] (.8,-.5) circle (1.2mm); \node at (0.8,-1) {$\scriptstyle{2}$};
\draw[fill] (.8,.5) circle (1.2mm); \node at (.8,1) {$\scriptstyle{4}$};
\draw[-] (0,.5) -- (0,-.5);
\draw[-] (.8,.5) -- (.8,-.5);
\draw[-] (0,-.5) -- (.8,.5);
\draw[-] (0,.5) -- (.8,-.5);
\end{tikzpicture}
\
\begin{tikzpicture}[scale=.50,baseline=0]
\draw[fill] (.3,-1) circle (1.2mm); \node at (.3,-1.5) {$\scriptstyle{1}$};
\draw[fill] (0,0) circle (1.2mm); \node at (-.5,0) {$\scriptstyle{2}$};
\draw[fill] (0,1) circle (1.2mm); \node at (0,1.5) {$\scriptstyle{4}$};
\draw[fill] (.6,.5) circle (1.2mm); \node at (.6,1) {$\scriptstyle{3}$};
\draw[-] (0,1) -- (0,0) -- (.3,-1);
\draw[-]           (.6,.5) -- (.3,-1);
\end{tikzpicture}
\
\begin{tikzpicture}[scale=.50,baseline=0]
\draw[fill] (.3,1) circle (1.2mm); \node at (.3,1.5) {$\scriptstyle{4}$};
\draw[fill] (0,0) circle (1.2mm); \node at (-.5,0) {$\scriptstyle{3}$};
\draw[fill] (0,-1) circle (1.2mm); \node at (0,-1.5) {$\scriptstyle{1}$};
\draw[fill] (.6,-.5) circle (1.2mm); \node at (.6,-1) {$\scriptstyle{2}$};
\draw[-] (0,-1) -- (0,0) -- (.3,1);
\draw[-]           (.6,-.5) -- (.3,1);
\end{tikzpicture}
\
\begin{tikzpicture}[scale=.50,baseline=0]
\draw[fill] (0,1.5) circle (1.2mm); \node at (-.5,1.5) {$\scriptstyle{4}$};
\draw[fill] (0,.5) circle (1.2mm); \node at (-.5,.5) {$\scriptstyle{3}$};
\draw[fill] (0,-.5) circle (1.2mm); \node at (-.5,-.5) {$\scriptstyle{2}$};
\draw[fill] (0,-1.5) circle (1.2mm); \node at (-.5,-1.5) {$\scriptstyle{1}$};
\draw[-] (0,1.5) -- (0,.5) -- (0,-.5) -- (0,-1.5);
\end{tikzpicture}\; ,
\end{equation}
respectively.
Labels, shown here to indicate the correspondence between poset elements
and matrix rows and columns, can be computed canonically from an unlabeled
unit interval order:
\begin{alg}\label{a:arlabeling}
  Given unlabeled unit interval order $P \in \U_n$, do
  \begin{enumerate}
  \item For each element $y \in P$, compute
  $\beta(y) \defeq
  \# \{ x \in P \,|\, x \leq_P y \} - \# \{ z \in P \,|\, z \geq_P y \}$.
\item Label the poset elements $1,\dotsc,n$
so that we have
  $\beta(1) \leq \cdots \leq \beta(n)$.
\end{enumerate}
\end{alg}

Entries of antiadjacency matrices which correspond to peaks in
the Dyck paths mentioned before Observation~\ref{o:antiadjTNN}
can be interpreted in terms of antichains the corresponding posets.
\begin{prop}\label{p:peak}
  Let $A$ be the antiadjacency matrix of unit interval order $P$,
  labeled as in Algorithm~\ref{a:arlabeling}.  Then an interval $[c,d]$
  is a maximal antichain in $P$ if and only if $a_{c,d} = 1$ and this entry is
  \begin{enumerate}
  \item the rightmost nonzero entry in row $c$ of $A$,
  \item the first nonzero entry in column $d$ of $A$.
  \end{enumerate}
\end{prop}
\begin{proof}
  
  ($\Rightarrow$) We assume that the interval $[c,d]$ is a maximal antichain in $P$. Then, $c$ and $d$ are incomparable in $P$, hence $a_{c,d} = 1$. Also, since $[c,d]$ is a maximal, $c$ and $d+j$ are comparable for all $j>0$, hence, $[c,d]$ is the rightmost nonzero entry in row $c$ of $A$. Similarly, $c-i$ and $d$ are comparable for all $i>0$, hence $a_{c,d}$ is the first nonzero entry in column $d$ of $A$. 

  ($\Leftarrow$) Next, we assume that $a_{c,d} = 1$ and assume (1) and (2). By the definition of $A$, $c$ is incomparable to $c+1,c+2,\cdots d$. Therefore, $[c,d]$ is an antichain in $P$. Since $a_{c,d}$ is the rightmost nonzero entry in row $c$ of $A$, $c$ is comparable to $d+1$, $d+2$, etc. Similarly, since $a_{c,d}$ is the first nonzero entry in column $d$ of $A$, $d$ is comparable to $c-1$, $c-2$, etc. Therefore, $[c,d]$ is a maximal and by the labeling from Algorithm~\ref{a:arlabeling}, it is a maximal antichain in $P$. 
\end{proof}

We will show in Section~\ref{s:main}
that matrices in $\A_n$ satisfy the inequalities (\ref{eq:evenoddintro}).
To do this, we relate unit interval orders and their antiadjacency matrices
to a
second family of combinatorical objects counted by Catalan numbers
and to a second family of totally nonnegative matrices.

The definition of this second family of combinatorial objects is
a bit more intricate than that of unit interval orders.
Define a {\em planar network of order $n$}
to be a directed, planar,
acyclic
multigraph
which can be embedded in a disc so that
$2n$ boundary vertices can be labeled clockwise as
{\em source} $1, \dotsc,$ {\em source} $n$,
{\em sink} $n, \dotsc,$ {\em sink} $1$.
We will assume all sources to have indegree $0$ and outdegree $1$,
and all sinks to have indegree $1$ and outdegree $0$.
Let $\net n$ denote the set of such networks.

For each subinterval $[a,b]$ of $[n]$ we define a {\em simple star network}
$F_{[a,b]} \in \net n$ by
\begin{enumerate}
\item Sources $1,\dotsc,n$ lie on a vertical line to the left;
  sinks $1,\dotsc,n$ lie on a vertical line to the right.
  Both are labeled from bottom to top.
\item An interior vertex lies between the sources and sinks.
\item For $i = 1, \dotsc, a-1$ and $i = b+1, \dotsc, n$, a directed edge
  begins at source $i$ and terminates at sink $i$.
\item For $i = a, \dotsc, b$, a directed edge begins at source $i$ and
  terminates at the interior vertex, and another directed edge begins
  at the interior vertex and terminates at sink $i$.
\end{enumerate}
For zero- and one-element subintervals
we define the trivial network $F_{\emptyset} = F_{[1,1]} = \cdots = F_{[n,n]}$
to have no interior vertex, and $n$ horizontal edges, each from source $i$
to sink $i$, for $i = 1,\dotsc,n$.
For example,
the element $F_{\smash{[2,4]}} \in \net 4$ is
\begin{equation*}
\begin{tikzpicture}[scale=.5,baseline=15]
\node at (-2.5,2.5) {$\scriptstyle{ \mathrm{source}\;4}$};
\node at (-2.5,1.5) {$\scriptstyle{ \mathrm{source}\;3}$};
\node at (-2.5,0.5) {$\scriptstyle{ \mathrm{source}\;2}$};
\node at (-2.5,-0.5) {$\scriptstyle{ \mathrm{source}\;1}$};  
\node at (2.25,2.5) {$\scriptstyle{ \mathrm{sink}\;4}$};
\node at (2.25,1.5) {$\scriptstyle{ \mathrm{sink}\;3}$};
\node at (2.25,0.5) {$\scriptstyle{ \mathrm{sink}\;2}$};
\node at (2.25,-0.5) {$\scriptstyle{ \mathrm{sink}\;1}$};  
\draw[->,>=stealth'] (-1,2.5) -- (-.2,1.65);
\draw[->,>=stealth'] (-1,1.5) -- (-.2,1.5);
\draw[->,>=stealth'] (-1,0.5) -- (-.2,1.35);
\draw[->,>=stealth'] (0,1.5) -- (.8,2.35);
\draw[->,>=stealth'] (0,1.5) -- (.8,1.5);
\draw[->,>=stealth'] (0,1.5) -- (.8,0.65);
\draw[->,>=stealth'] (-1,-0.5) -- (.8,-0.5);
\node at (-1,2.5) {$\bullet$}; 
\node at (-1,1.5) {$\bullet$}; 
\node at (-1,0.5) {$\bullet$}; 
\node at (-1,-0.5) {$\bullet$};
\node at (1,2.5) {$\bullet$}; 
\node at (1,1.5) {$\bullet$}; 
\node at (1,0.5) {$\bullet$}; 
\node at (1,-0.5) {$\bullet$};
\node at (0,1.5) {$\bullet$};
\end{tikzpicture}
\ .
\end{equation*}  
For economy, we will omit edge orientations and vertices from
drawings of planar networks.
Thus we will draw the
seven simple star networks in $\net 4$ as
\begin{equation}\label{eq:simplestarnets}
\begin{gathered}
   \begin{tikzpicture}[scale=.5,baseline=-20]
\node at (-.4,0) {$\scriptstyle 4$};
\node at (-.4,-1) {$\scriptstyle 3$};
\node at (-.4,-2) {$\scriptstyle{2}$};
\node at (-.4,-3) {$\scriptstyle 1$};  
\node at (1.4,0) {$\scriptstyle 4$};
\node at (1.4,-1) {$\scriptstyle 3$};
\node at (1.4,-2) {$\scriptstyle{2}$};
\node at (1.4,-3) {$\scriptstyle 1$};  
\draw[-] (0,0) -- (1,-3);
\draw[-] (0,-1) -- (1,-2);
\draw[-] (0,-2) -- (1,-1);
\draw[-] (0,-3) -- (1,0);
\end{tikzpicture}%,
,\\
\phantom{\sum}\ntksp F_{[1,4]}\phantom{\sum}
\end{gathered}
\begin{gathered}
\begin{tikzpicture}[scale=.5,baseline=-20]
\node at (-.4,0) {$\scriptstyle 4$};
\node at (-.4,-1) {$\scriptstyle 3$};
\node at (-.4,-2) {$\scriptstyle{2}$};
\node at (-.4,-3) {$\scriptstyle 1$};  
\node at (1.4,0) {$\scriptstyle 4$};
\node at (1.4,-1) {$\scriptstyle 3$};
\node at (1.4,-2) {$\scriptstyle{2}$};
\node at (1.4,-3) {$\scriptstyle 1$};  
\draw[-] (0,0) -- (1,-2);
\draw[-] (0,-1) -- (1,-1);
\draw[-] (0,-2) -- (1,0);
\draw[-] (0,-3) -- (1,-3);
\end{tikzpicture},\\
   \phantom{\sum}\ntksp F_{[2,4]}\phantom{\sum}
 \end{gathered}
\begin{gathered}
\begin{tikzpicture}[scale=.5,baseline=-20]
\node at (-.4,0) {$\scriptstyle 4$};
\node at (-.4,-1) {$\scriptstyle 3$};
\node at (-.4,-2) {$\scriptstyle{2}$};
\node at (-.4,-3) {$\scriptstyle 1$};  
\node at (1.4,0) {$\scriptstyle 4$};
\node at (1.4,-1) {$\scriptstyle 3$};
\node at (1.4,-2) {$\scriptstyle{2}$};
\node at (1.4,-3) {$\scriptstyle 1$};  
\draw[-] (0,0) -- (1,0);
\draw[-] (0,-1) -- (1,-3);
\draw[-] (0,-2) -- (1,-2);
\draw[-] (0,-3) -- (1,-1);
\end{tikzpicture},\\
   \phantom{\sum}\ntksp F_{[1,3]}\phantom{\sum}
 \end{gathered}
\begin{gathered}
\begin{tikzpicture}[scale=.5,baseline=-20]
\node at (-.4,0) {$\scriptstyle 4$};
\node at (-.4,-1) {$\scriptstyle 3$};
\node at (-.4,-2) {$\scriptstyle{2}$};
\node at (-.4,-3) {$\scriptstyle 1$};  
\node at (1.4,0) {$\scriptstyle 4$};
\node at (1.4,-1) {$\scriptstyle 3$};
\node at (1.4,-2) {$\scriptstyle{2}$};
\node at (1.4,-3) {$\scriptstyle 1$};  
\draw[-] (0,0) -- (1,-1);
\draw[-] (0,-1) -- (1,0);
\draw[-] (0,-2) -- (1,-2);
\draw[-] (0,-3) -- (1,-3);
\end{tikzpicture},\\
   \phantom{\sum}\ntksp F_{[3,4]}\phantom{\sum}
 \end{gathered}
\begin{gathered}
\begin{tikzpicture}[scale=.5,baseline=-20]
\node at (-.4,0) {$\scriptstyle 4$};
\node at (-.4,-1) {$\scriptstyle 3$};
\node at (-.4,-2) {$\scriptstyle{2}$};
\node at (-.4,-3) {$\scriptstyle 1$};  
\node at (1.4,0) {$\scriptstyle 4$};
\node at (1.4,-1) {$\scriptstyle 3$};
\node at (1.4,-2) {$\scriptstyle{2}$};
\node at (1.4,-3) {$\scriptstyle 1$};  
\draw[-] (0,0) -- (1,0);
\draw[-] (0,-1) -- (1,-2);
\draw[-] (0,-2) -- (1,-1);
\draw[-] (0,-3) -- (1,-3);
\end{tikzpicture},\\
   \phantom{\sum}\ntksp F_{[2,3]}\phantom{\sum}
 \end{gathered}
\begin{gathered}
\begin{tikzpicture}[scale=.5,baseline=-20]
\node at (-.4,0) {$\scriptstyle 4$};
\node at (-.4,-1) {$\scriptstyle 3$};
\node at (-.4,-2) {$\scriptstyle{2}$};
\node at (-.4,-3) {$\scriptstyle 1$};  
\node at (1.4,0) {$\scriptstyle 4$};
\node at (1.4,-1) {$\scriptstyle 3$};
\node at (1.4,-2) {$\scriptstyle{2}$};
\node at (1.4,-3) {$\scriptstyle 1$};  
\draw[-] (0,0) -- (1,0);
\draw[-] (0,-1) -- (1,-1);
\draw[-] (0,-2) -- (1,-3);
\draw[-] (0,-3) -- (1,-2);
\end{tikzpicture},\\
   \phantom{\sum}\ntksp F_{[1,2]}\phantom{\sum}
 \end{gathered}
\
\begin{gathered}
\begin{tikzpicture}[scale=.5,baseline=-20]
\node at (-.4,0) {$\scriptstyle 4$};
\node at (-.4,-1) {$\scriptstyle 3$};
\node at (-.4,-2) {$\scriptstyle{2}$};
\node at (-.4,-3) {$\scriptstyle 1$};  
\node at (1.4,0) {$\scriptstyle 4$};
\node at (1.4,-1) {$\scriptstyle 3$};
\node at (1.4,-2) {$\scriptstyle{2}$};
\node at (1.4,-3) {$\scriptstyle 1$};  
\draw[-] (0,0) -- (1,0);
\draw[-] (0,-1) -- (1,-1);
\draw[-] (0,-2) -- (1,-2);
\draw[-] (0,-3) -- (1,-3);
\end{tikzpicture},\\
\phantom{\sum}\ntksp F_\emptyset \phantom{\sum}
 \end{gathered}
\end{equation}
where $F_\emptyset = F_{[1,1]} = F_{[2,2]} = F_{[3,3]} = F_{[4,4]}$.
When there is no danger of confusion, we will omit source and sink labels
as well. 

The set $\net n$ is generated by simple star networks and
two types of concatenation operations.
Given networks $E, F \in \net n$,
define
the (ordinary) concatenation $E \circ F$
as follows.  For $i = 1,\dotsc, n$, do
\begin{enumerate}
\item remove sink $i$ of $E$ and source $i$ of $F$,
\item merge each edge $(\vertex x, \text{sink }i)$ in $E$ with each edge
  $( \text{source }i, \vertex y )$ in $F$
  to form a single edge $(\vertex x, \vertex y)$ in $E \circ F$.
\end{enumerate}
Sometimes in a concatenation $E \circ F$,
there may exist vertices $\vertex x$ in $E$, $\vertex y$ in $F$
with more than one edge incident upon both.  Let
$m(\vertex x, \vertex y)$ be this number.
Define the {\em condensed concatenation}
$E \bullet F$ to be the
subdigraph of $E \circ F$ obtained
by removing, for all such pairs $(\vertex x, \vertex y)$,
all but one of the $m(\vertex x, \vertex y)$ edges incident upon both.
For example, in $\net 4$ we have the isomorphic graphs
\begin{equation}\label{eq:allbutone1}
F_{[3,4]} \circ F_{[2,3]} =
\begin{tikzpicture}[scale=.5,baseline=0]
\node at (-.4,1.5) {$\scriptstyle 4$};
\node at (-.4,0.5) {$\scriptstyle 3$};
\node at (-.4,-0.5) {$\scriptstyle{2}$};
\node at (-.4,-1.5) {$\scriptstyle{1}$};  
\node at (2.4,1.5) {$\scriptstyle 4$};
\node at (2.4,0.5) {$\scriptstyle 3$};
\node at (2.4,-0.5) {$\scriptstyle{2}$};
\node at (2.4,-1.5) {$\scriptstyle{1}$};  
\draw[-] (0,1.5) -- (1,0.5) -- (2,-0.5);
\draw[-] (0,0.5) -- (1,1.5) -- (2,1.5);
\draw[-] (0,-0.5) -- (1,-0.5) -- (2,0.5);
\draw[-] (0,-1.5) -- (2,-1.5);
\end{tikzpicture}
\;\cong\;
F_{[3,4]} \bullet F_{[2,3]},
\end{equation}
and the nonisomorphic graphs
\begin{equation}\label{eq:allbutone2}
  F_{[2,4]} \circ F_{[1,3]} =
\begin{tikzpicture}[scale=.5,baseline=0]
\node at (.6,1.5) {$\scriptstyle 4$};
\node at (.6,0.5) {$\scriptstyle 3$};
\node at (.6,-0.5) {$\scriptstyle{2}$};
\node at (.6,-1.5) {$\scriptstyle{1}$};  
\node at (3.4,1.5) {$\scriptstyle 4$};
\node at (3.4,0.5) {$\scriptstyle 3$};
\node at (3.4,-0.5) {$\scriptstyle{2}$};
\node at (3.4,-1.5) {$\scriptstyle{1}$};  
\draw[-] %(0,1.5) --
(1,1.5) -- (2,-0.5) -- (3,-0.5);
\draw[-] %(0,0.5) --
(1,-1.5) -- (2,-1.5) -- (3,0.5);
\draw[-] %(0,-0.5) --
(1,-0.5) -- (2,1.5) -- (3,1.5);
\draw[-] %(0,-1.5) --
(1,0.5) -- (2,0.5) -- (3,-1.5);
\end{tikzpicture}
\; \not \cong \;
F_{[2,4]} \bullet F_{[1,3]} =
\begin{tikzpicture}[scale=.5,baseline=0]
\node at (.6,1.5) {$\scriptstyle 4$};
\node at (.6,0.5) {$\scriptstyle 3$};
\node at (.6,-0.5) {$\scriptstyle{2}$};
\node at (.6,-1.5) {$\scriptstyle{1}$};  
\node at (3.4,1.5) {$\scriptstyle 4$};
\node at (3.4,0.5) {$\scriptstyle 3$};
\node at (3.4,-0.5) {$\scriptstyle{2}$};
\node at (3.4,-1.5) {$\scriptstyle{1}$};  
\draw[-] %(0,1.5) --
(1,1.5) -- (1.5,0.5) -- (2.5,-0.5) -- (3,-0.5);
\draw[-] %(0,0.5) --
(1,-1.5) -- (2,-1.5) -- (3,0.5);
\draw[-] %(0,-0.5) --
(1,-0.5) -- (2,1.5) -- (3,1.5);
\draw[-] %(0,-1.5) --
(1,0.5) -- (1.5,0.5) -- (2.5,-0.5) -- (3,-1.5);
\end{tikzpicture}\; .
\end{equation}

\begin{defn}\label{d:dsn}
  Define a {\em descending star network} to be a
  condensed concatenation of the form
  \begin{equation}\label{eq:dsn}
    F = F_{[c_1,d_1]} \bullet \cdots \bullet F_{[c_t,d_t]}, \\
\end{equation}
  with
    intervals $[c_1,d_1], \dotsc, [c_t,d_t]$
    distinct, pairwise nonnesting,
    and satisfying $c_1 > \cdots > c_t$ (and $d_1 > \cdots > d_t$).
\end{defn}
\noindent
Let $\D_n$ be the set of all descending star networks of order $n$.
For example,
the $14$ descending star networks in $\D_4$ are
\begin{equation}\label{eq:xfigures2}
\raisebox{-6mm}{
  \includegraphics[height=13mm]{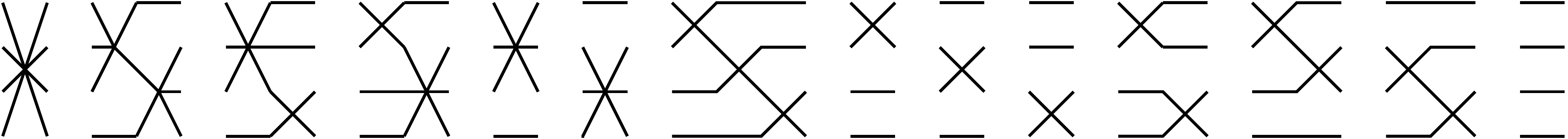}}\ .
\end{equation}
In the notation of
Definition~\ref{d:dsn},
these are
$F_{[1,4]}$,
$F_{[2,4]} \bullet F_{[1,3]}$,
$F_{[2,4]} \bullet F_{[1,2]}$,
$F_{[3,4]} \bullet F_{[1,3]}$,
etc.
By \cite[Thm.\,3.6]{CHSSkanEKL}, \cite[Thm.\,3.5, Lem\,5.3]{SkanNNDCB},
the number of descending star networks of order $n$
is $\tfrac1{n+1} \tbinom{2n}n$.
Explicit bijections $P \mapsto F(P)$ and $F \mapsto P(F)$ between
$\U_n$ and $\D_n$ are as follows. (See also~\cite{CHSSkanEhatKL}.)  
\begin{alg}\label{a:PtoF}
  Given $P \in \U_n$,
  \begin{enumerate}
  \item Label $P$ as in Algorithm~\ref{a:arlabeling}.
  \item Let $[c_1,d_1], \dotsc, [c_t,d_t]$ be the sequence of maximal
    antichains in $P$, ordered so that $c_1 > \cdots > c_t$.
  \item Define $F = F(P) = F_{[c_1,d_1]} \bullet \cdots \bullet F_{[c_t,d_t]}$.
  \end{enumerate}
\end{alg}
\begin{alg}\label{a:FtoP}
  Given $F \in \D_n$,
  \begin{enumerate}
  \item For $i = 1,\dotsc, n$, define $\pi_i$ to be the unique path
    in $F$ from source $i$ to sink $i$.
    \item Give elements of $P = P(F)$ the labels $1, \dotsc, n$.
    \item For all $i,j \in [n]$ with $i < j$ in $\mathbb Z$, declare
      $i <_P j$
      if
      $\pi_i$ does not intersect $\pi_j$ in $F$.
  \end{enumerate}
\end{alg}
Algorithms \ref{a:PtoF} and \ref{a:FtoP} invert one another.
The reader can check that Algorithm~\ref{a:PtoF}
maps the unit interval orders in (\ref{eq:14uios})
to the descending star networks in (\ref{eq:xfigures2}), respectively,
and that Algorithm~\ref{a:FtoP} inverts this.

Given a planar network $F$ of order $n$, define the {\em path matrix}
$B = (b_{i,j})$ of $F$
by
\begin{equation}\label{eq:pathmatrixdef}
  b_{i,j} = \text{number of paths in $F$ from source $i$ to sink $j$}.
\end{equation}
Let $\mathcal P_n$ be the set of $\tfrac 1{n+1}\tbinom{2n}n$ path matrices of
planar networks in $\mathcal D_n$.
For example, the set $\mathcal P_4$ of path matrices of
descending star networks in $\mathcal D_4$ 
(\ref{eq:xfigures2}) is
\begin{equation}\label{eq:14pathmatrices}
  \begin{gathered}
  \begin{bmatrix}
    1 & 1 & 1 & 1\\
    1 & 1 & 1 & 1\\
    1 & 1 & 1 & 1\\
    1 & 1 & 1 & 1
  \end{bmatrix}\ntksp,
  \quad
  \begin{bmatrix}
    1 & 1 & 1 & 0\\
    1 & 1 & 1 & 1\\
    1 & 1 & 1 & 1\\
    1 & 1 & 1 & 1
  \end{bmatrix}\ntksp,
  \quad
  \begin{bmatrix}
    1 & 1 & 0 & 0\\
    1 & 1 & 1 & 1\\
    1 & 1 & 1 & 1\\
    1 & 1 & 1 & 1
  \end{bmatrix}\ntksp,
  \quad
  \begin{bmatrix}
    1 & 1 & 1 & 0\\
    1 & 1 & 1 & 0\\
    1 & 1 & 1 & 1\\
    1 & 1 & 1 & 1
  \end{bmatrix}\ntksp,
  \quad
  \begin{bmatrix}
    1 & 0 & 0 & 0\\
    0 & 1 & 1 & 1\\
    0 & 1 & 1 & 1\\
    0 & 1 & 1 & 1
  \end{bmatrix}\ntksp,
  \\
  \begin{bmatrix}
    1 & 1 & 1 & 0\\
    1 & 1 & 1 & 0\\
    1 & 1 & 1 & 0\\
    0 & 0 & 0 & 1
  \end{bmatrix}\ntksp,
  \quad
  \begin{bmatrix}
    1 & 1 & 0 & 0\\
    1 & 1 & 1 & 0\\
    1 & 1 & 1 & 1\\
    1 & 1 & 1 & 1
  \end{bmatrix}\ntksp,
  \quad
  \begin{bmatrix}
    1 & 0 & 0 & 0\\
    1 & 1 & 0 & 0\\
    1 & 1 & 1 & 1\\
    1 & 1 & 1 & 1
  \end{bmatrix}\ntksp,
  \quad
  \begin{bmatrix}
    1 & 0 & 0 & 0\\
    0 & 1 & 1 & 0\\
    0 & 1 & 1 & 0\\
    0 & 0 & 0 & 1
  \end{bmatrix}\ntksp,
  \quad
  \begin{bmatrix}
    1 & 1 & 0 & 0\\
    1 & 1 & 0 & 0\\
    0 & 0 & 1 & 0\\
    0 & 0 & 0 & 1
  \end{bmatrix}\ntksp,
  \\
  \begin{bmatrix}
    1 & 1 & 0 & 0\\
    1 & 1 & 0 & 0\\
    0 & 0 & 1 & 1\\
    0 & 0 & 1 & 1
  \end{bmatrix}\ntksp,
  \quad
  \begin{bmatrix}
    1 & 0 & 0 & 0\\
    0 & 1 & 1 & 0\\
    0 & 1 & 1 & 1\\
    0 & 1 & 1 & 1
  \end{bmatrix}\ntksp,
  \quad
  \begin{bmatrix}
    1 & 1 & 0 & 0\\
    1 & 1 & 1 & 0\\
    1 & 1 & 1 & 0\\
    0 & 0 & 0 & 1
  \end{bmatrix}\ntksp,
  \quad
  \begin{bmatrix}
    1 & 0 & 0 & 0\\
    0 & 1 & 0 & 0\\
    0 & 0 & 1 & 0\\
    0 & 0 & 0 & 1
  \end{bmatrix}\ntksp.
  \end{gathered}
\end{equation}

By Definition~\ref{d:dsn}, we may describe the entries of
the path matrix of a descending star network
in terms of its defining intervals.
\begin{lem}\label{l:dsn}
  Let
  descending star network
  $F = F_{[c_1,d_1]} \bullet \cdots \bullet F_{[c_t,d_t]}$ as in (\ref{eq:dsn})
  have path matrix $B = (b_{i,j})$.
  Then each entry of $B$ is $0$ or $1$.  In particular,
  for $i \leq j$, we have
  \begin{equation}\label{eq:biji<j}
    b_{i,j} = \begin{cases}
      1 &\text{if some interval $[c_m,d_m]$ contains both $i$ and $j$},\\
      0 &\text{otherwise};
    \end{cases}
  \end{equation}
  for $i > j$ we have
  \begin{equation}\label{eq:biji>j}
    b_{i,j} = \begin{cases}
      1 &\text{if $b_{j,j+1} = b_{j+1,j+2} = \cdots = b_{i-1,i} = 1$},\\
      0 &\text{otherwise}.%\phantom{\bigg(\sum^\sum\bigg)}
    \end{cases}
  \end{equation}
\end{lem}
\begin{proof}
  The structure of a descending star network prevents the existence of
  more than one path from any vertex to any other vertex.
  Then
  Equation (\ref{eq:biji<j}) follows immediately from Definition~\ref{d:dsn}
  and
  (\ref{eq:pathmatrixdef}).

  To see the validity of (\ref{eq:biji>j}), observe that
  by the above definitions we have $b_{i,j} = 1$ with $i > j$ if and only if
  some indices $m \leq m'$ of intervals satisfy
  $i \in [c_m,d_m]$,
  $j \in [c_{m'},d_{m'}]$
  and
  \begin{equation}\label{eq:intersectingintervals}
  [c_k,d_k] \cap [c_{k+1},d_{k+1}] \neq \emptyset
  \text{ for $k = m,\dotsc,m'-1$}.
  \end{equation}
  For each such interval $[c_k,d_k]$ there are paths in $F$ from each
  source $c_k,\dotsc,d_k$ to each sink $c_k,\dotsc,d_k$.  In particular
  we have $b_{c_k,c_k+1} = b_{c_{k+1},d_{k+2}} = \cdots = b_{d_k-1,d_k} = 1$.
  Then since the intervals overlap consecutively,
  with the first interval containing $i$
  and the last interval containing $j$, we have the right-hand side of
  (\ref{eq:biji>j}).
\end{proof}

Some, but not all, of the matrices (\ref{eq:14pathmatrices}) 
appear in
(\ref{eq:14aamatrices}).
The precise connection between $\mathcal A_n$
and $\mathcal P_n$
can be easily described using \emph{block structures} of these matrices
and {\em parabolic subgroups} of $\sn$.
Let $\alpha = (\alpha_1,\dotsc,\alpha_r)$
be an integer composition of $n$,
i.e., a sequence of positive integers satisfying
$\alpha_1 + \cdots + \alpha_r = n$, and define the (ordered) set partition
$(I_1, \dotsc, I_r)$ of $[n]$ 
 into intervals
 of lengths $\alpha_1,\dotsc,\alpha_r$ by
\begin{equation}\label{eq:alphaintervals}
  I_1 = [1,\alpha_1], \quad I_2 = [\alpha_1+1,\alpha_1+\alpha_2],\quad \dotsc, \quad
  I_r = [n-\alpha_r, n].
\end{equation} 
For any $n\times n$ matrix $A$, we can write it in terms of \emph{block submatrices} $A_{I_i, I_j}$ by
\begin{equation}\label{eq:block}
  A = (A_{I_i,I_j}) =
  \begin{bmatrix}
    A_{I_1, I_1} & A_{I_1,I_2} & \cdots & A_{I_1,I_r} \\
    A_{I_2, I_1} & A_{I_2,I_2} & \cdots & A_{I_2,I_r} \\
    \vdots     & \vdots    &        & \vdots   \\
    A_{I_r, I_1} & A_{I_r,I_2} & \cdots & A_{I_r,I_r}
  \end{bmatrix}.
\end{equation}
If $A_{I_i, I_j} =0$ whenever $i<j$, we say $A$ is
{\em $\alpha$-block lower-triangular};
if $A$ and $A^\tr$ are both
$\alpha$-block lower-triangular, we say that $A$ is
{\em $\alpha$-block diagonal}.
It is clear that if $\beta$ is a composition of $n$ with
$\alpha$ refining $\beta$, then
any $\alpha$-block lower-triangular (diagonal) matrix
is also $\beta$-block lower-triangular (diagonal).
In particular, for any square matrix $A$ there is a
finest composition $\alpha = \alpha(A)$ of $n$ for which $A$ is $\alpha$-block
lower-triangular.
This composition is determined by the $0$-entries of $A$ located
just above the diagonal, with the additional property that all entries above
and to the right of these are also $0$.
\begin{obs}\label{eq:superdiagonal0}
  Let $A= (a_{i,j})$ be an $n \times n$ $0$-$1$ matrix 
  satisfying $a_{1,j} \leq \cdots \leq a_{j,j}$ for all $j$,
  and $a_{i,i} \geq \cdots \geq a_{i,n}$ for all $i$.
  Define the set
  $K = \{ k \,|\, a_{k,k+1} = 0 \} = \{k_1,\dotsc,k_{r-1}\}$
  with $k_1 < \cdots < k_{r-1}$, and define $\alpha = \alpha(A)$ as above.
  Then we have
\begin{equation}\label{eq:alphaK}
  \alpha(A) = (\alpha_1,\dotsc,\alpha_r) =
  (k_1, k_2 - k_1, \dotsc, k_{r-1} - k_{r-2}, n - k_{r-1}).
\end{equation}
\end{obs}

 Define $\mfs \alpha$ to be the subgroup of $\sn$ whose elements
 have 
 one-line notations of the form $w^{(i)} \circ \cdots \circ w^{(r)}$,
 where $w^{(j)}$ is a permutation of the letters $I_j$ and $\circ$ denotes
 concatenation of permutations.  Such a subgroup is called {\em parabolic}.
 (See, e.g., \cite{BBCoxeter} or \cite{Hum}.)
 It is known that the determinants of two block lower-triangular matrices
 are equal if their diagonal blocks are the same.
 More generally, we have the following.
\begin{prop}\label{p:blocklt}
  Let $A = (A_{I_i,I_j})$ be an $n \times n$ block lower-triangular matrix with 
  intervals $(I_1,\dotsc,I_r)$ and composition $\alpha = (\alpha_1,\dotsc,\alpha_r)$ as in (\ref{eq:alphaintervals}), 
  and suppose that for some $w \in \sn$ the product $\permmon aw$
  is nonzero.  Then we have $w \in \mfs \alpha$.
\end{prop}
\begin{proof}
  Let $j$ be the smallest index for which the letters in positions $I_j$
  of $w$ do not all belong to $I_j$.
  Choose any index $k \in I_j$ such that
  we have $w_k \not\in I_j$.
  We cannot have $w_k \in I_1 \cup \cdots \cup I_{j-1}$,
  because all of these letters appear in the first
  $\alpha_1 + \cdots + \alpha_{j-1}$ positions of $w$.  Thus $w_k$ belongs
  to $I_{j+1} \cup \cdots \cup I_r$.
  But then the block lower triangularity of $A$ implies that $a_{k,w_k} = 0$,
  contradicting our assumption that $\permmon aw$ is nonzero.
  \end{proof}

\begin{cor}\label{c:AB}
  Let $A = (A_{I_i,I_j})$ be $\alpha$-block lower-triangular,
  and let  $B = (B_{I_i,I_j})$ be the $\alpha$-block diagonal matrix satisfying
  \begin{equation}\label{eq:diagonal}
    B_{I_i,I_j} = \begin{cases}
      A_{I_i,I_j} &\text{if $i = j$},\\
      0 &\text{otherwise}.
    \end{cases}
  \end{equation}
  Then for every permutation $w=w_1w_2\cdots w_n \in \sn$, we have
  \begin{equation}\label{eq:awbw}
    \permmon aw = \permmon bw.
  \end{equation}
\end{cor}

Now we can precisely state the relationship between $\mathcal A_n$
and $\mathcal P_n$.
Define the map
\begin{equation}\label{eq:psi}
  \begin{aligned}
    \psi: \mathcal A_n &\rightarrow \mathcal P_n\\
    A &\mapsto A_{I_1,I_1} \oplus \cdots \oplus A_{I_r,I_r},
  \end{aligned}
\end{equation}
where $(I_1,\dotsc,I_r)$ is the set partition associated 
to the composition $\alpha(A)$ as in (\ref{eq:alphaintervals}).

\begin{prop}\label{p:P A F psiA}
  Given unit interval order $P \in \mathcal U_n$
  with antiadjacency matrix $A$,  
  the descending star network $F = F(P) \in \mathcal D_n$
  has path matrix $\psi(A)$.
\end{prop}
\begin{proof}
  Let $P$, $F$, $A$ be as above, $B = (b_{i,j})$ be the path matrix of $F$,
  and $B' = (b'_{i,j}) = \psi(A)$.

  Consider $i \leq j$.  We have $b'_{i,j} = a_{i,j}$.
  If $b'_{i,j} = a_{i,j} = 1$, then $i$, $j$ form an antichain in $P$.
  Let $[c,d]$ be a maximal antichain containing both.
  By Algorithm~\ref{a:PtoF}, the star network $F$ contains $F_{[c,d]}$
  as a factor.  By Lemma~\ref{l:dsn},
  there is a path from source $i$ to sink $j$
  in $F$ and thus we have $b_{i,j} = 1$.
  If on the other hand $b'_{i,j} = a_{i,j} = 0$,
  then we have $i <_P j$.
  Thus no antichain of $P$ contains both $i$ and $j$, and
  by Algorithm~\ref{a:PtoF}, no interval $[c,d]$ corresponding to a
  factor $F_{[c,d]}$ of $F$ contains both $i$ and $j$.
  By Lemma~\ref{l:dsn},
  there is no path in $F$ from source $i$ to sink $j$, and we have $b_{i,j} = 0$.

  Now consider $i > j$.  Let $\alpha(A) = \alpha(B')$ be the
  finest partition of $n$ such that $A$ is $\alpha$-block lower-triangular
  and $B'$ is $\alpha$-block diagonal.
  Define the set partition $(I_1,\dotsc,I_t)$ associated to $\alpha$
  as in (\ref{eq:alphaintervals}).
  If $b'_{i,j} = 1$ then some block $I_m = [k_{m-1}+1,k_m]$ of this set partition
  contains both $i$ and $j$. Then \begin{equation*}
      I_m = [c_{{p_1}},d_{p_1}] \cup [c_{p_1+1},d_{p_1+1}] \cup \cdots \cup [c_{q_1},d_{q_1}], 
  \end{equation*} for some $p_1,q_1 \in[n]$, where $[c_s,d_s] \cap [c_{s+1},d_{s+1}] \neq \emptyset$, and $s = p_1 ,p_1+1 , \dotsc, q_1-1$. Then, by Lemma ~\ref{l:dsn}, there is a path from source $i$ to sink $j$. Hence, $b_{i,j} = 1$. If on the other hand $b'_{i,j} = 0$, then there are two distinct blocks $I_{x}$ and $I_{y}$ with $y<x$ such that $i\in I_x$ and $j\in I_y$. Then, \begin{gather*}
      I_x  = [c_{p_2},d_{p_2}] \cup \cdots \cup [c_{q_2},d_{q_2}]  \\
      I_y  = [c_{p_3},d_{p_3}] \cup \cdots \cup [c_{q_3},d_{q_3}], 
  \end{gather*} for some $p_2,q_2,p_3,q_3 \in [n]$. Since $I_x$ and $I_y$ are distinct blocks, we have $[c_{q_3},d_{q_3}] \cap [c_{q_2},d_{q_2}] = \emptyset$. By Lemma ~\ref{l:dsn}, there is no path from source $i$ to sink $j$. Hence, $b_{i,j} = 0$.

\end{proof}

\section{Permanents, majorization, multinomial coefficients, and path families}\label{s:permpathfam}

Inequalities of the form 
\begin{equation}\label{eq:LK}
\perm(A_{L,L})\perm(A_{\ol L,\ol L}) \geq \perm(A_{K,K})\perm(A_{\ol K,\ol  K})
\end{equation} 
which hold for all $n\times n$ antiadjacency matrices $A \in \A_n$ (Section~\ref{s:fourcatalan}) define a partial order on set partitions $L|\ol L$ of $[1,n]$ into at most two blocks:  we define $K|\ol K \leq L|\ol L$ if \eqref{eq:LK} holds for matrices $A \in \A_n$.
For example, when $n=4$, one may evaluate the $2^{4-1} = 8$ products of complementary permanents at the $C_4 = 14$ matrices in $\A_4$ to obtain the Hasse diagram
\begin{center}
    \begin{tikzpicture}
  \node (a) at (0,-0.75) {$13|24$};
  \node (b) at (2,-0.1) {$14|23$};
  \node (c) at (-1,1) {$134|2$};
  \node (d) at (1,1) {$124|3$};
  \node (e) at (-2,2) {$234|1$};
  \node (f) at (0,2) {$12|34$};
  \node (g) at (2,2) {$123|4$};
  \node (h) at (0,3) {$1234$};
  
  \draw (a) -- (c) -- (e) -- (h)
  (a) -- (d) -- (g) -- (h)
  (b) -- (e)
  (a) -- (f) -- (h)
  (b) -- (g);

\end{tikzpicture}.\end{center} 
Contrary to our intuition \eqref{eq:evenoddintro}, we see that the odd-even partition $13|24$ is not the unique minimal element of this poset.
Similarly, when $n=5$, one may evaluate the $2^{5-1} = 16$ products of complementary permanents at the $C_5 = 42$ matrices in $\A_5$ to obtain the Hasse diagram
\begin{center}
    \begin{tikzpicture}[scale=0.7]
  \node (a) at (0,0) {$135|24$};
  \node (c) at (-2.05,1.75) {$124|35$};
  \node (d) at (2.05,1.75) {$13|245$};
  \node (b) at (-5,3) {$14|235$};
  \node (e) at (5,3) {$134|25$};
  \node (f) at (-2.25,4) {$125|34$};
  \node (h) at (2.25,4) {$145|23$};
  \node (g) at (0,4) {$15|234$};
  \node (i) at (-5,5) {$1235|4$};
  \node (j) at (5,5) {$1345|2$};
  \node (k) at (-4,7.5) {$1234|5$};
  \node (l) at (-2,7.5) {$123|45$};
  \node (m) at (0,7.5) {$1245|3$};
  \node (n) at (2,7.5) {$12|345$};
  \node (o) at (4,7.5) {$1|2345$};
  \node (p) at (0,9) {$12345$};

  \draw (a) -- (c) (a) -- (d) (c) -- (i) (d) -- (j) (b) -- (i) -- (k) -- (p) (e) -- (j) -- (o) -- (p) (p) -- (k) 
  (p) -- (l) (p) -- (m) (p) -- (n) (p) -- (o) (f) -- (k) (f) -- (n) (g) -- (k) (g) -- (o) (h) -- (l) (h) -- (o);
  \draw (k) .. controls (3.5,5.5) .. (e);
  \draw (o) .. controls (-3.5,5.5) .. (b);
  \draw (l) .. controls (-4,5) .. (c);
  \draw (n) .. controls (4,5) .. (d);
  \draw (m) -- (c) (m) -- (d); 
\end{tikzpicture}. \end{center}
Again, the odd-even partition $135|24$ is not the unique minimal element of the poset.
In both of the above posets, the odd-even set partitions are lower bounds for
sets of set partitions consisting of complementary subintervals.
For $n=4$ these are $1234$, $123|4$, $12|34$, $1|234$.
For $n=5$ they are $12345$, $1234|5$, $123|45$, $12|345$, $1|2345$.

The above facts generalize to all $n$, and we will state the generalization in 
Theorem~\ref{t:main}. 
The proof of this theorem relies upon
planar networks and the following
facts concerning majorization and products of factorials.

It is well known that the greatest coefficients in the expansion of 
$(x_1 + \cdots + x_k)^n$ are those of the monomials 
$x_1^{a_1} \cdots x_k^{a_k}$ having
exponents as close to one another as possible.
\begin{lem}\label{l:multiincrement}
    Let $(a_1,\dotsc, a_k)$ be a weak composition of $n$ and suppose that for some indices $h$, $\ell$ we have $a_\ell - a_h \geq 2$.
    Let $(\hat a_1, \dotsc, \hat a_k)$ be the composition obtained from
    $(a_1,\dotsc,a_k)$ by defining
    \begin{equation*}
        \hat a_i = \begin{cases}
            a_i + 1 &\text{if $i = h$},\\
            a_i -1 &\text{if $i = \ell$},\\
            a_i &\text{otherwise}.
        \end{cases}
    \end{equation*}
    Then we have
    \begin{equation*}
        {n \choose a_1,\dotsc,a_k} \leq {n \choose \hat a_1,\dotsc, \hat a_k }.
    \end{equation*}
\end{lem}
\begin{proof}
    Since $\tfrac{a_\ell}{a_h+1} \geq 1$,
    the left-hand side is at most
    \begin{equation*}
        \frac{a_\ell}{a_h+1} \cdot \frac{n!}{a_1!a_2!\cdots a_k!} = \frac{n!}{\hat a_1!\cdots \hat a_k!} = {n \choose \hat a_1, \dotsc, \hat a_k}. 
    \end{equation*}
\end{proof}

\begin{cor}\label{c:max central multi}
    Let $(a_1,\dotsc,a_{k-1})$, $(b_1,\dotsc,b_{k})$ be weak compositions of $n$
    whose components satisfy $|a_i-a_j| \leq 1$ and $|b_i-b_j| \leq 1$ for all $i, j$. Then we have
    \begin{equation}\label{eq:abcompositions}
        a_1! \cdots a_{k-1}! \geq b_1!\cdots b_{k}!.
    \end{equation}
\end{cor}
\begin{proof}
If $\max(a_1,\dotsc,a_{k-1}) \geq 2$ then
we may apply Lemma~\ref{l:multiincrement} repeatedly to the weak composition $(a_1,\dotsc, a_{k-1},0)$ to obtain
the weak composition $(b_1,\dotsc,b_k)$.
On the other hand, if $\max(a_1,\dotsc,a_{k-1}) \leq 1$, then we also have $\max(b_1,\dotsc,b_k) \leq 1$ and \eqref{eq:abcompositions} holds with equality.
\end{proof}

Given a composition $(a_1,\dotsc,a_k)$ of $n$ satisfying $|a_i - a_j| \leq 1$ for all $i$, $j$, let us say that it has {\em nearly equal parts}. For example, $(3,2,2)$ is a composition of $7$ into $3$ nearly equal parts. We can see that the composition is unique up to reordering.

Let $x$ and $y$ be partitions of $n$, we say $x$ is (weakly) majorized by $y$, written $x \preceq y$, if \begin{equation*}\sum_{i=1}^{k} x_i \leq \sum_{i=1}^k y_i \quad \text{for all $k = 1,\dotsc,n$.}\end{equation*}
It is well known that the partitions of $n$ under the majorization order form a lattice. We also define $\lambda \precdot \mu$ if $\lambda \prec \mu$ and there is no $\tau$ such that $\lambda \prec \tau \prec \mu$. 
The following lemma is due to Hardy, Littlewood and Pólya~\cite{MOA2011}, and is a famous result in majorization theory. 
\begin{lem}\label{l:majorize}
Let $\lambda, \mu$ be distinct partitions of $n$. If $\lambda \precdot \mu$ then there are indices $i < j$ such that we have
$\lambda_i = \mu_i - 1$ and $\lambda_j = \mu_j + 1$, and all other parts of $\lambda$ and $\mu$ are equal.  
\end{lem}

The (weak) majorization order gives rise to a factorial inequality for its parts. More generally, we have the following.

\begin{prop}\label{p:majorfracineq}
    Let partitions $\lambda = (\lambda_1,\dotsc,\lambda_s)$, $\mu = (\mu_1,\dotsc, \mu_r)$ of $n$ satisfy $\lambda \preceq \mu$.
    Then we have $\lambda_1! \cdots \lambda_s! \leq \mu_1! \cdots \mu_r!$.
\end{prop}
\begin{proof}
    By Lemma~\ref{l:majorize}, the majorization order contains a saturated chain
    \begin{equation*}
        \lambda = \lambda^{(0)} \precdot \lambda^{(1)} \precdot \cdots \precdot \lambda^{(k)} = \mu
        \end{equation*}
        in which consecutive partitions differ in exactly two parts. 
    By $k$ applications of Lemma~\ref{l:multiincrement}, we therefore have 
    \begin{equation*}
        {n\choose \lambda_1,\dotsc,\lambda_s} \geq {n\choose \mu_1,\dotsc,\mu_r},
    \end{equation*} 
    which is equivalent to the claimed inequality.
\end{proof}

We now have an immediate well-known consequence. 

\begin{prop}\label{p:factorial inequality}
    For $a, b \in \mathbb N$ with
    $a+b=m$, we have 
    \begin{equation}\label{eq:factorial inequality}
        a! \cdot b! \geq \floor*{\frac{m}{2}}! \cdot \ceil*{\frac{m}{2}}!
    \end{equation}
\end{prop}
\begin{proof}
    Without loss of generality, we assume $a\leq b$. If $a = \floor*{\frac{m}{2}}$, then there is nothing to prove. If $a \neq \floor*{\frac{m}{2}}$, then the partition $(b,a)$ majorizes $(\ceil{\tfrac{m}{2}},\floor{\tfrac{m}{2}})$. Then~\eqref{eq:factorial inequality} follows from Proposition~\ref{p:majorfracineq}.

\end{proof}

A function $f: \mathbb{R} \rightarrow \mathbb{R}$ is convex if for all $0\leq t \leq1$ and distinct $x_1,x_2 \in \mathbb{R}$, 
\begin{equation}f(tx_1+(1-t)x_2) \leq tf(x_1) + (1-t)f(x_2).
\end{equation}
In calculus, we know that the ``slope" of a convex function is increasing as $x$ increases. In the discrete case, if $f$ is convex, then the difference $d(x) \defeq f(x+1) - f(x)$ increases as $x$ increases.  

In addition to the well-known characterizations of the majorization order, we have the following characterization 
due to Karamata~\cite[p.~3--13]{MOA2011}
which involves
convex functions.

\begin{lem}(Karamata's Inequality)\label{l:Karamata}
    Given partitions $x=(x_1,\dotsc,x_n)$,
    $y=(y_1,\dotsc,y_n)$,
    we have $y \preceq x$ if and only if for every convex function $f$, we have 
    \begin{equation}\label{eq:Karamata}
        \sum_{i=1}^n f(x_i) \geq \sum_{i=1}^n f(y_i). 
    \end{equation}
\end{lem}

One can see that the factorial function (equivalently, the Gamma function) is convex. Therefore, Proposition~\ref{p:majorfracineq} is an immediate consequence of Karamata’s inequality.

Consider the Karamata sums~\eqref{eq:Karamata}, applied to the partition $\rho(n,p)$ of $n$ into $p$ nearly equal parts. For any function $f: \mathbb R \rightarrow \mathbb R$, define
\begin{equation}\label{eq:Gdef}
        G_f(n,p) \defeq \sum_{i=1}^{p} f(\rho(n,p)_i).
    \end{equation} 
    It is straightforward to show that the difference of two such sums for fixed $p$ and consecutive values of $n$ can be expressed as follows: 

\begin{lem}\label{l:G function 1}
    Fix $p \leq n$ and any function $f: \mathbb R \rightarrow \mathbb R$. 
    Then, we have 
    \begin{equation}\label{eq:G function 1}
        G_f(n+1,p) - G_f(n,p) = f(\floor{\tfrac{n}{p}}+1) - f(\floor{\tfrac{n}{p}}).
    \end{equation}
\end{lem}

\begin{proof}
    Write $n=pq +r$, then the partition $\rho(n,p)$ has $p-r$ parts of size $\floor{\tfrac{n}{p}} = q$ and $r$ parts of size $\floor{\tfrac{n}{p}}+1 = q+1$. Therefore, \begin{equation*}
        G_f(n,p) = (p-r)f(q) + r f(q+1). 
    \end{equation*} Note that to move a partition of $n$ into $p$ nearly equal parts to a partition of $n+1$ into $p$ nearly equal parts, we just pick one part that was of size $q$ and increase it to $q+1$. Therefore, \begin{equation*}
        G_f(n+1,p) - G_f(n,p) = f(q+1) - f(q). 
        \end{equation*}
\end{proof}

For example, the sums $G_f(8,3)$, $G_f(7,3)$ satisfy 
\begin{equation}
    \begin{gathered}
        \rho(8,3) = (3,3,2), \quad \rho(7,3) = (3,2,2), \\
        G_f(8,3) = f(3) + f(3) + f(2), \quad G_f(7,3) = f(3) + f(2) + f(2) \\
        \floor{\tfrac{7}{3}} = 2, \quad \floor{\tfrac{7}{3}} +1 = 3, \\
        G_f(8,3) - G_f(7,3) = f(3) - f(2). 
    \end{gathered}
\end{equation}

If the function $f$ is convex, then we can relate differences~\eqref{eq:G function 1} by the following inequality. 

\begin{lem}\label{l:G function 2}
    Fix $p \leq n$ and a convex function $f: \mathbb R \rightarrow \mathbb R$. Then, we have \begin{equation}\label{eq:Delta n increasing}
        G_f(n+1,p) - G_f(n,p) \geq G_f(n+1,p+1)- G_f(n,p+1).
    \end{equation}
\end{lem}

\begin{proof}
    By Lemma~\ref{l:G function 1}, the above inequality is equivalent to the nonnegativity of 
    \begin{equation}\label{eq:G into f}
        \left[f(\floor{\tfrac{n}{p}} +1) - f(\floor{\tfrac{n}{p}})\right] - \left[f(\floor{\tfrac{n}{p+1}} +1) - f(\floor{\tfrac{n}{p+1}})\right]. 
    \end{equation}
    Next, consider the difference $D(\tau) \defeq f(\tau +1) - f(\tau)$. Since $f$ is convex, $D(\tau)$ is nondecreasing (as $\tau$ increases). Since $\floor{\tfrac{n}{p}} \geq \floor{\tfrac{n}{p+1}}$, we have \begin{equation*}
        D(\floor{\tfrac{n}{p}}) - D(\floor{\tfrac{n}{p+1}}) \geq 0, 
    \end{equation*}which implies~\eqref{eq:G into f} is nonnegative. 
\end{proof}

For example, the differences 
\begin{equation*}
    \begin{gathered}
        G_f(8,3) - G_f(7,3) = f(3) - f(2) \\ 
        G_f(8,4) - G_f(7,4) = f(\floor{\tfrac{7}{4}}+1) -f(\floor{\tfrac{7}{4}}) = f(2) - f(1) \\ 
    \end{gathered}
\end{equation*}satisfy 
\begin{equation*}
    G_f(8,3) - G_f(7,3) \geq G_f(8,4) - G_f(7,4)
\end{equation*} since the convexity of $f$ guarantees that 
\begin{equation*}
    f(3) - f(2) \geq f(2) - f(1). 
\end{equation*}

Define $\sort(\lambda,\mu)$ to be the partition obtained from listing the parts of $\lambda$ and $\mu$ (as a multiset) in weakly decreasing order. 
Somewhat surprisingly, Lemmas~\ref{l:Karamata} - \ref{l:G function 2} imply that sorting certain
pairs of partitions into nearly equal parts
produces two new partitions which are related by majorization.

\begin{prop}\label{p:nearly equal parts maj order}
    Fix $m \geq k > s$ and let $\kappa$ be the partition of $m$ into $s$ nearly equal parts and $\mu$ be the partition of $k$ into $s$ nearly equal parts. Let $\lambda$ be the partition of $m$ into $s+1$ nearly equal parts and $\nu$ be the partition of $k$ into $s+1$ nearly equal parts. 
    Define the partitions 
    \begin{equation}\label{eq:xydef}
    x\defeq \sort( 
    \kappa, \nu), \qquad y \defeq \sort(\lambda,\mu)
    \end{equation}
    of $m+k$ into $2s+1$ parts.  Then we have $y \preceq x$.
\end{prop}

\begin{proof}
    Let $\rho(n, p)$ denote the partition of $n$ into $p$ nearly equal parts as before. Then, 
\begin{gather*}
    \kappa = \rho(m, s),\ \mu = \rho(k, s) \\
    \lambda = \rho(m, s+1),\ \nu = \rho(k, s+1).
\end{gather*}

We claim that for every convex function $f$, we have
\begin{equation}\label{eq:Karamata2}
    \sum_{i=1}^{2s+1} f(y_i) \leq \sum_{i=1}^{2s+1} f(x_i).
\end{equation}

To see this, define $G_f(n,p)$ as in Lemma~\ref{l:G function 1}. 
By Lemma~\ref{l:G function 2}, we have 
\begin{equation*}
    G_f(n+1,p) - G_f(n+1,p+1) \geq G_f(n,p)- G_f(n,p+1).
\end{equation*} Since $m\geq k$, we have 
\begin{equation*}
    G_f(m,s) - G_f(m,s+1) \geq G_f(k,s) - G_f(k,s+1),
\end{equation*}
i.e.,
\begin{equation*}
    G_f(m,s) + G_f(k,s+1)  \geq G_f(k,s) + G_f(m,s+1).
\end{equation*}
By the definitions of $G_f$ \eqref{eq:Gdef} and $x$, $y$ \eqref{eq:xydef},
we have \begin{gather*}
    G_f(m,s) + G_f(k,s+1) = \sum_{i=1}^{2s+1} f(x_i), \quad
    G_f(k,s) + G_f(m,s+1) = \sum_{i=1}^{2s+1} f(y_i). 
\end{gather*} Hence, \eqref{eq:Karamata2} is true. By Lemma~\ref{l:Karamata}, we have proved this proposition.

\end{proof}

By Proposition~\ref{p:majorfracineq}, the majorization of one partition by another implies a factorial inequality for the parts of those two partitions.  Applying this to the partitions $x$, $y$ in \eqref{eq:xydef}, we obtain the following.  

\begin{cor}\label{c:multi factorial inequality}
    Fix $s,m,k \in \mathbb{N}$ with $s < k\leq m$. Let $(k_1,\dotsc,k_{s})$, $(k'_1,\dotsc, k'_{s+1})$ be the partitions of $k$ into $s$ nearly equal parts and $s+1$ nearly equal parts respectively. Let $(m_1,\dotsc,m_{s})$, $(m'_1,\dotsc, m'_{s+1}) $ the partitions of $m$ into $s$ nearly equal parts and $s+1$ nearly equal parts respectively. Then we have \begin{equation}\label{eq:multi factorial inequality}
        \frac{m_1! \cdots m_{s}!}{k_1! \cdots k_{s}!} \geq \frac{m'_1! \cdots m'_{s+1}!}{k'_1! \cdots k'_{s+1}!}.
    \end{equation}
\end{cor}

\begin{proof}
    The result follows immediately from Proposition~\ref{p:majorfracineq} and Proposition~\ref{p:nearly equal parts maj order}.  
\end{proof}

For example, let $m=11$, $k=7$, and $s=3$. Then, we have 
\begin{align*}
    m_1&=4 & m_2 &=4 & m_3&=3 &k'_1&=2 & k'_1 &= 2 & k'_3&=2 & k'_4&=1 \\
    m'_1&=3 & m'_2 &=3 & m'_3&=3 & m'_4&=2 &k_1&=3 & k_2&=2 & k_3 &= 2 .\\
    \end{align*}
One can verify that \begin{gather*}
    \sort(m_1,m_2,m_3,k'_1,k'_2,k'_3,k'_4) = (4,4,3,2,2,2,1), \\
    \sort(m'_1,m'_2,m'_3,m'_4,k_1,k_2,k_3) = (3,3,3,3,2,2,2), \text{ and} \\ 
    (3,3,3,3,2,2,2)\preceq(4,4,3,2,2,2,1), \\
    144 = \frac{3!4!4!}{2!2!3!} \geq \frac{2!3!3!3!}{1!2!2!2!} = 54. 
\end{gather*}

The next proposition is a special case of Corollary~\ref{c:multi factorial inequality}. 

\begin{prop}\label{p:max factorial}
    For $m,k\in \mathbb N$ with $m \geq k$, we have
    \begin{equation*}
    \frac{\floor{\frac{m}{2}}!\ceil{\frac{m}{2}}!}{\floor{\frac{k}{2}}!\ceil{\frac{k}{2}}!} \leq \frac{m!}{k!}.
    \end{equation*}
\end{prop}

\begin{proof}
    This is a special case of Corollary~\ref{c:multi factorial inequality} as $s=1$. 
\end{proof}

We next prove a stronger version of the previous proposition. One can see that this is no longer a special case of Corollary~\ref{c:multi factorial inequality}. 

\begin{prop}\label{p:max factorial2}
    For $m,k\in \mathbb N$ with $m \geq k$, we have 
    \begin{equation}\label{eq:mk2}
        \frac{\floor{\frac{m}{2}}!\ceil{\frac{m}{2}}!}{\floor{\frac{k}{2}}!\ceil{\frac{k}{2}}!} \leq \frac{(m-1)!}{(k-1)!}.
    \end{equation}
\end{prop}

\begin{proof}
    If $m=k$, \eqref{eq:mk2} is trivially true. Suppose therefore that $k<m$. 
    Multiplying both sides of \eqref{eq:mk2} by $(m+k-1)!$, we obtain the equivalent inequality
    \begin{equation}\label{eq:mk2equi}
        {{m+k-1}\choose {m-1, \ceil{\frac{k}{2}}, \floor{\frac{k}{2}}} } \leq { {m+k-1}\choose {k-1, \ceil{\frac{m}{2}}, \floor{\frac{m}{2}}}}. 
    \end{equation}
    By Lemma~\ref{l:multiincrement}, we have \begin{equation*}\begin{split}
    {{m+k-1}\choose {m-1, \ceil{\frac{k}{2}}, \floor{\frac{k}{2}}} } &\leq { {m+k-1}\choose {m-2, \ceil{\frac{k+1}{2}}, \floor{\frac{k+1}{2}}}} \leq \cdots \\ 
   \cdots &\leq {{m+k-1}\choose {m-1 - (m-k), \ceil{\frac{k+(m-k)}{2}}, \floor{\frac{k+(m-k)}{2}}} },
\end{split}\end{equation*} since for $i=0,1,\dotsc, m-k$, we have $\floor{\frac{k+i}{2}}, \ceil{\frac{k+i}{2}} \leq k+i \leq m+i-1$ and \begin{equation*}
    \bigg(\floor*{\frac{k+1}{2}}, \ceil*{\frac{k+i}{2}}\bigg) = \begin{cases}
        (\floor{\frac{k+i}{2}}, \ceil{\frac{k+i-1}{2}}+1) \ \text{when $k + i$ is even,} \\ 
        (\ceil{\frac{k+i-1}{2}}, \floor{\frac{k+i-1}{2}}+1) \ \text{when $k + i$ is odd.}
    \end{cases} 
\end{equation*}
\end{proof}

For example, let $m=11$ and $k=7$. Then we have 
\begin{gather*}
    m-1 = 10,\quad   k-1 = 6, \\ 
    \floor{\tfrac{m}{2}} = 5, \quad \ceil{\tfrac{m}{2}} = 6, \quad
    \floor{\tfrac{k}{2}} = 3, \quad \ceil{\tfrac{k}{2}} = 4. 
\end{gather*}
One can verify that \begin{equation*}
    \frac{(m-1)!}{(k-1)!} = \frac{10!}{6!} = 5040 \geq \frac{\floor{\tfrac{m}{2}}!\ceil{\tfrac{m}{2}}!}{\floor{\tfrac{k}{2}}!\ceil{\tfrac{k}{2}}!} = \frac{5!6!}{3!4!} = 600.
\end{equation*}

It would be interesting to describe the analogous poset defined by inequalities~\eqref{eq:LK} holding for \emph{all} totally nonnegative matrices $A$ (not just those in $\mathcal{A}_n$). This problem appears to be more difficult than its determinantal analog solved in~\cite{FGJMult}. (See also~\cite{SkanIneq}.) We therefore restrict our attention to the set $\mathcal{A}_n$. Indeed other total nonnegativity results are more easily proved for $A\in\mathcal{A}_n$, then for $A$ totally nonnegative. For instance, compare the proofs~\cite[Thm. 2]{GashInc} and~\cite[Cor.3.2]{StemImm} of the inequalities $\imm{\chi}(A)\geq 0$ which hold for $A\in\mathcal{A}_n$, and $A$ totally nonnegative, respectively.
To study this poset induced by~\eqref{eq:LK}, we look more closely at the planar networks defined in Section~\ref{s:fourcatalan} and their connection to permanents. 

Let $\pi=(\pi_1,\dotsc,\pi_n)$ be a sequence of source-to-sink paths in a
planar network $F \in \F_n$.
We call $\pi$ a \textit{path family} if there exists a permutation
$w=w_1\cdots w_n\in \sn$ such that $\pi_i$
is a path from source $i$ to sink $w_i$.
In this case, we say that $\pi$ has \emph{type} $w$.
Further, we say that the path family \emph{covers $F$}
if it contains every edge exactly once. 

If $A \in \mat n$ is the path matrix of $F \in \F_n$, then some
immanants (\ref{eq:immdef}) of $A$
can be interpreted in terms of path families in $R$.  For instance
the simple immanant $\permmon aw$ equals the number of
path families of type $w$ in $F$.

Let $F$ be a planar network of order $n$ and let $I, J$ be subsets of $[n]$ with $|I|=|J|$. Let $\Pi_{I,J}(F)$ denote the set of path families from sources indexed by $I$ to sinks indexed by $J$. For economy, let $\Pi(F) = \Pi_{[n],[n]}(F)$.

The following interpretation of
the determinant is
due to Karlin and McGregor~\cite{KMG}
and usually attributed to Lindstrom~\cite{LinVrep}.

\begin{lem}\label{l:llem}
  We have that $\det(A)$ equals the number of path families $\pi$ in $F$
  in which $\pi_1,\dotsc,\pi_n$ are pairwise nonintersecting.
\end{lem}

\begin{cor}\label{c:llem}
  For $I, J \subseteq [n]$ with $|I| = |J| = k$, we have
  \begin{equation*}
    \det(A_{I,J}) = \# \{ \pi = (\pi_1, \dotsc, \pi_k) \in \Pi_{I,J}(F)
    \,|\, \pi_1,\dotsc,\pi_k \text{ pairwise nonintersecting } \}.
  \end{equation*}
\end{cor}  

We have the following combinatorial interpretation of the permanent using path families.

\begin{obs}\label{o:perlem}
  Let $A$ be the path matrix of planar network $F \in \F_n$.
  Then $\perm(A)$ is the number of path families in $F$.
  \end{obs}
\begin{cor}\label{c:perlem}
    For $I, J \subseteq [n]$ with $|I| = |J| = k$, we have
    $\perm(A_{I,J}) = |\Pi_{I,J}(F)|$.
\end{cor}

For example, the planar network $F$
\begin{center}
    \begin{tikzpicture}[scale=.7,baseline=0]
\node at (.6,1.5) {$\scriptstyle 4$};
\node at (.6,0.5) {$\scriptstyle 3$};
\node at (.6,-0.5) {$\scriptstyle{2}$};
\node at (.6,-1.5) {$\scriptstyle{1}$};  
\node at (3.4,1.5) {$\scriptstyle 4$};
\node at (3.4,0.5) {$\scriptstyle 3$};
\node at (3.4,-0.5) {$\scriptstyle{2}$};
\node at (3.4,-1.5) {$\scriptstyle{1}$};  
\draw[-] %(0,1.5) --
(1,1.5) -- (1.5,0.5) -- (2.5,-0.5) -- (3,-0.5);
\draw[-] %(0,0.5) --
(1,-1.5) -- (2,-1.5) -- (3,0.5);
\draw[-] %(0,-0.5) --
(1,-0.5) -- (2,1.5) -- (3,1.5);
\draw[-] %(0,-1.5) --
(1,0.5) -- (1.5,0.5) -- (2.5,-0.5) -- (3,-1.5);
\end{tikzpicture}
\end{center}
has path matrix
\begin{equation*}
A = \begin{bmatrix}
1 & 1 & 1 & 0 \\
1 & 1 & 1 & 1 \\ 
1 & 1 & 1 & 1 \\
1 & 1 & 1 & 1
\end{bmatrix}.
\end{equation*}
We see that $\perm(A_{\{1,3\},\{2,3\}}) = 2$. (Note that we define $\perm(A_{\emptyset,\emptyset}) = 1$.) The $2$ path families in $\Pi_{\{1,3\},\{2,3\}}(F)$ are 
\begin{center}
\begin{tikzpicture}[scale=.65,baseline=0]
\node at (.6,1.5) {$\scriptstyle 4$};
\node at (.6,0.5) {$\scriptstyle 3$};
\node at (.6,-0.5) {$\scriptstyle{2}$};
\node at (.6,-1.5) {$\scriptstyle{1}$};  
\node at (3.4,1.5) {$\scriptstyle 4$};
\node at (3.4,0.5) {$\scriptstyle 3$};
\node at (3.4,-0.5) {$\scriptstyle{2}$};
\node at (3.4,-1.5) {$\scriptstyle{1}$};  
\draw[-] %(0,1.5) --
(1,1.5) -- (1.5,0.5) -- (2.5,-0.5);
\draw[-, very thick, color=blue] (2.5,-0.5) -- (3,-0.5);
\draw[-, very thick, color=red] %(0,0.5) --
(1,-1.5) -- (2,-1.5) -- (3,0.5);
\draw[-] %(0,-0.5) --
(1,-0.5) -- (2,1.5) -- (3,1.5);
\draw[-, very thick, color=blue] %(0,-1.5) --
(1,0.5) -- (1.5,0.5) -- (2.5,-0.5);
\draw[-] (2.5,-0.5) -- (3,-1.5);
\end{tikzpicture} and
\begin{tikzpicture}[scale=.65,baseline=0]
\node at (.6,1.5) {$\scriptstyle 4$};
\node at (.6,0.5) {$\scriptstyle 3$};
\node at (.6,-0.5) {$\scriptstyle{2}$};
\node at (.6,-1.5) {$\scriptstyle{1}$};  
\node at (3.4,1.5) {$\scriptstyle 4$};
\node at (3.4,0.5) {$\scriptstyle 3$};
\node at (3.4,-0.5) {$\scriptstyle{2}$};
\node at (3.4,-1.5) {$\scriptstyle{1}$};  
\draw[-] %(0,1.5) --
(1,1.5) -- (1.5,0.5) -- (2.5,-0.5);
\draw[-, very thick, color=red] (2.5,-0.5) -- (3,-0.5);
\draw[-, very thick, color=red] %(0,0.5) --
(1,-1.5) -- (2,-1.5) --(2.5,-0.5);
\draw[-, very thick, color=blue] (2.5,-0.5) -- (3,0.5);
\draw[-] %(0,-0.5) --
(1,-0.5) -- (2,1.5) -- (3,1.5);
\draw[-, very thick, color=blue] %(0,-1.5) --
(1,0.5) -- (1.5,0.5) -- (2.5,-0.5);
\draw[-] (2.5,-0.5) -- (3,-1.5);
\end{tikzpicture}.
\end{center}

There is a simple formula which counts path families covering a descending star network. 

\begin{lem}\label{eq:counting number of path families}
Given descending star network $F = F_{[c_1,d_1]} \bullet \cdots \bullet F_{[c_r,d_r]}$, define numbers $m_1, \dotsc, m_r, \, k_1,\dotsc, k_{r-1}$ by

\begin{equation}\label{equ: size of stars}
    \begin{gathered}
        m_i=d_i-c_i+1, \qquad \qquad \;\\
        k_i=\max\{d_{i+1}-c_i+1,0\}.
    \end{gathered}
\end{equation}
Then we have 
\begin{equation}\label{equ: size}
    |\Pi_{[n],[n]}(F)| = \frac{m_1 !\cdots m_{r-1} !}{  k_{1} !\cdots k_{r-1}!} \cdot m_r!.
\end{equation}
\end{lem}

\begin{proof}
If $F$ is a simple star network $F=F_{[c_1, d_1]}$. Then $|\Pi_{[n],[n]}(F)|$ is clearly $(d_1-c_1+1)!$. Now suppose \eqref{equ: size} holds for descending star networks with $1, \ldots, r-1$ factors and consider $F$ with $r$ factors. Let $F'=F_{[c_1,d_1]} \bullet \cdots \bullet F_{[c_{r-1},d_{r-1}]}$, $F=F'\bullet F_{[c_r, d_r]}$ and consider $\Pi(F'\circ F_{[c_r, d_r]})$. Clearly $|\Pi(F'\circ F_{[c_r, d_r]})|=|\Pi(F')|\cdot |\Pi(F_{[c_r,d_r]})|=\frac{m_1!\cdots m_{r-1}!}{k_1 !\cdots k_{r-2} !}\cdot m_r!$. Now recall that $F'\bullet F_{[c_r, d_r]}$ is obtained from $F'\circ F_{[c_r, d_r]}$ by removing all but one of the $k_{r-1}$ edges from central vertices $z_{r-1}$ to $z_r$. Call this edge $e$. This gives a surjective map $\pi\mapsto \widehat{\pi}$ from $\Pi(F\circ F_{[c_r, d_r]}$ to $\Pi(F\bullet F_{[c_r, d_r]})$ by $\widehat{\pi}=\pi_i \text{ with its $z_{r-1}$ to $z_r$ edge replaced by $e$}$. It is clear that the preimage of any path family $\widehat{\pi}\in\Pi(F'\bullet F_{[c_r, d_r]})$ has cardinality $k_{r-1}!$. Thus we have the claimed formula for $|\Pi(F'\bullet F_{[c_r, d_r]})|$.   % This is straightforward using induction. 
\end{proof}

For example, when $F=F_{[8,10]}\bullet F_{[3,8]}\bullet F_{[1,5]}$, paths in $\Pi_{[10],[10]}(F)$ are the paths covering the edges in 

\begin{center}
    \begin{tikzpicture}
    \draw (0.5,4) -- (3.5,2.25) -- (6,1); 
    \draw (0,4.5) -- (0.5,4); 
    \draw (0,4) -- (0.5,4); 
    \draw (0,3.5) -- (0.5,4); 
    \draw (0,3) -- (1,3) -- (3.5,2.25); 
    \draw (0,2.5) -- (1,2.5) -- (3.5,2.25); 
    \draw (0.5,4) -- (1,4.5) -- (7,4.5); 
    \draw (0.5,4) -- (1,4) -- (7,4);

    \draw (0,1.5) -- (2.5,1.5) -- (3.5,2.25);
    \draw (0,1) -- (2.5,1) -- (3.5,2.25);
    \draw (0,2) -- (2.5,2) -- (3.5,2.25);
    \draw (3.5,2.25) -- (4.5,2.5) -- (7,2.5);
    \draw (3.5,2.25) -- (4.5,3) -- (7,3);
    \draw (3.5,2.25) -- (4.5,3.5) -- (7,3.5);

    \draw (0,0.5) -- (5,0.5) -- (6,1); 
    \draw (0,0) -- (5,0) -- (6,1);
    \draw (6,1) -- (7,0);
    \draw (6,1) -- (7,0.5);
    \draw (6,1) -- (7,1);
    \draw (6,1) -- (7,1.5);
    \draw (6,1) -- (7,2);

    \node[scale=0.75] at (-0.5,0) {1};
    \node[scale=0.75] at (-0.5,1) {3};
    \node[scale=0.75] at (-0.5,2) {5};
    \node[scale=0.75] at (-0.5,3) {7};
    \node[scale=0.75] at (-0.5,4) {9};
    \node[scale=0.75] at (-0.5,0.5) {2};
    \node[scale=0.75] at (-0.5,1.5) {4};
    \node[scale=0.75] at (-0.5,2.5) {6};
    \node[scale=0.75] at (-0.5,3.5) {8};
    \node[scale=0.75] at (-0.5,4.5) {10};
    
    \node[scale=0.75] at (7.5,0) {1};
    \node[scale=0.75] at (7.5,1) {3};
    \node[scale=0.75] at (7.5,2) {5};
    \node[scale=0.75] at (7.5,3) {7};
    \node[scale=0.75] at (7.5,4) {9};
    \node[scale=0.75] at (7.5,0.5) {2};
    \node[scale=0.75] at (7.5,1.5) {4};
    \node[scale=0.75] at (7.5,2.5) {6};
    \node[scale=0.75] at (7.5,3.5) {8};
    \node[scale=0.75] at (7.5,4.5) {10};
    \end{tikzpicture}
\end{center}

By \eqref{equ: size}, we see that $|\Pi_{[10],[10]}(F) = \frac{3!6!5!}{1!3!}$.

Now consider the variation of \eqref{equ: size} which counts paths in $\Pi_{[n],[n]}(F)$ having sources and sinks with only odd (only even) indices. For example, when $F=F_{[8,10]}\bullet F_{[3,8]}\bullet F_{[1,5]}$, these are the paths covering the solid and dotted (dashed and dotted) edges in 

\begin{center}
\begin{tikzpicture}
    \draw[dotted] (0.5,4) -- (3.5,2.25) -- (6,1); 
    \draw[dashed] (0,4.5) -- (0.5,4); 
    \draw (0,4) -- (0.5,4); 
    \draw[dashed] (0,3.5) -- (0.5,4); 
    \draw (0,3) -- (1,3) -- (3.5,2.25); 
    \draw[dashed] (0,2.5) -- (1,2.5) -- (3.5,2.25); 
    \draw[dashed] (0.5,4) -- (1,4.5) -- (7,4.5); 
    \draw (0.5,4) -- (1,4) -- (7,4);

    \draw[dashed] (0,1.5) -- (2.5,1.5) -- (3.5,2.25);
    \draw (0,1) -- (2.5,1) -- (3.5,2.25);
    \draw (0,2) -- (2.5,2) -- (3.5,2.25);
    \draw[dashed] (3.5,2.25) -- (4.5,2.5) -- (7,2.5);
    \draw (3.5,2.25) -- (4.5,3) -- (7,3);
    \draw[dashed] (3.5,2.25) -- (4.5,3.5) -- (7,3.5);

    \draw[dashed] (0,0.5) -- (5,0.5) -- (6,1); 
    \draw (0,0) -- (5,0) -- (6,1);
    \draw (6,1) -- (7,0);
    \draw[dashed] (6,1) -- (7,0.5);
    \draw (6,1) -- (7,1);
    \draw[dashed] (6,1) -- (7,1.5);
    \draw (6,1) -- (7,2);

    \node[scale=0.75] at (-0.5,0) {1};
    \node[scale=0.75] at (-0.5,1) {3};
    \node[scale=0.75] at (-0.5,2) {5};
    \node[scale=0.75] at (-0.5,3) {7};
    \node[scale=0.75] at (-0.5,4) {9};
    \node[scale=0.75] at (-0.5,0.5) {2};
    \node[scale=0.75] at (-0.5,1.5) {4};
    \node[scale=0.75] at (-0.5,2.5) {6};
    \node[scale=0.75] at (-0.5,3.5) {8};
    \node[scale=0.75] at (-0.5,4.5) {10};
    
    \node[scale=0.75] at (7.5,0) {1};
    \node[scale=0.75] at (7.5,1) {3};
    \node[scale=0.75] at (7.5,2) {5};
    \node[scale=0.75] at (7.5,3) {7};
    \node[scale=0.75] at (7.5,4) {9};
    \node[scale=0.75] at (7.5,0.5) {2};
    \node[scale=0.75] at (7.5,1.5) {4};
    \node[scale=0.75] at (7.5,2.5) {6};
    \node[scale=0.75] at (7.5,3.5) {8};
    \node[scale=0.75] at (7.5,4.5) {10};
    \end{tikzpicture}
\end{center}

We see that $|\Pi_{\{1,3,\dotsc,9\},\{1,3,\dotsc,9\}}(F)| = \frac{1!3!3!}{0!2!}$ and $|\Pi_{\{2,4,\dotsc,10\},\{2,4,\dotsc,10\}}(F)|=\frac{2!3!2!}{1!1!}$.

More generally, we have the following.

\begin{lem}\label{eq:counting odd-evens} Let $F = F_{[c_1,d_1]} \bullet \cdots \bullet F_{[c_{r+1},d_{r+1}]}$ be a descending star network with $m_i$, $k_i$ defined as \eqref{equ: size of stars}. Define
\begin{equation*}
\begin{gathered}
I \defeq 2\mathbb Z \cap [n], \qquad \Bar I \defeq [n] \ssm I.
\end{gathered}
\end{equation*}
Then we have
\begin{equation}\label{eq:evenoddfraction}
    |\Pi_{I,I}(F)| \cdot |\Pi_{\Bar I, \Bar I}(F)| = \frac{\lfloor \frac{m_1}{2}\rfloor!\lceil \frac{m_1}{2}\rceil!\cdots \lfloor \frac{m_{r+1}}{2}\rfloor!\lceil\frac{m_{r+1}}{2}\rceil!}{\lfloor \frac{k_1}{2}\rfloor!\lceil \frac{k_1}{2}\rceil!\cdots \lfloor \frac{k_{r}}{2}\rfloor!\lceil\frac{k_{r}}{2}\rceil!}.
\end{equation}
\end{lem}

\begin{proof}
Observe that the product on the left-hand side of (\ref{eq:evenoddfraction}) equals the cardinality of path families $\pi = (\pi_1,\dotsc,\pi_n)$ covering $F$ such that
for all $i$, the number
$\snk(\pi_i)$ has the same parity as $i$.  Call this set $\Pi_{\text{par}}(F)$ so that $|\Pi_{I,I}(F)| \cdot |\Pi_{\Bar I, \Bar I}(F)| = |\Pi_{\text{par}}(F)|$.

Observe that when $F$ is the simple star $F_{[c_1, d_1]}$, we have $|\Pi_{\text{par}}(F)|=\lfloor{\frac{m_1}{2}}\rfloor!\lceil \frac{m_1}{2}\rceil!$ and \eqref{eq:evenoddfraction} is true.

    Now suppose that the equation is true for descending star networks which are 
    concatenations of at most $r$ stars, and consider a concatenation 
    \begin{equation}\label{eq:Fs+2}
        F = F_{[c_1,d_1]} \bullet \cdots \bullet F_{[c_{r+1},d_{r+1}]}
    \end{equation}
    of $r+1$ stars. Define $F' = F_{[c_1,d_1]} \bullet \cdots \bullet F_{[c_{r},d_{r}]}$ so that we have $F = F' \bullet F_{[c_{r+1},d_{r+1}]}$.  

Clearly each element of $\Pi_{\text{par}}(F)$ can be obtained by
\begin{enumerate}
    \item concatenating an element of $\Pi_{\text{par}}(F')$ with an element of $\Pi_{\text{par}}(F_{[c_{r+1},d_{r+1}]})$ in $F' \circ F_{[c_{r+1}, d_{r+1}]}$, 
\item collapsing the $k_{r}$ edges from the central vertex of $F_{[c_{r},d_{r}]}$ to $F_{[c_{r+1},d_{r+1}]}$ into
a single edge $(z_r, z_{r+1})$ of $F$ which is covered by the $k_{r}$ corresponding path segments.
\end{enumerate}
Now observe that for any permutation $w \in \sn$,
the number of path families of type $w$ constructed in step $(1)$ above
equals $\lfloor \tfrac{k_{r}}2 \rfloor !
    \lceil \tfrac{k_{r}}2 \rceil !$
times the number of path families ($0$ or $1$) of type $w$ in $\Pi_{\text{par}}(F)$ after the collapsing done in step $(2)$.
It follows that we have

\begin{align*}
    |\Pi_{\text{par}}(F)|&=\frac{1}{\lfloor \frac{k_r}{2}\rfloor!\lceil \frac{k_r}{2}\rceil!}|\Pi_{\text{par}}(F'\circ F_{[c_{r+1},d_{r+1}]})|\\
    & =\frac{1}{\lfloor \frac{k_r}{2}\rfloor!\lceil \frac{k_r}{2}\rceil!}|\Pi_{\text{par}(F')}| |\Pi_{\text{par}}(F_{[c_{r+1},d_{r+1}]})|\\
    & =\frac{1}{\lfloor \frac{k_r}{2}\rfloor!\lceil \frac{k_r}{2}\rceil!}\frac{\lfloor \frac{m_1}{2}\rfloor!\lceil \frac{m_1}{2}\rceil!\cdots \lfloor \frac{m_r}{2}\rfloor!\lceil\frac{m_r}{2}\rceil!}{\lfloor \frac{k_1}{2}\rfloor!\lceil \frac{k_1}{2}\rceil!\cdots \lfloor \frac{k_{r-1}}{2}\rfloor!\lceil\frac{k_{r-1}}{2}\rceil!}\cdot \floor*{ \frac{m_{r+1}}{2}} !\ceil*{ \frac{m_{r+1}}{2}}!
\end{align*}

as desired.

\end{proof}

We keep using the definitions of $m_i$ and $k_i$ as in Lemma \ref{eq:counting odd-evens}.   Let $h \in [1,n]$, then define sets
  \begin{equation}\label{IJs}
    \begin{alignedat}{2}
      J &= [h], &\qquad I &= [n] \cap 2\mathbb Z,\\
      \ol J &= [h+1,n], &\qquad \ol I &= [n] \ssm 2\mathbb Z.
    \end{alignedat}
    \end{equation}

Now consider the variation of \eqref{equ: size} which counts paths in $\Pi_{[h],[h]}(F)$ when $F$ is a descending star network of order $n$ and $h<n$. For example, when $F=F_{[8,10]}\bullet F_{[3,8]}\bullet F_{[1,5]}$ and $h=6$, these are the paths covering the solid edges in 

\begin{center}
    \begin{tikzpicture}
    \draw[dashed] (0.5,4) -- (3.5,2.25);
    \draw (3.5,2.25) -- (6,1);
    
    \draw[dashed] (0,4.5) -- (0.5,4); 
    \draw[dashed] (0,4) -- (0.5,4); 
    \draw[dashed] (0,3.5) -- (0.5,4); 
    \draw[dashed] (0,3) -- (1,3) -- (3.5,2.25); 
    \draw (0,2.5) -- (1,2.5) -- (3.5,2.25); 
    \draw[dashed] (0.5,4) -- (1,4.5) -- (7,4.5); 
    \draw[dashed] (0.5,4) -- (1,4) -- (7,4);

    \draw (0,1.5) -- (2.5,1.5) -- (3.5,2.25);
    \draw (0,1) -- (2.5,1) -- (3.5,2.25);
    \draw (0,2) -- (2.5,2) -- (3.5,2.25);
    \draw (3.5,2.25) -- (4.5,2.5) -- (7,2.5);
    \draw[dashed] (3.5,2.25) -- (4.5,3) -- (7,3);
    \draw[dashed] (3.5,2.25) -- (4.5,3.5) -- (7,3.5);

    \draw (0,0.5) -- (5,0.5) -- (6,1); 
    \draw (0,0) -- (5,0) -- (6,1);
    \draw (6,1) -- (7,0);
    \draw (6,1) -- (7,0.5);
    \draw (6,1) -- (7,1);
    \draw (6,1) -- (7,1.5);
    \draw (6,1) -- (7,2);

    \node[scale=0.75] at (-0.5,0) {1};
    \node[scale=0.75] at (-0.5,1) {3};
    \node[scale=0.75] at (-0.5,2) {5};
    \node[scale=0.75] at (-0.5,3) {7};
    \node[scale=0.75] at (-0.5,4) {9};
    \node[scale=0.75] at (-0.5,0.5) {2};
    \node[scale=0.75] at (-0.5,1.5) {4};
    \node[scale=0.75] at (-0.5,2.5) {6};
    \node[scale=0.75] at (-0.5,3.5) {8};
    \node[scale=0.75] at (-0.5,4.5) {10};
    
    \node[scale=0.75] at (7.5,0) {1};
    \node[scale=0.75] at (7.5,1) {3};
    \node[scale=0.75] at (7.5,2) {5};
    \node[scale=0.75] at (7.5,3) {7};
    \node[scale=0.75] at (7.5,4) {9};
    \node[scale=0.75] at (7.5,0.5) {2};
    \node[scale=0.75] at (7.5,1.5) {4};
    \node[scale=0.75] at (7.5,2.5) {6};
    \node[scale=0.75] at (7.5,3.5) {8};
    \node[scale=0.75] at (7.5,4.5) {10};
    \end{tikzpicture}
\end{center}

We see that $|\Pi_{[6,6]}(F)|=\frac{(6-3+1)!}{3!} \cdot 5!$.

More generally, we have the following.

\begin{lem}\label{eq:counting top-bots}
    Fix $h\in[n]$ and let $F=F_{[c_1,d_1]} \bullet \cdots \bullet F_{[c_r,d_r]}$ be a descending star network such that $F_{[c_p,d_p]}$ is the only star with $[c_p,d_p] \cap [h,h+1] \neq \emptyset$. Then 
\begin{equation*}
        | \Pi_{[h],[h]}(F) | \cdot | \Pi_{[h+1,n],[h+1,n]}(F) | = \frac{(h-c_p+1)!(d_p-h)!m_1! \cdots m_{p-1}!m_{p+1}!\cdots m_r!}{k_1! \cdots k_{p-1}! k_{p+1}! \cdots k_{r-1}!}.
    \end{equation*}

\end{lem}

\begin{proof}
Observe that the number $|\Pi_{[h],[h]}(F)|$ is $\frac{(h-c_p+1)!}{k_p!}\frac{m_{p+1}!}{k_{p+1}!}\cdots\frac{m_{r-1}!}{k_{r-1}!} \cdot m_r!$. 
Similarly, the number $|\Pi_{[h+1,n],[h+1,n]}(F)|$ is  $\frac{m_{1}!}{k_{1}!}\cdots\frac{m_{p-1}!}{k_{p-1}!} \cdot (d_p-h)!$.
\end{proof}

\begin{cor}\label{cor: ascending star network}
    In the same situation as Lemma \ref{eq:counting top-bots}, we have \begin{equation*}
        | \Pi_{I,I}(F)  | \cdot |\Pi_{\Bar{I},\Bar{I}}(F)| \leq | \Pi_{J,J}(F) | \cdot | \Pi_{\Bar{J},\Bar{J}}(F)  |.
    \end{equation*}
\end{cor}

\begin{proof}

By Lemma \ref{eq:counting odd-evens}, we have 
    \begin{equation}\label{equ: no induction}
        |\Pi_{I,I}(F)  | \cdot | \Pi_{\Bar{I},\Bar{I}}(F)| =  \frac{\lfloor \frac{m_1}{2}\rfloor!\lceil \frac{m_1}{2}\rceil!\cdots \lfloor \frac{m_{r}}{2}\rfloor!\lceil\frac{m_{r}}{2}\rceil!}{\lfloor \frac{k_1}{2}\rfloor!\lceil \frac{k_1}{2}\rceil!\cdots \lfloor \frac{k_{r-1}}{2}\rfloor!\lceil\frac{k_{r-1}}{2}\rceil!}.
    \end{equation}

Now apply Proposition~\ref{p:max factorial} to \begin{equation*}\frac{\floor*{\frac{m_i}{2}}!\ceil*{\frac{m_i}{2}}!}{\floor*{\frac{k_i}{2}}!\ceil*{\frac{k_i}{2}}!}\end{equation*} for $i=1,\ldots, p-1, p+1, \ldots, r-1$ and apply Proposition~\ref{p:factorial inequality} to $\floor*{\frac{m_p}{2}}!\ceil*{\frac{m_p}{2}}!$ and $\floor*{\frac{m_r}{2}}!\ceil*{\frac{m_r}{2}}!$, using the sums $m_p=(h-c_p+1)+(d_p-h)$, $m_r=m_r+0$. We have that \eqref{equ: no induction} is at most
\begin{equation*}
        \frac{(h-c_p+1)!(d_p-h)!}{\floor*{\frac{k_p}{2}}!\ceil*{\frac{k_p}{2}}!} \cdot \frac{m_1!}{k_1!}\cdots \frac{m_{p-1}!}{k_{p-1}!}\frac{m_{p+1}!}{k_{p+1}!}\cdots \frac{m_{r-1}!}{k_{r-1}!}\cdot m_r!.
    \end{equation*}
By Lemma~\ref{eq:counting top-bots}, this equals to \begin{equation*}
        \frac{| \Pi_{J,J}(F) | \cdot | \Pi_{\Bar{J},\Bar{J}}(F)  |}{\floor{\frac{k_p}{2}}!\ceil{\frac{k_p}{2}}!}.
    \end{equation*}
    
\end{proof}

Now consider the variation of \eqref{equ: size} which counts paths in $\Pi_{[n]\setminus\{i\},[n]\setminus\{i\}}(F)$ when $F$ is a descending star network of order $n$. For example, when $F=F_{[8,10]}\bullet F_{[3,8]}\bullet F_{[1,5]}$ and $i =3$, these are the paths covering the solid edges in 

\begin{center}
    \begin{tikzpicture}
    \draw (0.5,4) -- (3.5,2.25) -- (6,1); 
    \draw (0,4.5) -- (0.5,4); 
    \draw (0,4) -- (0.5,4); 
    \draw (0,3.5) -- (0.5,4); 
    \draw (0,3) -- (1,3) -- (3.5,2.25); 
    \draw (0,2.5) -- (1,2.5) -- (3.5,2.25); 
    \draw (0.5,4) -- (1,4.5) -- (7,4.5); 
    \draw (0.5,4) -- (1,4) -- (7,4);

    \draw (0,1.5) -- (2.5,1.5) -- (3.5,2.25);
    \draw[dashed] (0,1) -- (2.5,1) -- (3.5,2.25);
    \draw (0,2) -- (2.5,2) -- (3.5,2.25);
    \draw (3.5,2.25) -- (4.5,2.5) -- (7,2.5);
    \draw (3.5,2.25) -- (4.5,3) -- (7,3);
    \draw (3.5,2.25) -- (4.5,3.5) -- (7,3.5);

    \draw (0,0.5) -- (5,0.5) -- (6,1); 
    \draw (0,0) -- (5,0) -- (6,1);
    \draw (6,1) -- (7,0);
    \draw (6,1) -- (7,0.5);
    \draw[dashed] (6,1) -- (7,1);
    \draw (6,1) -- (7,1.5);
    \draw (6,1) -- (7,2);

    \node[scale=0.75] at (-0.5,0) {1};
    \node[scale=0.75] at (-0.5,1) {3};
    \node[scale=0.75] at (-0.5,2) {5};
    \node[scale=0.75] at (-0.5,3) {7};
    \node[scale=0.75] at (-0.5,4) {9};
    \node[scale=0.75] at (-0.5,0.5) {2};
    \node[scale=0.75] at (-0.5,1.5) {4};
    \node[scale=0.75] at (-0.5,2.5) {6};
    \node[scale=0.75] at (-0.5,3.5) {8};
    \node[scale=0.75] at (-0.5,4.5) {10};
    
    \node[scale=0.75] at (7.5,0) {1};
    \node[scale=0.75] at (7.5,1) {3};
    \node[scale=0.75] at (7.5,2) {5};
    \node[scale=0.75] at (7.5,3) {7};
    \node[scale=0.75] at (7.5,4) {9};
    \node[scale=0.75] at (7.5,0.5) {2};
    \node[scale=0.75] at (7.5,1.5) {4};
    \node[scale=0.75] at (7.5,2.5) {6};
    \node[scale=0.75] at (7.5,3.5) {8};
    \node[scale=0.75] at (7.5,4.5) {10};
    \end{tikzpicture}
\end{center}

We see that $|\Pi_{[n]\setminus\{i\},[n]\setminus\{i\}}(F)| = \frac{3!5!4!}{1!2!}$. 

More generally, we have the following. 

\begin{lem}\label{lem: minusxi}
    Let $F=F_{[c_1,d_1]} \bullet \cdots \bullet F_{[c_r,d_r]}$ be a descending star network with $m_i$, $k_i$ defined as \eqref{equ: size of stars}. Let $\{i\} \in [c_p,d_p] \cap [c_{p+1},d_{p+1}] \cap \cdots \cap [c_q,d_q]$ for some $p ,q \in [1,r]$. Then, we have 

\begin{equation*}
    |\Pi_{[n]\setminus\{i\},[n]\setminus\{i\}}(F)| \cdot |\Pi_{\{i\}, \{i\}}(F)| = \frac{m_1! \cdots m_{p-1}! (m_{p} -1)! \cdots (m_q -1)! m_{q+1}! \cdots m_r!}{k_1! \cdots k_{p-1}! (k_p -1)! \cdots (k_{q-1} -1)! k_q! \cdots k_{r-1}!}.
\end{equation*}

\end{lem}
\begin{proof}
    Note that $|\Pi_{\{i\}, \{i\}}(F)| =1$, so we just need to compute $|\Pi_{J,J}(F)|$. Delete source $i$ and sink $i$ from $F$, and delete all edges incident upon source $i$ or sink $i$ from $F$. Then the resulting graph is a new descending star network with $n-1$ sources and $n-1$ sinks. Re-label sources and sinks by from bottom to top by $\{1,\cdots, n-1\}$, we obtain the desired expression by Lemma \ref{eq:counting number of path families}. 
\end{proof}

\section{Merging stars in a star network}\label{s:merge}

We introduce some elementary permanent inequalities that will be useful later.

\begin{prop}\label{prop: obvious perm inequality}
    Let $A$ and $B$ be two $n\times n$ matrices with $0\leq a_{ij} \leq b_{ij}$ for all $i,j$.
    Then $\perm(A) \leq \perm(B)$. In particular, for $L\subseteq [n]$, we have \begin{equation}\perm(A_{L,L})\perm(A_{\Bar{L}\Bar{L}}) \leq \perm(B_{L,L})\perm(B_{\Bar{L}\Bar{L}}).\end{equation}
\end{prop}

\begin{proof}
   For all %permutation 
   $w \in \sn$, 
   we have 
   $\permmon aw \leq \permmon bw$, and
   $\perm(A) \leq \perm(B)$.
\end{proof}

Observe that for a general descending star network $F$, there may be multiple stars $F_{[c_p,d_p]}$ with $[c_p, d_p]\cap[h, h+1]\neq \emptyset$. We now define a function that can modify the star network $F$ so that the resulting star network satisfies the conditions in Lemma~\ref{eq:counting top-bots}. For each number $h \in [n]$, define the \emph{$h$-merge function}
$\phi_h: \D_n \rightarrow \D_n$ as follows.

\begin{enumerate}
\item Given $F = F_{[c_1,d_1]} \bullet \cdots \bullet F_{[c_r,d_r]}$, define
  \begin{equation*}
    \begin{aligned}
      p &= \max \{k \,|\, [c_k,d_k] \cap [h,h+1] \neq \emptyset \},\\
      q &= \min \{k \,|\, [c_k,d_k] \cap [h,h+1] \neq \emptyset \}.
    \end{aligned}
  \end{equation*}
\item Define $\phi_h(F) = F_{[c_1,d_1]} \bullet \cdots \bullet F_{[c_{q-1},d_{q-1}]} \bullet F_{[c_p,d_q]} \bullet F_{[c_{p+1},d_{p+1}]} \bullet \cdots \bullet F_{[c_r,d_r]}$.
\end{enumerate}

For example, consider 

\begin{equation*}
    F = F_{[6,10]} \bullet F_{[5,8]} \bullet F_{[3,6]} \bullet F_{[1,3]}=\begin{tikzpicture}[baseline=60]
    \draw (0.5,3.5) -- (2.5,2.75) -- (4.5,1.75) -- (6.5,0.5); 
    \draw (0,4.5) -- (0.5,3.5); 
    \draw (0,4) -- (0.5,3.5); 
    \draw (0,3.5) -- (0.5,3.5); 
    \draw (0,3) -- (0.5,3.5); 
    \draw (0,2.5) -- (0.5,3.5); 
    \draw (0.5,3.5) -- (1,4) -- (7,4);
    \draw (0.5,3.5) -- (1,4.5) -- (7,4.5);
    
    \draw (0,2) -- (2,2) -- (2.5,2.75);
    \draw (2.5,2.75) -- (3,3.5) -- (7,3.5);
    \draw (2.5,2.75) -- (3,3) -- (7,3);
    
    \draw (0,1.5) -- (4,1.5) -- (4.5,1.75);
    \draw (0,1) -- (4,1) -- (4.5,1.75);
    \draw (4.5,1.75) -- (5,2.5) -- (7,2.5);
    \draw (4.5,1.75) -- (5,2) -- (7,2);
    \draw (4.5,1.75) -- (5,1.5) -- (7,1.5);

    \draw (0,0.5) -- (7,0.5); 
    \draw (0,0) -- (6,0) -- (7,1);
    \draw (6.5,0.5) -- (7,0);

    \node[scale=0.75] at (-0.5,0) {1};
    \node[scale=0.75] at (-0.5,1) {3};
    \node[scale=0.75] at (-0.5,2) {5};
    \node[scale=0.75] at (-0.5,3) {7};
    \node[scale=0.75] at (-0.5,4) {9};
    \node[scale=0.75] at (-0.5,0.5) {2};
    \node[scale=0.75] at (-0.5,1.5) {4};
    \node[scale=0.75] at (-0.5,2.5) {6};
    \node[scale=0.75] at (-0.5,3.5) {8};
    \node[scale=0.75] at (-0.5,4.5) {10};
    
    \node[scale=0.75] at (7.5,0) {1};
    \node[scale=0.75] at (7.5,1) {3};
    \node[scale=0.75] at (7.5,2) {5};
    \node[scale=0.75] at (7.5,3) {7};
    \node[scale=0.75] at (7.5,4) {9};
    \node[scale=0.75] at (7.5,0.5) {2};
    \node[scale=0.75] at (7.5,1.5) {4};
    \node[scale=0.75] at (7.5,2.5) {6};
    \node[scale=0.75] at (7.5,3.5) {8};
    \node[scale=0.75] at (7.5,4.5) {10};

\end{tikzpicture}.
\end{equation*}
Observe that we have
    $\phi_1(F) = \phi_9(F) = \phi_{10}(F) = F$,  %\end{equation*}
since $[1,2]$, $[9,10]$, $[10,11]$ each intersects only one of the four given intervals.  On the other hand, we have
\begin{equation*}
    \phi_{4}(F) = F_{[6,10]} \bullet F_{[3,8]} \bullet F_{[1,3]} = \begin{tikzpicture}[baseline=60]
    \draw (0.5,3.5) -- (3.5,2.25) -- (4.5,1) -- (6,1) -- (6.5,0.5); 
    \draw (0,4.5) -- (0.5,3.5); 
    \draw (0,4) -- (0.5,3.5); 
    \draw (0,3.5) -- (0.5,3.5); 
    \draw (0,3) -- (0.5,3.5); 
    \draw (0,2.5) -- (0.5,3.5); 
    \draw (0.5,3.5) -- (1,4.5) -- (7,4.5);
    \draw (0.5,3.5) -- (1,4) -- (7,4);

    \draw (0,1.5) -- (2.5,1.5) -- (3.5,2.25);
    \draw (0,1) -- (2.5,1) -- (3.5,2.25);
    \draw (0,2) -- (2.5,2) -- (3.5,2.25);
    \draw (3.5,2.25) -- (4.5,1.5) -- (7,1.5);
    \draw (3.5,2.25) -- (4.5,2) -- (7,2);
    \draw (3.5,2.25) -- (4.5,2.5) -- (7,2.5);
    \draw (3.5,2.25) -- (4.5,3) -- (7,3);
    \draw (3.5,2.25) -- (4.5,3.5) -- (7,3.5);

    \draw (0,0.5) -- (7,0.5); 
    \draw (0,0) -- (6,0) -- (7,1);
    \draw (6.5,0.5) -- (7,0);

    \node[scale=0.75] at (-0.5,0) {1};
    \node[scale=0.75] at (-0.5,1) {3};
    \node[scale=0.75] at (-0.5,2) {5};
    \node[scale=0.75] at (-0.5,3) {7};
    \node[scale=0.75] at (-0.5,4) {9};
    \node[scale=0.75] at (-0.5,0.5) {2};
    \node[scale=0.75] at (-0.5,1.5) {4};
    \node[scale=0.75] at (-0.5,2.5) {6};
    \node[scale=0.75] at (-0.5,3.5) {8};
    \node[scale=0.75] at (-0.5,4.5) {10};
    
    \node[scale=0.75] at (7.5,0) {1};
    \node[scale=0.75] at (7.5,1) {3};
    \node[scale=0.75] at (7.5,2) {5};
    \node[scale=0.75] at (7.5,3) {7};
    \node[scale=0.75] at (7.5,4) {9};
    \node[scale=0.75] at (7.5,0.5) {2};
    \node[scale=0.75] at (7.5,1.5) {4};
    \node[scale=0.75] at (7.5,2.5) {6};
    \node[scale=0.75] at (7.5,3.5) {8};
    \node[scale=0.75] at (7.5,4.5) {10};

\end{tikzpicture},
\end{equation*}
since the intersection of $[4,5]$ with $[3,6]$ and $[5,8]$ leads to the merging of these two intervals into $[3,8]$.
Similarly, we have
    \begin{gather*}
 \phi_2(F) = \phi_3(F) = F_{[6,10]} \bullet F_{[5,8]} \bullet F_{[1,6]}, \\   
 \phi_5(F) = \phi_6(F) = F_{[3,10]} \bullet F_{[1,3]}, \\
    \phi_7(F) = \phi_8(F) = F_{[5,10]} \bullet F_{[3,6]} \bullet F_{[1,3]}.
\end{gather*}

The path matrices of $F$ and $\phi_h(F)$ have the following permanent inequality. 

\begin{cor}\label{more ones implies greater per}
    Let $F$ be a descending star network of order $n$ and fix $h\in[n]$. 
    Let $A$ be the path matrix of $F$ and $B$ be the path matrix of $\phi_h(F)$. Then for all $L \subseteq [n]$ we have \begin{equation*}
        \perm(A_{L,L})\perm(A_{\Bar{L}\Bar{L}})\leq \perm(B_{L,L})\perm(B_{\Bar{L}\Bar{L}}).
    \end{equation*} 
\end{cor}

\begin{proof}
    Observe that if there is a path from source $i$ to sink $j$ in $F$, then there is also a path from the same source and sink in $\phi_h(F)$. Therefore, $0\leq a_{ij}\leq b_{ij}$ for all $i,j\in[n]$. Then we have the desired inequality by Proposition~\ref{prop: obvious perm inequality}.
 \end{proof}

 When $L=[h]$, Corollary~\ref{more ones implies greater per} has a special case where the inequality becomes an equality.

\begin{prop}\label{p: star network contraction}
  Fix $h \in [n]$, and let $F$ be a descending star network of order $n$
  with path matrix $A$.
  Let $B$ be the path matrix of the descending star network
  $\phi_h(F)$. Then we have 
  \begin{enumerate}[(i)]
  \setlength{\itemsep}{0.7em}
  \item $A_{[h],[h]} = B_{[h],[h]}$ and
    $A_{[h+1,n],[h+1,n]} = B_{[h+1,n],[h+1,n]}$.
    
  \item $|\Pi_{[h],[h]}(F)|\cdot|\Pi_{[h+1,n],[h+1,n]}(F)| = |\Pi_{[h],[h]}(\phi_h(F))|\cdot|\Pi_{[h+1,n],[h+1,n]}(\phi_h(F))|$.
  \end{enumerate}
\end{prop}
\begin{proof}
(i) First, we show that $A_{[1,h],[1,h]} = B_{[1,h],[1,h]}$. 
  
  Case 1: Let $i > j \in [1,h]$. If $a_{ij} =1$, then, source $i$ and sink $j$ belong to the same component of $F$. It's easy to see that they also belong to the same component of $\phi_h(F)$. Hence, $b_{ij} = 1$. If $a_{ij} = 0$, then source $i$ and sink $j$ do not belong to the same component of $F$. By the definition of $\phi_h$, $i$ and $j$ do not belong to the same component of $\phi_h(F)$. Hence, $b_{ij} = 0$. 

  Case 2: Let $i \leq j \in [1,h]$. If $a_{ij} = 1$, there is a path from source $i$ to sink $j$ in $F$. So, there exists an index $m$ such that $i,j\in [c_m,d_m]$. Since every interval appearing in $\phi_h(F)$ is a union of some intervals of $F$, $i$, $j$ are in an interval of $\phi_h(F)$. Hence, $b_{ij} = 1$. If $a_{ij} = 0$, then no intervals contain both $i$ and $j$. In particular, no intervals contain both $h$ and $i$, then all intervals appearing in $\phi_h(F)$ containing $i$ are equal to intervals containing $i$ in $F$. None of these intervals also contain $j$, so $b_{ij} = 0$. 

  Therefore, $A_{[1,h],[1,h]} = B_{[1,h],[1,h]}$. Similarly, we have $A_{[h+1,n],[h+1,n]} = B_{[h+1,n],[h+1,n]}$.

 \noindent
  (ii) By Observation \ref{o:perlem} and part (i), we have \begin{gather*}
  |\Pi_{[1,h],[1,h]}(F)| = \perm(A_{[1,h],[1,h]})=\perm(B_{[1,h],[1,h]})=|\Pi_{[1,h],[1,h]}(\phi_h(F))|,\\
|\Pi_{[h+1,n],[h+1,n]}(F)| = \perm(A_{[h+1,n],[h+1,n]})=\perm(B_{[h+1,n],[h+1,n]})=|\Pi_{[h+1,n],[h+1,n]}(\phi_h(F))|.
\end{gather*}  \end{proof}

In the previous example $F=F_{[6,10]} \bullet F_{[5,8]}\bullet F_{[3,6]} \bullet F_{[1,3]}$ and $\phi_4(F)=F_{[6,10]} \bullet F_{[3,8]} \bullet F_{[1,3]}$, one can see that \begin{align*}
    |\Pi_{[1,4],[1,4]}(F)| &= \perm(A_{[1,4],[1,4]}) = |\Pi_{[1,4],[1,4]}(\phi_4(F))| = \perm(B_{[1,4],[1,4]}) = 12, \\
    |\Pi_{[5,10],[5,10]}(F)| &= \perm(A_{[5,10],[5,10]}) = |\Pi_{[5,10],[5,10]}(\phi_4(F))| = \perm(B_{[5,10],[5,10]}) = 480, 
\end{align*}where \begin{equation*}
    A=\begin{bmatrix}
1 & 1 & 1 & 0 & 0 & 0 & 0 & 0 & 0 & 0 \\
1 & 1 & 1 & 0 & 0 & 0 & 0 & 0 & 0 & 0 \\
1 & 1 & 1 & 1 & 1 & 1 & 0 & 0 & 0 & 0 \\
1 & 1 & 1 & 1 & 1 & 1 & 0 & 0 & 0 & 0 \\
1 & 1 & 1 & 1 & 1 & 1 & 1 & 1 & 0 & 0 \\
1 & 1 & 1 & 1 & 1 & 1 & 1 & 1 & 1 & 1 \\
1 & 1 & 1 & 1 & 1 & 1 & 1 & 1 & 1 & 1 \\
1 & 1 & 1 & 1 & 1 & 1 & 1 & 1 & 1 & 1 \\
1 & 1 & 1 & 1 & 1 & 1 & 1 & 1 & 1 & 1 \\
1 & 1 & 1 & 1 & 1 & 1 & 1 & 1 & 1 & 1 
\end{bmatrix}, \quad  B=\begin{bmatrix}
1 & 1 & 1 & 0 & 0 & 0 & 0 & 0 & 0 & 0 \\
1 & 1 & 1 & 0 & 0 & 0 & 0 & 0 & 0 & 0 \\
1 & 1 & 1 & 1 & 1 & 1 & \mathbf{1} & \mathbf{1} & 0 & 0 \\
1 & 1 & 1 & 1 & 1 & 1 & \mathbf{1} & \mathbf{1} & 0 & 0 \\
1 & 1 & 1 & 1 & 1 & 1 & 1 & 1 & 0 & 0 \\
1 & 1 & 1 & 1 & 1 & 1 & 1 & 1 & 1 & 1 \\
1 & 1 & 1 & 1 & 1 & 1 & 1 & 1 & 1 & 1 \\
1 & 1 & 1 & 1 & 1 & 1 & 1 & 1 & 1 & 1 \\
1 & 1 & 1 & 1 & 1 & 1 & 1 & 1 & 1 & 1 \\
1 & 1 & 1 & 1 & 1 & 1 & 1 & 1 & 1 & 1 
\end{bmatrix}.
\end{equation*}

\section{Main results}\label{s:main}

Now we can state and prove our main result about products of permanents.
\begin{thm}\label{t:main}
Let $P$ be an $n$-element unit interval order 
labeled as in Algorithm~\ref{a:arlabeling}
and let $A$ be its antiadjacency matrix.
  Fix $h \in [n]$ and define sets
  \begin{equation*}
      I = [h] ,\qquad
      J = [n] \cap 2 \mathbb Z.
  \end{equation*}
  Then we have
  \begin{equation}\label{eq:main}
    \perm(A_{I,I})\perm(A_{\Bar I, \Bar I}) \geq \perm(A_{J,J})\perm(A_{\Bar J, \Bar J}).
  \end{equation}
  \end{thm}
\begin{proof}
Let 
\begin{equation}\label{eq:Ffactor}
F = F_{[c_1,d_1]} \bullet \cdots \bullet F_{[c_r,d_r]}    
\end{equation}
be the descending star network corresponding to $P$
as in Algorithm~\ref{a:PtoF}
and let $B$ be the path matrix of $F$.
By Proposition~\ref{p:P A F psiA} 
we have $B = \psi(A)$.  
By Corollary~\ref{c:AB}
the inequality (\ref{eq:main}) holds for $A$ if and only if
it holds for $B$.
By Corollary~\ref{c:perlem} we have
\begin{equation}
    \begin{gathered}
      \perm(B_{I,I}) \perm(B_{\Bar I,\Bar I}) = 
      | \Pi_{I,I}(F)| \cdot |\Pi_{\Bar I,\Bar I}(F) |,\\
     \perm(B_{J,J}) \perm(B_{\Bar J,\Bar J}) = 
     |\Pi_{J,J}(F) | \cdot | \Pi_{\Bar J,\Bar J}(F) |,
     \end{gathered}
     \end{equation}
and it suffices to show that we have
$|\Pi_{I,I}(F)|\cdot |\Pi_{\Bar I, \Bar I}(F)| \geq |\Pi_{J,J}(F)|\cdot |\Pi_{\Bar J, \Bar J}(F)|$.

Suppose first that no interval 
$[c_i,d_i]$ in
(\ref{eq:Ffactor})
contains $[h,h+1]$ as a subinterval.
Then each source-to-sink path in $F$ has its
source and sink both in $[h]$ or both in $[h+1,n]$. Thus we have
$|\Pi_{I,I}(F)| \cdot |\Pi_{\Bar I, \Bar I}(F)| = |\Pi_{[n],[n]}(F)|$.
By Lemma~\ref{eq:counting number of path families}
this cardinality is
\begin{equation}\label{eq:mkratio}  
 \frac{m_1! \cdots m_r!}{ k_1! \cdots k_{r-1}!},
\end{equation}
where
    $m_i \defeq \text{size}(F_{[c_i,d_i]})$,
and
$k_i \defeq \text{overlap}(F_{[c_i,d_i]},F_{[c_{i+1},d_{i+1}]})$.
But by Propositions~\ref{p:factorial inequality} and \ref{p:max factorial} this is at least
\begin{equation*}
    \frac{\floor{\tfrac{m_1}{2}}!\ceil{\tfrac{m_1}{2}}! \cdots  \floor{\tfrac{m_r}{2}}!\ceil{\tfrac{m_r}{2}}!}
     {\floor{\tfrac{k_1}{2}}!\ceil{\tfrac{k_1}{2}}! \cdots \floor{\tfrac{k_{r-1}}{2}}!\ceil{\tfrac{k_{r-1}}{2}}!},
     \end{equation*}
     which by Lemma~\ref{eq:counting odd-evens} 
     equals
    $|\Pi_{J,J}(F) | \cdot |\Pi_{\Bar{J},\Bar{J}}(F)|$.

Now suppose that some interval 
$[c_i,d_i]$ in
(\ref{eq:Ffactor})
does contain $[h,h+1]$ as a subinterval.
Then we do not have the factorization above (\ref{eq:mkratio}),
but by Proposition~\ref{p: star network contraction} we have
\begin{equation*}
         |\Pi_{I,I}(F)| \cdot |\Pi_{\bar{I},\Bar{I}}(F)| 
        = |\Pi_{I,I}(\phi_h(F)) | \cdot |\Pi_{\Bar{I},\Bar{I}}(\phi_h(F))|.
        \end{equation*}
 By Corollary~\ref{cor: ascending star network}, this is at least
     $|\Pi_{J,J}(\phi_h(F))| \cdot |\Pi_{\Bar{J},\Bar{J}}(\phi_h(F))|$,
     and by Corollary~\ref{more ones implies greater per}
     this is at least
     $|\Pi_{J,J}(F)| \cdot |\Pi_{\Bar{J},\Bar{J}}(F)|$.
\end{proof}

\begin{thm}\label{t:maintwo}
Let $P$ be an $n$-element unit interval order 
labeled as in Algorithm~\ref{a:arlabeling}
and let $A$ be its antiadjacency matrix. Let $J = [n] \cap 2 \mathbb Z$. Then for any $i \in [n]$, we have 
    \begin{equation}\label{eq:secondmain}
        \perm(A_{[n]\setminus \{i\},[n]\setminus \{i\}})\perm(A_{\{i\},\{i\}}) \geq \perm(A_{J,J})\perm(A_{\Bar J, \Bar J}).
    \end{equation}
\end{thm}
\begin{proof}
    Let 
\begin{equation*}
F = F_{[c_1,d_1]} \bullet \cdots \bullet F_{[c_r,d_r]}    
\end{equation*}
be the descending star network corresponding to $P$
as in Algorithm~\ref{a:PtoF}
and let $B$ be the path matrix of $F$.
By Proposition~\ref{p:P A F psiA} 
we have $B = \psi(A)$.  
By Corollary~\ref{c:AB}
the inequality (\ref{eq:secondmain}) holds for $A$ if and only if
it holds for $B$. 
By Corollary~\ref{c:perlem} we have
\begin{equation}
    \begin{gathered}
      \perm(B_{[n]\setminus \{i\},[n]\setminus \{i\}}) \perm(B_{ \{i\}, \{i\}}) = 
      | \Pi_{[n]\setminus \{i\},[n]\setminus \{i\}}(F)| \cdot |\Pi_{ \{i\}, \{i\}}(F) |,\\
     \perm(B_{J,J}) \perm(B_{\Bar J,\Bar J}) = 
     |\Pi_{J,J}(F) | \cdot | \Pi_{\Bar J,\Bar J}(F) |,
     \end{gathered}
     \end{equation}
and it suffices to show that we have
\begin{equation}\label{eq:secondmainalt}
    | \Pi_{[n]\setminus \{i\},[n]\setminus \{i\}}(F)| \cdot |\Pi_{ \{i\}, \{i\}}(F) | \geq |\Pi_{J,J}(F)|\cdot |\Pi_{\Bar J, \Bar J}(F)|.
    \end{equation}
    
Suppose that $i \notin [c_p,d_p] $ for all $p\in [1,r]$. Then, in the descending star network $F$, the path from source $i$ to sink $i$ is a simple (directed) edge. Hence, the left hand side of \eqref{eq:secondmainalt} is \begin{equation*}
    | \Pi_{[n]\setminus \{i\},[n]\setminus \{i\}}(F)| = | \Pi_{[n],[n]}(F)|. 
\end{equation*}
Since every path family with odd sources going to odd sinks and even sources going to even sinks belongs to $\Pi_{[n],[n]}(F)$, this is at least 
\begin{equation*}
   |\Pi_{J,J}(F)|\cdot |\Pi_{\Bar J, \Bar J}(F)|,
\end{equation*} and \eqref{eq:secondmainalt} is true.

On the other hand, suppose that $i \in [c_p,d_p] \cap [c_{p+1},d_{p+1}] \cap \cdots \cap [c_q,d_q]$ for some $p ,q \in [1,r]$. Then, by Lemma \ref{lem: minusxi}, the left-hand side of \eqref{eq:secondmainalt} is \begin{equation*}
     \frac{m_1! \cdots m_{p-1}! (m_{p} -1)! \cdots (m_{q-1}-1)! (m_q -1)! m_{q+1}! \cdots m_r!}{k_1! \cdots k_{p-1}! (k_p -1)! \cdots (k_{q-1} -1)! k_q! \cdots k_{r-1}!}.
\end{equation*}
Note that $m_{j} \geq k_{j}$ for $j=1, \dotsc, q-1$ and $m_{j+1} \geq k_{j}$ for $j=q, \dotsc r-1$. By Propositions~\ref{p:max factorial}, \ref{p:max factorial2}, the above fraction is at least 
\begin{equation*}
    \frac{\lfloor \frac{m_1}{2}\rfloor!\lceil \frac{m_1}{2}\rceil!}{\lfloor \frac{k_1}{2}\rfloor!\lceil \frac{k_1}{2}\rceil!} \cdots 
    \frac{\lfloor \frac{m_{q-1}}{2}\rfloor!\lceil \frac{m_{q-1}}{2}\rceil!}{\lfloor \frac{k_{q-1}}{2}\rfloor!\lceil \frac{k_{q-1}}{2}\rceil!} (m_q -1)!
    \frac{\lfloor \frac{m_{q+1}}{2}\rfloor!\lceil \frac{m_{q+1}}{2}\rceil!}{\lfloor \frac{k_{q}}{2}\rfloor!\lceil \frac{k_{q}}{2}\rceil!} \cdots
    \frac{\lfloor \frac{m_{r}}{2}\rfloor!\lceil \frac{m_{r}}{2}\rceil!}{\lfloor \frac{k_{r-1}}{2}\rfloor!\lceil \frac{k_{r-1}}{2}\rceil!}. 
\end{equation*}
Applying Proposition \ref{p:factorial inequality} to $(m_q-1)! = (m_q-1)! \cdot 1!$, we see that
this is at least 
\begin{equation*}
    \frac{\lfloor \frac{m_1}{2}\rfloor!\lceil \frac{m_1}{2}\rceil!\cdots \lfloor \frac{m_{r}}{2}\rfloor!\lceil\frac{m_{r}}{2}\rceil!}{\lfloor \frac{k_1}{2}\rfloor!\lceil \frac{k_1}{2}\rceil!\cdots \lfloor \frac{k_{r-1}}{2}\rfloor!\lceil\frac{k_{r-1}}{2}\rceil!}, 
\end{equation*}which is $|\Pi_{J,J}(F)| \cdot| \Pi_{\Bar J,\Bar J}(F) |$, by Lemma \ref{eq:counting odd-evens}. Again \eqref{eq:secondmainalt} is true.
\end{proof}

One may guess that we can extend Theorem \ref{t:maintwo} to all TNN matrices. However, there is a counterexample when $n=4$ and $i=3$. Let \begin{equation*}
    M(k) = \begin{bmatrix}
        1& 1& 1& 0 \\1& 1+k& 1+k& k \\ 1& 1+k& 1+k& k \\ 0& 1& 1& 1 
    \end{bmatrix}, \ \text{with $k\geq 0$}. 
\end{equation*} Note that $M(k)$ is TNN since it is the path matrix of the weighted planar network \begin{center}
    \begin{tikzpicture}[scale = 0.8]
    \draw (0,0) -- (2,2); 
    \draw (0,1) -- (2,3); 
    \draw (0,2) -- (2,0);
    \draw (0,3) -- (2,1);
    \draw[thick, color = red] (0.5,1.5) -- (1,2);
    \node[scale=0.75] at (0.6,1.9) {\color{red}$k$};

    \node at (-0.3,0) {1};
    \node at (-0.3,1) {2};
    \node at (-0.3,2) {3};
    \node at (-0.3,3) {4};
    \node at (2.3,0) {1};
    \node at (2.3,1) {2};
    \node at (2.3,2) {3};
    \node at (2.3,3) {4};
    \end{tikzpicture}.
\end{center} Then, we have \begin{equation*}
    \perm(M(k)_{124,124})\perm(M(k)_{3,3}) - \perm(M(k)_{13,13})\perm(M(k)_{24,24}) 
    = -k. 
\end{equation*} Therefore, \eqref{eq:secondmain} is not always true for all TNN matrices.

\section{A Generalization on $J$}\label{s:A gen on J}

A natural generalization of a product of two permanents indexed by even or odd integers is a product of $s$ permanents indexed by integers which are equivalent modulo $s$.
Define \begin{equation*}
    [c,d]_{r}^{s} = \{x\in [c,d] : \ x \equiv r \ (\text{mod } s) \}.  
\end{equation*}
For example, $[5,12]_{0}^{3} = \{6, 9,12\}$, $[5,12]_{1}^{3} = \{7, 10\}$, and $[5,12]_{2}^{3} = \{5, 8, 11\}$.
Naturally, the cardinalities of these sets differ by at most one, and sum to $|[5,12]| = 8$.
More generally,
we have the following elementary number theoretic result.
\begin{obs}\label{o:r mod s counting}
    For $0<s\leq n$ and $[c,d]\subseteq [1, n]$, the cardinalities of $[c,d]_{0}^{s}, \dotsc, [c,d]_{s-1}^{s}$ differ by at most one.
    
\end{obs}

\begin{proof}
Use integer division to write 
$d-c+1=qs +r$
with $r<s$. Then $r$ intervals have cardinality $q+1$ and $s-r$ intervals have cardinality $q$.
\end{proof}

Observe that $q=\floor*{\frac{d-c+1}{s}}$ and $q+1 = \ceil*{\frac{d-c+1}{s}}$. This leads to the following generalization of Proposition~\ref{p:max factorial}.
\begin{prop}
Fix $s,m,k \in \mathbb{N}$ with $k\leq m$ and define integers $k_0, \dotsc, k_{s-1}, m_0, \dotsc, m_{s-1}$ by $k_i=|[1,k]_i^s|$, $m_i=|[1,m]_i^s|$. These numbers satisfy 

\begin{equation}\label{eq:genfracinq}
        \frac{m!}{k!} \geq \frac{m_0! m_1! \cdots
         m_{s-1}!}{k_0! k_1! \cdots k_{s-1}!}.
    \end{equation}

\end{prop}

\begin{proof}

The inequality \eqref{eq:genfracinq} is equivalent to 
\begin{equation}\label{eq:multi ineq}{ m \choose {m_0,m_1,\dotsc,m_{s-1}}} \geq {k\choose k_0,k_1,\dotsc,k_{s-1}}.
\end{equation} 
By the multinomial analog of the Pascal's identity, 
\begin{equation}\label{eq:multi Pascal}
{n \choose a_1, \dotsc, a_t}={n-1 \choose a_1-1, a_2, \dotsc, a_t}+ {n-1 \choose a_1, a_2-1, \dotsc, a_t}+\cdots+ {n-1 \choose a_1, \dotsc, a_t-1},
\end{equation} 
we see that the left-hand side of \eqref{eq:multi Pascal} is greater than each term on the right-hand side of \eqref{eq:multi Pascal}. 
It follows that for $b_1, \cdots, b_t\in\mathbb{N}$, with $b_i\leq a_i$ for $i=1, \dotsc, t$ and $b=b_1+\cdots b_t$, we have 
\begin{equation*}
{n \choose a_1, \dotsc, a_t}\geq {n-b \choose a_1-b_1, a_2-b_2, \dotsc, a_t-b_t}.\end{equation*} The inequality \eqref{eq:multi ineq} follows. 
\end{proof}

Now consider the variation of \eqref{equ: size} which counts paths in $\Pi_{[n],[n]}(F)$ having sources and sinks with only indices that are 1 modulo $3$. For example, when $F=F_{[9,13]}\bullet F_{[5,11]}\bullet F_{[3,6]} \bullet F_{[1,3]}$, these are the paths covering the solid and dotted edges in 
\begin{center}
\begin{tikzpicture}[x=0.75pt,y=0.75pt,yscale=-0.55,xscale=0.75]
\draw  [dash pattern={on 4.5pt off 4.5pt}]  (80,517.79) -- (500,517.79) ;
\draw  [dash pattern={on 4.5pt off 4.5pt}]  (500,517.79) -- (540,477.79) ;
\draw  [dash pattern={on 4.5pt off 4.5pt}]  (540,477.79) -- (580,517.79) ;
\draw    (80,477.79) -- (620,477.79) ;
\draw  [dash pattern={on 4.5pt off 4.5pt}]  (620,517.79) -- (580,517.79) ;
\draw    (580,437.79) -- (620,437.79) ;
\draw    (80,437.79) -- (360,437.79) ;
\draw  [dash pattern={on 4.5pt off 4.5pt}]  (80,397.79) -- (360,397.79) ;
\draw    (360,437.79) -- (420,357.79) ;
\draw  [dash pattern={on 4.5pt off 4.5pt}]  (360,397.79) -- (420,357.79) ;
\draw    (420,357.79) -- (540,477.79) ;
\draw  [dash pattern={on 4.5pt off 4.5pt}]  (420,357.79) -- (480,397.79) -- (620,397.79) ;
\draw    (420,357.79) -- (620,357.79) ;
\draw    (420,357.79) -- (480,317.79) -- (620,317.79) ;
\draw    (420,357.79) -- (300,237.79) ;
\draw    (80,357.79) -- (260,357.79) -- (300,237.79) ;
\draw    (80,317.79) -- (240,317.79) -- (300,237.79) ;
\draw  [dash pattern={on 4.5pt off 4.5pt}]  (80,277.79) -- (220,277.79) -- (300,237.79) ;
\draw    (80,237.79) -- (300,237.79) ;
\draw    (180,117.79) -- (140,197.79) -- (80,197.79) ;
\draw  [dash pattern={on 4.5pt off 4.5pt}]  (80,157.79) -- (140,157.79) -- (180,117.79) ;
\draw    (80,117.79) -- (180,117.79) ;
\draw    (80,77.79) -- (140,77.79) -- (180,117.79) ;
\draw  [dash pattern={on 4.5pt off 4.5pt}]  (80,37.79) -- (140,37.79) -- (180,117.79) ;
\draw  [dash pattern={on 4.5pt off 4.5pt}]  (300,237.79) -- (380,277.79) -- (620,277.79) ;
\draw    (300,237.79) -- (620,237.79) ;
\draw    (300,237.79) -- (380,197.79) -- (620,197.79) ;
\draw  [dash pattern={on 4.5pt off 4.5pt}]  (300,237.79) -- (380,157.79) -- (620,157.79) ;
\draw    (300,237.79) -- (380,117.79) -- (620,117.79) ;
\draw    (180,117.79) -- (220,77.79) -- (620,77.79) ;
\draw  [dash pattern={on 4.5pt off 4.5pt}]  (180,117.79) -- (220,37.79) -- (620,37.79) ;
\draw    (540,477.79) -- (580,437.79) ;
\draw  [dash pattern={on 0.84pt off 2.51pt}]  (180,117.79) -- (300,237.79) ;

\draw (51.03,31.99) node [anchor=north west][inner sep=0.75pt]  [rotate=-359] [align=left] {13};
\draw (51.03,71.99) node [anchor=north west][inner sep=0.75pt]  [rotate=-359] [align=left] {12};
\draw (51.69,111.99) node [anchor=north west][inner sep=0.75pt]  [rotate=-359] [align=left] {11};
\draw (51.03,151.99) node [anchor=north west][inner sep=0.75pt]  [rotate=-359] [align=left] {10};
\draw (57.67,192.21) node [anchor=north west][inner sep=0.75pt]  [rotate=-359] [align=left] {9};
\draw (57.67,230.99) node [anchor=north west][inner sep=0.75pt]  [rotate=-359] [align=left] {8};
\draw (57.67,270.99) node [anchor=north west][inner sep=0.75pt]  [rotate=-359] [align=left] {7};
\draw (57.67,310.99) node [anchor=north west][inner sep=0.75pt]  [rotate=-359] [align=left] {6};
\draw (57.67,352.21) node [anchor=north west][inner sep=0.75pt]  [rotate=-359] [align=left] {5};
\draw (57.69,392.19) node [anchor=north west][inner sep=0.75pt]  [rotate=-359] [align=left] {4};
\draw (58.69,430.83) node [anchor=north west][inner sep=0.75pt]  [rotate=-359] [align=left] {3};
\draw (57.69,470.83) node [anchor=north west][inner sep=0.75pt]  [rotate=-359] [align=left] {2};
\draw (58.02,510.84) node [anchor=north west][inner sep=0.75pt]   [align=left] {1};
\draw (630.69,32.33) node [anchor=north west][inner sep=0.75pt]  [rotate=-359] [align=left] {13};
\draw (630.69,72.33) node [anchor=north west][inner sep=0.75pt]  [rotate=-359] [align=left] {12};
\draw (631.34,112.33) node [anchor=north west][inner sep=0.75pt]  [rotate=-359] [align=left] {11};
\draw (630.69,152.33) node [anchor=north west][inner sep=0.75pt]  [rotate=-359] [align=left] {10};
\draw (637.32,192.56) node [anchor=north west][inner sep=0.75pt]  [rotate=-359] [align=left] {9};
\draw (637.32,231.33) node [anchor=north west][inner sep=0.75pt]  [rotate=-359] [align=left] {8};
\draw (637.32,271.33) node [anchor=north west][inner sep=0.75pt]  [rotate=-359] [align=left] {7};
\draw (637.32,311.33) node [anchor=north west][inner sep=0.75pt]  [rotate=-359] [align=left] {6};
\draw (637.32,352.56) node [anchor=north west][inner sep=0.75pt]  [rotate=-359] [align=left] {5};
\draw (637.34,392.54) node [anchor=north west][inner sep=0.75pt]  [rotate=-359] [align=left] {4};
\draw (638.34,431.17) node [anchor=north west][inner sep=0.75pt]  [rotate=-359] [align=left] {3};
\draw (637.34,471.17) node [anchor=north west][inner sep=0.75pt]  [rotate=-359] [align=left] {2};
\draw (637.67,511.19) node [anchor=north west][inner sep=0.75pt]   [align=left] {1};
\end{tikzpicture}.
\end{center}

Then, we see that \begin{equation*} 
    \Pi_{\{1,4,7,10,13\},\{1,4,7,10,13\}}(F) = \frac{2!2!1!1!}{1!0!0!} = 4. 
\end{equation*}   

More generally, we have the following lemma.

\begin{lem}\label{eq:counting r mod s} Let $F = F_{[c_0,d_0]} \bullet \cdots \bullet F_{[c_t,d_t]}$ be a descending star network. Fix $s\in[n]$. For every $r\in[0,s-1]$, we have
\begin{equation}\label{eq:evenoddfractiongen}
    |\Pi_{[1,n]_{r}^{s},[1,n]_{r}^{s}}(F)| = \frac{ |[c_0,d_0]_{r}^{s}|! \cdot |[c_1,d_1]_r^s|! \cdots |[c_t,d_t]_{r}^{s}|!}{|[c_0,d_1]_{r}^{s}|! \cdot |[c_1,d_2]_r^s|! \cdots |[c_{t-1},d_{t}]_r^s|!}.
\end{equation}
\end{lem}

\begin{proof}
   Observe that $|\Pi_{[1,n]_{r}^{s},[1,n]_{r}^{s}}(F)|$ corresponds to a sub-star network with sources and sinks labeled in $[1,n]_{r}^{s}$. By Lemma \ref{eq:counting number of path families}, we have~\eqref{eq:evenoddfractiongen}. 
\end{proof}

We can now extend Theorem~\ref{t:main} to prove inequalities such as the following for $9\times 9$ matrices:
\begin{equation*}
\begin{split}
\perm(A)
\geq & \perm(A_{13579,13579})\perm(A_{2468,2468}) \\
\geq & \perm(A_{147,147})\perm(A_{258,258})\perm(A_{369,369}) \\
\geq & \perm(A_{159,159})\perm(A_{26,26})\perm(A_{37,37})\perm(A_{48,48}) \geq \cdots \\
\geq & \perm(A_{19,19})a_{22}\cdots a_{88} \\
\geq & a_{11}\cdots a_{99}.
\end{split}
\end{equation*}

\begin{thm}\label{t:rmodsmain}
   Let $A$ be the antiadjacency matrix of an $n$-element unit interval order.

    Then for $s=2,\ldots, n-1$, we have 
    \begin{equation}\label{eq:new inequality}
   \perm(A_{[1,n]_0^s,[1,n]_0^s})\cdots \perm(A_{[1,n]_{s-1}^s,[1,n]_{s-1}^s})\geq \perm(A_{[1,n]_0^{s+1},[1,n]_0^{s+1}})\cdots\perm(A_{[1,n]_s^{s+1},[1,n]_s^{s+1}}).
    \end{equation}
    
\end{thm}

\begin{proof}

    Let $F=F_{[c_0,d_0]} \bullet \cdots \bullet F_{[c_t,d_t]}$ be the descending star network corresponding to $A$, i.e., %$F$ be 
    the descending star network of $\psi(A)$ from Proposition~\ref{p:P A F psiA}. Then by 
    Corollary~\ref{c:perlem}, the $r$th factor on the left-hand-side of \eqref{eq:new inequality} is $|\Pi_{[1,n]_r^s, [1, n]_r^s} (F)|$. By Lemma~\ref{eq:counting r mod s}, this equals $|[c_t, d_t]_r^s|!$ times
\begin{equation}
   \alpha_r:= \frac{ |[c_0,d_0]_{r}^{s}|!}{|[c_0,d_1]_{r}^{s}|!} \cdot \frac{[c_1,d_1]_{r}^{s}|!}{|[c_1,d_2]_r^s|!} \cdots \frac{|[c_{t-1},d_{t-1}]_{r}^{s}|!}{|[c_{t-1},d_{t}]_r^s|!},
\end{equation} i.e., the left-hand-side of \eqref{eq:new inequality} equals
\begin{equation}\label{eq:alphaLHS}
    |[c_t,d_t]_0^s|!\alpha_0\cdots |[c_t, d_t]_{s-1}^s|!\alpha_{s-1}.
\end{equation}
    Now fix $i\in[0,t-1]$ and collect the ratios $\frac{|[c_i, d_i]_r^s|!}{|[c_i,d_{i+1}]_r^s|!}$ for $r=0, \dotsc, s-1$. The product 
    \begin{equation}\label{eq:collect ratios product s}
        \frac{ |[c_i,d_i]_{0}^{s}|!}{|[c_i,d_{i+1}]_{0}^{s}|!} \cdots \frac{|[c_{i},d_{i}]_{s-1}^{s}|!}{|[c_{i},d_{i+1}]_{s-1}^s|!}
    \end{equation}
satisfies 
\begin{equation}\label{eq: cardinality sums s}
\begin{gathered}
|[c_i, d_i]_0^s|+\cdots+|[c_i, d_i]_{s-1}^s|=d_i-c_i+1,  \\ |[c_i, d_{i+1}]_0^s|+\cdots+|[c_i, d_{i+1}]_{s-1}^s|=d_{i+1}-c_i+1,
\end{gathered}
\end{equation}
with terms within each sum differing from one another by at most $1$, and \begin{equation*}d_i-c_i+1\geq d_{i+1}-c_i+1.\end{equation*}

We similarly apply Corollary~\ref{c:perlem} to the $r$th factor on the right-hand-side of \eqref{eq:new inequality} and Lemma~\ref{eq:counting r mod s} to $|\Pi_{[1,n]_r^{s+1},[1,n]_r^{s+1}} (F)|$. This expression equals $|[c_t,d_t]_r^{s+1}|!$ times \begin{equation}
   \beta_r:= \frac{ |[c_0,d_0]_{r}^{s+1}|!}{|[c_0,d_1]_{r}^{s+1}|!} \cdot \frac{[c_1,d_1]_{r}^{s+1}|!}{|[c_1,d_2]_r^{s+1}|!} \cdots \frac{|[c_{t-1},d_{t-1}]_{r}^{s+1}|!}{|[c_{t-1},d_{t}]_r^{s+1}|!},
\end{equation}
 i.e., the right-hand-side of \eqref{eq:new inequality} equals
\begin{equation}\label{eq:betaRHS}
    |[c_t,d_t]_0^{s+1}|!\beta_0\cdots |[c_t, d_t]_{s}^{s+1}|!\beta_{s}.
\end{equation}
 Again fixing $i \in [0,t-1]$ and collecting ratios
 $\frac{|[c_i, d_i]_r^{s+1}|!}{|[c_i,d_{i+1}]_r^{s+1}|!}$ for $r=0, \dotsc, s$, we obtain a product analogous to \eqref{eq:collect ratios product s} in which numerator (denominator) cardinalities differ from one another by at most $1$, and have sums matching those in \eqref{eq: cardinality sums s}. 
By Corollary~\ref{c:multi factorial inequality}, the two products satisfy 
\begin{equation}
    \frac{ |[c_i,d_i]_{0}^{s}|!}{|[c_i,d_{i+1}]_{0}^{s}|!} \cdots \frac{|[c_{i},d_{i}]_{s-1}^{s}|!}{|[c_{i},d_{i+1}]_{s-1}^s|!} \geq \frac{ |[c_i,d_i]_{0}^{s+1}|!}{|[c_i,d_{i+1}]_{0}^{s+1}|!} \cdots \frac{|[c_{i},d_{i}]_{s}^{s+1}|!}{|[c_{i},d_{i+1}]_{s}^{s+1}|!},
\end{equation}
and therefore we have 
\begin{equation}\label{eq:alphabeta}
    \alpha_0\cdots\alpha_{s-1}\geq \beta_0\cdots\beta_s.
\end{equation}

Now consider the cardinalities of the intervals indexed by $t$. By Corollary~\ref{c:max central multi} we have
\begin{equation}\label{eq:t interval products}
    |[c_t, d_t]_0^s|! \cdots |[c_t, d_t]_{s-1}^s|! \geq |[c_t, d_t]_0^{s+1}|! \cdots |[c_t, d_t]_{s}^{s+1}|!. 
\end{equation} 
Combining \eqref{eq:alphabeta} and \eqref{eq:t interval products}, we have that \eqref{eq:alphaLHS} is at least \eqref{eq:betaRHS}, as desired.
\end{proof}

Combine Theorem~\ref{t:rmodsmain} and Theorem~\ref{t:main}, and we have the corollary.  
\begin{cor}
    Let $A$ be an $n\times n$ antiadjacency matrix of a unit interval order. Fix $h \in [n]$. Then, for $s=2,\dotsc,n-1$, we have \begin{equation*}
        \perm(A_{[1,h],[1,h]})\perm(A_{[h+1,n],[h+1,n]}) \geq \perm(A_{[1,n]_0^s,[1,n]_0^s})\cdots \perm(A_{[1,n]_{s-1}^s,[1,n]_{s-1}^s}).
    \end{equation*}
\end{cor}

Combine Theorem~\ref{t:rmodsmain} and Theorem~\ref{t:maintwo}, and we have the corollary. 
\begin{cor}
    Let $A$ be an $n\times n$ antiadjacency matrix of a unit interval order. Fix $i \in [n]$. Then, for $s=2,\dotsc,n-1$, we have \begin{equation*}
        \perm(A_{[n]\setminus\{i\},[n]\setminus\{i\}})\perm(A_{\{i\},\{i\}}) \geq \perm(A_{[1,n]_0^s,[1,n]_0^s})\cdots \perm(A_{[1,n]_{s-1}^s,[1,n]_{s-1}^s}).
    \end{equation*}
\end{cor}

\section{A Generalization to all TNN matrices}\label{s:A gen on all}

In this section, we aim to show that \eqref{eq:evenoddintro} holds for $I = [1,\lfloor \frac{n}{2} \rfloor]$. 

A polynomial $p(x_{11}, x_{12},\cdots, x_{nn})$ in $n^2$ variables is called totally non-negative (TNN) if it satisfies 
\begin{equation}\label{eq:TNNpolydef}
    p(A)\defeq p(a_{11}, a_{12},\cdots, a_{nn}) \geq 0
\end{equation} 
for each $n\times n$ totally non-negative matrix $A=(a_{ij})$. Therefore, \eqref{eq:evenoddintro} is true if and only if the polynomial \begin{equation}\label{TNNpoly}
    \perm(X_{[1,\lfloor \frac{n}{2} \rfloor],[1,\lfloor \frac{n}{2} \rfloor]}) \perm(X_{[\lfloor \frac{n}{2} \rfloor +1,n],[\lfloor \frac{n}{2} \rfloor +1,n]}) - \perm(X_{[n]\cap 2\mathbb Z , [n]\cap 2\mathbb Z})\perm(X_{[n]\setminus 2\mathbb Z , [n]\setminus 2\mathbb Z})
\end{equation} is TNN.

Drake, Gerrish and Skandera~\cite{DGSBruhat} gave a necessary and sufficient condition for a polynomial which is a difference of ``permutation monomials" to be TNN. 

\begin{thm}[Drake, Gerrish, Skandera]\label{pairing test}
            Let $w$ and $u$ be two permutations in $\sn$. Then $w$ is less than or equal to $u$ in the (strong) Bruhat order if and only if the polynomial \[x_{1,w(1)} \cdots x_{n,w(n)} - x_{1,u(1)} \cdots x_{n,u(n)}\] is totally non-negative. 
\end{thm}

Let $t=\floor{\tfrac{n}{2}}$. Observe that the positive term $\perm(X_{\floor{\frac{n}{2}},\floor{\frac{n}{2}}})\perm(X_{\ceil{\frac{n}{2}},\ceil{\frac{n}{2}}})$ contains monomials \begin{equation*}
x_{1,w_1} \cdots x_{t,w_{t}} x_{t+1,w_{t+1}}\cdots x_{n,w_{n}} \quad \text{such that} 
\end{equation*}  
\begin{align*}
            w_1&\cdots w_t \ \text{permutes} \ \{1,\dotsc,t\}, \\
            w_{t+1}&\cdots w_n \ \text{permutes} \ \{t+1,\dotsc,n\}, 
        \end{align*}
        or, in other words, using the Young subgroup, \begin{equation*}
            w=w_1\cdots w_n \in \mathfrak{S}_{1, \dotsc, t}\times \mathfrak{S}_{t+1, \dotsc, n}. 
        \end{equation*}
For convenience, we set 
\begin{equation}\label{def:An}
    \begin{split}
    &\mathfrak A_n = \mathfrak{S}_{1, \dotsc, t}\times \mathfrak{S}_{t+1, \dotsc, n} \\
    = \{&w \in \sn \ | \ w_1\cdots w_{t} \ \text{permutes} \ \{1,\dotsc,t\} \ \text{and} \ w_{t+1} \cdots w_n \ \text{permutes} \ \{t+1,\dotsc,n\} \}.
\end{split}
\end{equation}
Since $t=\floor{\tfrac{n}{2}}$, it is straightforward to see that $|\mathfrak A_n|=t!(n-t)!= \floor{\tfrac{n}{2}}!\ceil{\tfrac{n}{2}}!$. 

For example, $\mathfrak A_4 = \{1234, 1243, 2134, 2143\}$, and $|\mathfrak A_4| = 4$.  

The negative term $\perm(X_{[n]\cap 2\mathbb{Z},[n]\cap 2\mathbb{Z}})\perm(X_{[n]\setminus 2\mathbb{Z}, [n]\setminus 2\mathbb{Z}})$ contains monomials \begin{equation*}   
x_{1,u_1} \cdots x_{n,u_{n}} \quad \text{such that} \end{equation*}
\begin{align*}
            u_1u_3u_5&\cdots  \ \text{permutes} \ \{1,3,5,\cdots\} = [n]\setminus 2\mathbb{Z}, \\
            u_2u_4u_6&\cdots  \ \text{permutes} \ \{2,4,6,\cdots\} = [n]\cap 2\mathbb{Z}.
        \end{align*}
       
We call the set that consists of $u = u_1\cdots u_n$ satisfying the above conditions $\mathfrak B_n$, i.e., \begin{equation}\label{def:Bn}
    \mathfrak B_n = \{u\in \sn \ | \ u_1u_2\cdots u_n \ \text{takes odd and even integers alternately and $u_1$ is odd} \}
\end{equation}

 In the literature~\cite{TanimotoPAP}, $\mathfrak B_n$ is also referred to as the set of parity alternating permutations 
 \begin{equation*}
 \{w=w_1\dotsc w_n \in \sn \,|\, w_1 \text{ is odd and } w_i+w_{i+1} \text{ is odd for all }i=1,\dotsc n-1\}.
 \end{equation*}
 It is easy to see that $|\mathfrak B_n|=\floor{\tfrac{n}{2}}! \ceil{\tfrac{n}{2}}!$.
Then, the polynomial~\eqref{TNNpoly} has this number 
of positive terms (on the left) and negative terms (on the right). Therefore, if we can find a bijective map $f_n: \mathfrak A_n \rightarrow \mathfrak B_n$, such that $w \leq f_n(w) $ for all $w\in \mathfrak A_n$, by Theorem~\ref{pairing test}, the polynomial~\eqref{TNNpoly} is TNN.

Next, we introduce a few necessary maps that will be helpful. 

For $p, q\in[n]$, define an insertion map $\ins p$ and an insertion-swapping map $\inss q$ (when $n+1-q$ is even) given by
\begin{equation*}
\begin{aligned}
\ins p: \sn &\to \mathfrak{S}_{n+1},\\
w= w_1 \cdots w_p \cdots w_n&\mapsto w_1\cdots w_{p-1}(n+1) w_p \cdots w_n,\\
\text{inss}_q: \sn &\to \mathfrak{S}_{n+1},\\
w= w_1 \cdots w_q \cdots w_n&\mapsto w_1 \cdots w_{q-1} (n+1) w_{q+1}w_{q}w_{q+3}w_{q+2}\cdots w_{n}w_{n-1}.\\
\end{aligned}
\end{equation*}

For example, let $w=216435$. Then, \begin{gather*}
    \ins 3(w) = 2176435, \quad \inss3(w) = 2174653, 
\end{gather*}
but $\inss4(w)$ is not well-defined since $n+1-q=7-4=3$ is not even. 

Next, we introduce a straightforward property of the insertion map.

\begin{lem}
    Let $v,w \in \sn$, and $v \leq w$ in the Bruhat order. If $j \leq i$ as integers, then $\ins i(v) \leq \ins j(w) $. 
\end{lem}

\begin{proof}
    $\ins i(v) \leq \ins i(w) \leq \ins j(w)$. 
\end{proof}

When $n$ is even, we can build the set $\mathfrak A_{n+1}$ from $\mathfrak A_n$, by applying the insertion map. 

\begin{obs}\label{obs:ins on An}
    Fix $n = 2k$ and let $\mathfrak A_n$ be defined as in~\eqref{def:An}. Then, we have \begin{equation*}
        \mathfrak A_{n+1} = \{\ins {i}(a) \ | \ i\in[k+1,n+1] \ \text{and} \ a\in \mathfrak A_n \}.
    \end{equation*}
\end{obs}
\begin{proof}
    Let $a\in \mathfrak A_n$. It is easy to see that the first $k$ letters of $\ins i(a)$ with $i\in [k+1,n+1]$ are numbers in $[1,k]$. Therefore, \begin{equation*}
        \mathfrak A_{n+1} \supset \{\ins {i}(a) \ | \ i\in[k+1,n+1] \ \text{and} \ a\in \mathfrak A_n \}.
    \end{equation*}
    
    On the other hand, let $w=w_1\cdots w_{n+1}\in \mathfrak A_{n+1}$ with $w_j = n+1$. Then, $j\in[k+1,n+1]$ by definition. Hence, the permutation $\Bar{w} = \Bar{w}_1\cdots \Bar{w}_n=w_1\cdots w_{j-1}w_{j+1} \cdots w_{n+1} \in \mathfrak A_n$. Thus, $\ins {j}(\Bar{w}) = w$, which gives \begin{equation*}
        \mathfrak A_{n+1} \subset \{\ins {i}(a) \ | \ i\in[k+1,n+1] \ \text{and} \ a\in \mathfrak A_n \}.
    \end{equation*}
\end{proof}

However, this construction of $\mathfrak{A}_{n+1}$ from $\mathfrak{A}_n$ does not hold when $n$ is odd. For example, consider $321654 \in \mathfrak A_{6}$. Removing $6$, we get $32154$. However, $32154$ is not in $\mathfrak A_5$ since $3$ is at position $1$. So we need a new set to build $\mathfrak A_6$, which we now define.

Let $\tilde{t} = \ceil{\frac{n}{2}}$. Define \begin{equation}\label{def:tildeAn}
\begin{split}
    &\tilde{\mathfrak A}_n = \mathfrak{S}_{1, \dotsc, \tilde{t}}\times \mathfrak{S}_{\tilde{t}+1, \dotsc, n} \\
    = \{&w \in \sn \ | \ w_1\cdots w_{\tilde{t}} \ \text{permutes} \ \{1,\dotsc,\tilde{t}\} \ \text{and} \ w_{\tilde{t}+1} \cdots w_n \ \text{permutes} \ \{\tilde{t}+1,\dotsc,n\} \}.
\end{split}
    \end{equation} 
    Now, let's look at the example $321654$ again. Removing $6$, we get $32154$ which is in $\tilde{\mathfrak A_5}$. Therefore, we can build the set $\mathfrak A_{n+1}$ from $\tilde{\mathfrak A}_n$, by applying the insertion maps. 
    
\begin{obs}\label{obs:ins on tildeAn}
    Fix an odd integer $n=2k+1$ and let $\tilde{\mathfrak A}_n$ be defined as in~\eqref{def:tildeAn}. Then, we have \begin{equation*}
        \mathfrak A_{n+1} = \{\ins {i}(\tilde{a}) \ | \ i\in[k+2,n+1] \ \text{and} \ \tilde{a}\in \tilde{\mathfrak A}_n \}.
    \end{equation*} 
\end{obs}
\begin{proof}
    The proof is similar to the proof of Observation~\ref{obs:ins on An}. 
\end{proof}

Similar to $\mathfrak A_n$, we can build $\mathfrak B_{n+1}$ from $\mathfrak B_{n}$ by applying insertion-swapping maps. 

\begin{obs}\label{obs:ins on Bn}
    Let $\mathfrak B_n$ be defined as in~\eqref{def:Bn}. Then, we have \begin{equation*}
        \mathfrak B_{n+1} =\begin{cases}
            \{\inss{j}(b) \ | \ j\in[1,n+1]\setminus 2\mathbb{Z} \ \text{and} \ b\in \mathfrak B_n \} \quad\text{when $n$ is even,}\\
            \{\inss{j}(b) \ | \ j\in[1,n+1]\cap 2\mathbb{Z} \ \text{and} \ b\in \mathfrak B_n \} \quad \text{when $n$ is odd}. 
        \end{cases}
    \end{equation*}
\end{obs}
\begin{proof}
    We will give a proof for when $n$ is even. 

    Let $n$ be even. Since $j\in [1,n+1] \setminus 2\mathbb{Z}$, $n+1-j$ is even. So $\inss j(b)$ is well-defined. 

    Now, since $n$ is even, and $j$ is odd, by the definition of the insertion-swapping function $\inss j$, $\inss j(b)$ has odd numbers at odd positions and even numbers at even positions. Therefore, \begin{equation*}
        \mathfrak B_{n+1} \supset \{\inss {j}(b) \ | \ j\in[1,n+1]\setminus 2\mathbb{Z} \ \text{and} \ b\in \mathfrak B_n \}.
    \end{equation*}

    On the other hand, let $u = u_1\cdots u_{n+1} \in \mathfrak B_{n+1}$ with $u_i = n+1$. Then, $i$ and $n+1$ are both odd. Then we have, \begin{equation*}
        \Bar{u} = \Bar{u}_1\cdots \Bar{u}_n=u_1\cdots u_{i-1} u_{i+2} u_{i+1} u_{i+4}u_{i+3} \cdots u_{n+1}u_{n} \in \mathfrak B_n.
    \end{equation*} Thus, $\inss {i}(\Bar{u}) = u $, which gives \begin{equation*}
        \mathfrak B_{n+1} \subset \{\inss {j}(b) \ | \ j\in[1,n+1]\setminus 2\mathbb{Z} \ \text{and} \ b\in \mathfrak B_n \}.
    \end{equation*}

For odd $n$, the proof is similar. \end{proof}

Let $w\in \sn$, one can also see that the \emph{inversion} 
\begin{equation}\label{eq:invdef}
\inv (w) \defeq |\{(w_i,w_j) \ | \ i < j \text{ and } w_i > w_j \}|
\end{equation} increases after applying $\ins p $ and $\inss q$ on $w$. Then, we have the following lemmas. 

\begin{lem}\label{l:ins inv}
    Let $w \in \sn$, and fix $p \in [1,n+1]$. Then, we have \begin{equation*}\inv (\ins p (w)) = \inv(w) + n-p+1.\end{equation*} 
\end{lem}
\begin{proof}
    Let $w = w_1\cdots w_n$. Then $\ins p w = w_1\cdots w_{p-1}(n+1)w_{p}\cdots w_n$. Therefore, $(w_i,w_j)$ counts an inversion in $w$ if and only if $(w_i,w_j)$ counts an inversion in $\ins p (w)$. Since $n+1$ is at position $p$ of $\ins p (w)$, $(n+1, w_p), (n+1, w_{p+1}) , \dotsc , (n+1, w_n)$ count inversions in $\ins p (w)$. Hence we have \begin{equation*}
        \inv (\ins p (w)) = \inv(w) + n-p+1.
    \end{equation*}
\end{proof}

\begin{lem}\label{l:inss inv}
    Let $w\in \sn$, and fix $q \in [1,n+1]$ such that $n-q+1$ is even. Then, we have 
    \begin{equation*}
        \inv (\inss{q} (w)) \geq \inv (w) + \frac{n-q+1}{2}.
    \end{equation*}
\end{lem}
\begin{proof}
    By Lemma~\ref{l:ins inv}, inserting $n+1$ at position $q$ creates $n-q+1$ inversions. After swapping the letters after position $q$, we lose at most $\tfrac{n-q+1}{2}$ inversions. Therefore, \begin{equation*}
        \inv (\inss{q} (w)) \geq \inv (w) + n-q+1 - \frac{n-q+1}{2} = \inv (w) + \frac{n-q+1}{2}.
    \end{equation*}\end{proof}

For example, consider $w = 126435$, and $p=q=3$. Then, we have $n=6$, and \begin{gather*}
    \ins p (w) = 2176435, \quad \inss q (w) = 2174653. 
\end{gather*}We see that 
    \begin{gather*}
        \inv (w) = 4, \\ 
        \inv (\ins 3 (w)) = 8 = \inv (w) + n -p +1, \\
        \inv (\inss 3 (w)) = 8 \geq \inv (w) + \frac{n-q+1}{2}. 
\end{gather*}

Observe that $|\mathfrak A_n|=|\tilde{\mathfrak A}_n|=\floor{\tfrac{n}{2}}!\ceil{\tfrac{n}{2}}!$. Then there exists a bijective map between $\mathfrak A_n$ and $\tilde{\mathfrak A}_n$. Let $w =  w_1 \cdots w_n \in \sn$. Define two maps $U$ and $R$ given by
\begin{equation*}
\begin{aligned}
U: \sn &\to \sn,\\
w &\mapsto w^U\defeq (n+1 -w_1)(n+1-w_2)\cdots (n+1-w_n),\\
R: \sn &\to \sn,\\
w &\mapsto w^R\defeq w_n \cdots w_1.\\
\end{aligned}
\end{equation*}

We call the first map $U$ the upside-down of $w$, and the second map $R$ the reverse of $w$. For example if $w = 216435 $, then $w^U = 561342$ and $w^R = 534612$. It is not difficult to see that the two maps are both involutions and the two maps commute. We abbreviate $U\circ R = R \circ U =RU$. This induces bijective maps between $\mathfrak A_n$ and $\tilde{\mathfrak A}_n$, and $\mathfrak B_n$ to itself.

\begin{prop}\label{p: U and R}
The map $U \circ R: \sn \rightarrow \sn$ sending $w \mapsto w^{UR}$ restricts to a bijection between
$\mathfrak A_n$, $\tilde{\mathfrak A}_n$, and to an involution on $\mathfrak B_n$.
\end{prop}

\begin{proof}
    Consider the map $U\circ R \big|_{\mathfrak A_n}$, we claim that its inverse is $U\circ R \big|_{\tilde{\mathfrak A}_n}$. 

    To see this, consider $w\in \mathfrak A_n$, and let $t= \floor{\tfrac{n}{2}}$. Since \begin{equation*}
        \{w_{t+1},\dotsc,w_n\} = \{t+1,\dotsc , n\},
    \end{equation*} we have \begin{equation*}
        \{w^R_1,\dotsc,w^R_t\} =\{t+1,\dotsc,n\}. 
    \end{equation*}Hence, we have \begin{equation*}
        \{w^{RU}_1,\dotsc,w^{RU}_t\} = \{1,\dotsc,t\},. 
    \end{equation*}Thus, we have $w^{RU} \in \tilde{\mathfrak A}_n$. Similarly, if $\tilde{w} \in \tilde{\mathfrak A}_n$, then $\tilde{w}^{RU} \in \mathfrak A_n$. Therefore, $U\circ R$ restricts to a bijection between $\mathfrak A_n$ and $\tilde{\mathfrak A}_n$.

    Next, we show that $U\circ R \big|_{\mathfrak B_n}$ is an involution. 

    Let $u = u_1 \cdots u_n \in \mathfrak B_n$. Then, $u_i$ and $i$ have the same parity for all $1\leq i \leq n$. Applying $U\circ R$ on $u$, we have 
    \begin{equation*}
         \tilde{u} \defeq u^{RU} = (n+1 - u_n) \cdots (n+1 - u_1). 
    \end{equation*}
    Since $u_i$ and $i$ have the same parity for all $i$, \begin{equation}\label{eq:parity of sigma}
        n+1 - u_i = 1 + n-i +i -u_i \equiv 1+n-i \quad \pmod{2} .
    \end{equation}Write $\tilde{u} = \tilde{u}_1 \cdots \tilde{u}_n$, then $\tilde{u}_j = n+1 - u_{n+1-j}$ for all $1\leq j \leq n$. By~\eqref{eq:parity of sigma}, we have \begin{equation*}
        n+1 - u_{n+1-j} \equiv 1+n-(n+1-j) = j \quad \pmod{2}.
    \end{equation*} Therefore, $\tilde{u}_j$ and $j$ have the same parity. Therefore, $\tilde{u} \in \mathfrak B_n$. Now, apply $U \circ R\big|_{B-n}$ on $\tilde{u}$, we have \begin{equation*}
        \tilde{u}^{RU} = (n+1 -(n+1-u_1)) \cdots (n+1 -(n+1-u_n)) = u_1 \cdots u_n = u.
    \end{equation*} Therefore $U \circ R$ restricts to an involution on $\mathfrak B_n$.
    \end{proof}

In addition, the map $U\circ R$ preserves the Bruhat order. 

\begin{prop}\label{upside-down reverse}
    Let $v, \ w \in \sn$ with $v \leq w$ under the Bruhat order. Then, $v^{RU} \leq w^{RU} $ under the Bruhat order. 
\end{prop}

\begin{proof}
    For all $u\in \sn$, $u^{RU} = w_0vw_0$, where $w_0 = n(n-1)\cdots 1$. Also, we have $v \leq w$ under the Bruhat order if and only if 
    \begin{equation*}
    w_0vw_0 \leq w_0 w w_0.
    \end{equation*} (See~\cite[p.~36--38]{BBCoxeter}.)
\end{proof}

One can also see that $U\circ R$ preserves inversions. 

\begin{lem}\label{l:RU inv}
    For all $w\in \sn$, $\inv (w) = \inv (w^{RU})$. 
\end{lem}
\begin{proof}
    Let $w \in \sn$ and $w^{RU} \defeq v$. Then, we have \begin{equation*}
        v = v_1 \cdots v_n = (n+1 - w_n) \cdots (n+1 -w_1). 
    \end{equation*} Let $(v_i,v_j)$ count as an inversion, then $(n+1-w_{n+1-i}, n+1-w_{n+1-j})$ counts as an inversion. Hence, $n+1-w_{n+1-i} > n+1-w_{n+1-j}$, which implies $w_{n+1-j} > w_{n+1-i}$. Since $i<j$, we have $n+1-j < n+1-i$. Then $(w_{n+1-j},w_{n+1-i})$ counts as an inversion. Similarly, assume $(w_i,w_j)$ is an inversion in $w$. Observe that $w=v^{RU}$, then $(v_{n+1-j},v_{n+1-i})$ counts as an inversion in $v = w^{RU}$. Therefore, we have\begin{equation*}
        \inv (v) = \inv (w^{RU}) = \inv (w). 
    \end{equation*}
    \end{proof}

In order to show \eqref{TNNpoly} is TNN, we aim to define a bijective map $f_n: \mathfrak{A}_n\rightarrow \mathfrak{B}_n$ such that $w \leq f_n (w)$ for all $w\in \mathfrak A_n$. Although many such bijections appear in the literature, none is suitable for our purpose. (See~\cite{KFRFPAP}.) Therefore, we define the required bijective maps inductively and illustrate them with a diagram. 

\begin{center}

\tikzset{every picture/.style={line width=0.75pt}} %set default line width to 0.75pt        

\begin{tikzpicture}[x=0.75pt,y=0.75pt,yscale=-1.1,xscale=1.1]

\draw    (300,380) -- (300,342) ;
\draw [shift={(300,340)}, rotate = 90] [color={rgb, 255:red, 0; green, 0; blue, 0 }  ][line width=0.75]    (10.93,-3.29) .. controls (6.95,-1.4) and (3.31,-0.3) .. (0,0) .. controls (3.31,0.3) and (6.95,1.4) .. (10.93,3.29)   ;
\draw    (170,380) -- (170,342) ;
\draw [shift={(170,340)}, rotate = 90] [color={rgb, 255:red, 0; green, 0; blue, 0 }  ][line width=0.75]    (10.93,-3.29) .. controls (6.95,-1.4) and (3.31,-0.3) .. (0,0) .. controls (3.31,0.3) and (6.95,1.4) .. (10.93,3.29)   ;
\draw    (300,460) -- (300,422) ;
\draw [shift={(300,420)}, rotate = 90] [color={rgb, 255:red, 0; green, 0; blue, 0 }  ][line width=0.75]    (10.93,-3.29) .. controls (6.95,-1.4) and (3.31,-0.3) .. (0,0) .. controls (3.31,0.3) and (6.95,1.4) .. (10.93,3.29)   ;
\draw    (170,460) -- (170,422) ;
\draw [shift={(170,420)}, rotate = 90] [color={rgb, 255:red, 0; green, 0; blue, 0 }  ][line width=0.75]    (10.93,-3.29) .. controls (6.95,-1.4) and (3.31,-0.3) .. (0,0) .. controls (3.31,0.3) and (6.95,1.4) .. (10.93,3.29)   ;
\draw    (170,140) -- (170,101) ;
\draw [shift={(170,99)}, rotate = 90] [color={rgb, 255:red, 0; green, 0; blue, 0 }  ][line width=0.75]    (10.93,-3.29) .. controls (6.95,-1.4) and (3.31,-0.3) .. (0,0) .. controls (3.31,0.3) and (6.95,1.4) .. (10.93,3.29)   ;
\draw    (300,140) -- (300,102) ;
\draw [shift={(300,100)}, rotate = 90] [color={rgb, 255:red, 0; green, 0; blue, 0 }  ][line width=0.75]    (10.93,-3.29) .. controls (6.95,-1.4) and (3.31,-0.3) .. (0,0) .. controls (3.31,0.3) and (6.95,1.4) .. (10.93,3.29)   ;
\draw    (200,480) -- (268,480) ;
\draw [shift={(270,480)}, rotate = 180] [color={rgb, 255:red, 0; green, 0; blue, 0 }  ][line width=0.75]    (10.93,-3.29) .. controls (6.95,-1.4) and (3.31,-0.3) .. (0,0) .. controls (3.31,0.3) and (6.95,1.4) .. (10.93,3.29)   ;
\draw    (200,400) -- (268,400) ;
\draw [shift={(270,400)}, rotate = 180] [color={rgb, 255:red, 0; green, 0; blue, 0 }  ][line width=0.75]    (10.93,-3.29) .. controls (6.95,-1.4) and (3.31,-0.3) .. (0,0) .. controls (3.31,0.3) and (6.95,1.4) .. (10.93,3.29)   ;
\draw    (200,159) -- (268,159) ;
\draw [shift={(270,159)}, rotate = 180] [color={rgb, 255:red, 0; green, 0; blue, 0 }  ][line width=0.75]    (10.93,-3.29) .. controls (6.95,-1.4) and (3.31,-0.3) .. (0,0) .. controls (3.31,0.3) and (6.95,1.4) .. (10.93,3.29)   ;
\draw    (300,220) -- (300,182) ;
\draw [shift={(300,180)}, rotate = 90] [color={rgb, 255:red, 0; green, 0; blue, 0 }  ][line width=0.75]    (10.93,-3.29) .. controls (6.95,-1.4) and (3.31,-0.3) .. (0,0) .. controls (3.31,0.3) and (6.95,1.4) .. (10.93,3.29)   ;
\draw    (170,220) -- (170,182) ;
\draw [shift={(170,180)}, rotate = 90] [color={rgb, 255:red, 0; green, 0; blue, 0 }  ][line width=0.75]    (10.93,-3.29) .. controls (6.95,-1.4) and (3.31,-0.3) .. (0,0) .. controls (3.31,0.3) and (6.95,1.4) .. (10.93,3.29)   ;
\draw    (300,300) -- (300,262) ;
\draw [shift={(300,260)}, rotate = 90] [color={rgb, 255:red, 0; green, 0; blue, 0 }  ][line width=0.75]    (10.93,-3.29) .. controls (6.95,-1.4) and (3.31,-0.3) .. (0,0) .. controls (3.31,0.3) and (6.95,1.4) .. (10.93,3.29)   ;
\draw    (170,300) -- (170,262) ;
\draw [shift={(170,260)}, rotate = 90] [color={rgb, 255:red, 0; green, 0; blue, 0 }  ][line width=0.75]    (10.93,-3.29) .. controls (6.95,-1.4) and (3.31,-0.3) .. (0,0) .. controls (3.31,0.3) and (6.95,1.4) .. (10.93,3.29)   ;
\draw    (200,320) -- (268,320) ;
\draw [shift={(270,320)}, rotate = 180] [color={rgb, 255:red, 0; green, 0; blue, 0 }  ][line width=0.75]    (10.93,-3.29) .. controls (6.95,-1.4) and (3.31,-0.3) .. (0,0) .. controls (3.31,0.3) and (6.95,1.4) .. (10.93,3.29)   ;
\draw    (200,240) -- (268,240) ;
\draw [shift={(270,240)}, rotate = 180] [color={rgb, 255:red, 0; green, 0; blue, 0 }  ][line width=0.75]    (10.93,-3.29) .. controls (6.95,-1.4) and (3.31,-0.3) .. (0,0) .. controls (3.31,0.3) and (6.95,1.4) .. (10.93,3.29)   ;

\draw (171.5,479) node    {$\mathfrak A_{4}$};
\draw (300.5,479) node    {$\mathfrak B_{4}$};
\draw (300.5,399) node    {$\mathfrak B_{5}$};
\draw (171.5,399) node    {$\mathfrak A_{5}$};
\draw (174,71) node [anchor=north west][inner sep=0.75pt]  [rotate=-90]  {$\cdots $};
\draw (304,71) node [anchor=north west][inner sep=0.75pt]  [rotate=-90]  {$\cdots $};
\draw (141,437.4) node [anchor=north west][inner sep=0.75pt]  [font=\footnotesize]  {$\ins p$};
\draw (266,437.4) node [anchor=north west][inner sep=0.75pt]  [font=\footnotesize]  {$\inss q$};
\draw (131,359.4) node [anchor=north west][inner sep=0.75pt]  [font=\scriptsize]  {$\mathnormal{{\textstyle U\circ R}}$};
\draw (261,359.4) node [anchor=north west][inner sep=0.75pt]  [font=\scriptsize]  {$\mathnormal{{\textstyle U\circ R}}$};
\draw (224,462.4) node [anchor=north west][inner sep=0.75pt]  [font=\footnotesize]  {$f_{4}$};
\draw (224,382.4) node [anchor=north west][inner sep=0.75pt]  [font=\footnotesize]  {$f_{5}$};
\draw (171.5,319) node    {$\tilde{\mathfrak A}_{5}$};
\draw (300.5,319) node    {$\mathfrak B_{5}$};
\draw (300.5,239) node    {$\mathfrak B_{6}$};
\draw (171.5,239) node    {$\mathfrak A_{6}$};
\draw (140,277.4) node [anchor=north west][inner sep=0.75pt]  [font=\footnotesize]  {$\ins p$};
\draw (266,277.4) node [anchor=north west][inner sep=0.75pt]  [font=\footnotesize]  {$\inss q$};
\draw (131,199.4) node [anchor=north west][inner sep=0.75pt]  [font=\scriptsize]  {$\mathnormal{{\textstyle U\circ R}}$};
\draw (261,199.4) node [anchor=north west][inner sep=0.75pt]  [font=\scriptsize]  {$\mathnormal{{\textstyle U\circ R}}$};
\draw (224,300) node [anchor=north west][inner sep=0.75pt]  [font=\footnotesize]  {$\tilde{f}_{5}$};
\draw (224,222.4) node [anchor=north west][inner sep=0.75pt]  [font=\footnotesize]  {$f_{6}$};
\draw (171.5,159) node    {$\tilde{\mathfrak A}_{6}$};
\draw (300.5,161) node    {$\mathfrak B_{6}$};
\draw (224,140) node [anchor=north west][inner sep=0.75pt]  [font=\footnotesize]  {$\tilde{f}_{6}$};

\end{tikzpicture}

\end{center}

To inductively define $f_n$ and $\tilde{f}_n$, we start with the base case $f_4$ defined as follows. 
\begin{equation}\label{eq:f4}
    1234 \longmapsto 1234 \quad 1243 \longmapsto 1432 \quad 2134 \longmapsto 3214 \quad 2143 \longmapsto 3412.
\end{equation}
Before we define $f_n$ with $n=5,6,\dotsc$, we need an auxiliary map $\tilde{f}_n:\tilde{\mathfrak A}_n \rightarrow \mathfrak B_n$.

Let $\tilde{w} \in \tilde{\mathfrak A}_n$. By Proposition~\ref{p: U and R}, there is a unique $w\in \mathfrak A_n$ such that $w^{RU} = \tilde{w}$. Since $|\mathfrak A_n|=|\mathfrak B_n| = |\tilde{\mathfrak A}_n|=\floor{\tfrac{n}{2}}!\ceil{\tfrac{n}{2}}!$, there exists a bijection between $\mathfrak A_n$ and $\mathfrak B_n$, and a bijection between $\tilde{\mathfrak A}_n$ and $\mathfrak B_n$. Denote $f_n: \mathfrak A_n\rightarrow \mathfrak B_n$ where when $n=4$, $f_4$ is defined as above. Further, define $\tilde{f}_n$ by \begin{equation}\label{eq:def tilde f}
\begin{aligned}
    \tilde{f}_n : \tilde{\mathfrak A}_n &\rightarrow \mathfrak B_n, \\
    \tilde{w} &\mapsto [f_n(\tilde{w}^{RU})]^{RU}. 
    \end{aligned}
\end{equation}

It is not obvious that $\tilde{f}_n$ is surjective, which we now show. 

\begin{obs}\label{o:well-define tilde f}
    Given a bijection $f_n:\mathfrak  A_n \rightarrow \mathfrak  B_n$, the map $\tilde{f}_n: \tilde{\mathfrak A}_n \rightarrow \mathfrak B_n$ is a bijection. 
\end{obs}

\begin{proof}
    By Proposition~\ref{p: U and R}, if $b\in \mathfrak  B_n$, then $b^{RU} \in \mathfrak B_n$. Therefore, the inverse of $\tilde{f}_n$ can be given by \begin{equation*}
    \begin{aligned}
        \tilde{f}_n^{-1}:\mathfrak  B_n &\rightarrow \tilde{\mathfrak A}_n, \\
        b &\mapsto [f_n^{-1}(b^{RU})]^{RU}. 
        \end{aligned}
    \end{equation*}
    One can check that $\tilde{f}_n\circ\tilde{f}_n^{-1}=\tilde{f}_n^{-1}\circ\tilde{f}_n=\text{id}$.
\end{proof}

By our construction of $\tilde{f}_n$, many properties of $f_n$ are inherited in $\tilde{f}_n$.

\begin{lem}\label{lem:induced bijection}
    Let $f_n: \mathfrak A_n \rightarrow \mathfrak B_n$ be a bijection such that $w \leq f_n(w)$ for all $w\in \mathfrak A_n$. Then, the bijection $\tilde{f}_n$~\eqref{eq:def tilde f} satisfies $\tilde{w} \leq \tilde{f}_n(\tilde{w})$ for every $\tilde{w} \in \tilde{\mathfrak A}_n$. 
\end{lem}

\begin{proof}
    Let $w\in \mathfrak A_n$ with $w^{RU} = \tilde{w}$. By Observation~\ref{o:well-define tilde f}, $\tilde{f}_n$ is a bijection. By Proposition~\ref{upside-down reverse}, $\tilde{w} = w^{RU}  \leq f_n(w)^{RU} = \tilde{f}_n(\tilde{w})$.   
\end{proof}

In addition, if $f_n$ increases the number of inversions, so does $\tilde{f}_n$. 

\begin{lem}\label{l:inv tilde f}
    Let $f_n: \mathfrak A_n \rightarrow \mathfrak B_n$ be a bijection such that $\inv(w) \leq \inv(f_n(w))$ for all $w\in \mathfrak A_n$. Then, the bijection $\tilde{f}_n$ satisfies $\inv (\tilde{w}) \leq \inv (\tilde{f}_n(\tilde{w}))$ for every $\tilde{w} \in \tilde{\mathfrak A}_n$. 
\end{lem}
\begin{proof}
    Let $\tilde{w} \in \tilde{\mathfrak A}_n$ be arbitrary. Then, by Proposition~\ref{p: U and R}, $\tilde{w}^{RU} \defeq w \in \mathfrak A_n$. By Lemma~\ref{l:RU inv}, 
     \begin{equation*}
        \inv (\tilde{w}) = \inv (w) \text{ and } \inv(f_n(w)) = \inv([f_n(w)]^{RU}). 
    \end{equation*}Since $\inv(w) \leq \inv(f_n(w))$ for all $w\in \mathfrak A_n$, we have \begin{equation*}
        \inv (\tilde{w}) \leq  \inv([f_n(w)]^{RU}) = \inv([f_n(\tilde{w}^{RU})]^{RU}). 
    \end{equation*}By \eqref{eq:def tilde f}, $[f_n(\tilde{w}^{RU})]^{RU} = \tilde{f}(\tilde{w})$. Hence, $\inv (\tilde{w}) \leq \inv (\tilde{f}(\tilde{w}))$. 
\end{proof}

One can see from the proof that Lemma~\ref{l:inv tilde f} remains correct if we replace all weak inequalities by strict inequalities and restrict $f_{n}$ on non-identity elements of $\mathfrak A_n$, i.e, $\mathfrak A_n \setminus \{e\}$. 
\begin{cor}\label{c:inv tilde f}
    Let $f_n: \mathfrak A_n \rightarrow \mathfrak B_n$ be a bijection such that $\inv(w) < \inv(f_n(w))$ for all $w\in \mathfrak A_n \setminus \{e\}$. Then, the bijection $\tilde{f}_n$ satisfies $\inv (\tilde{w}) < \inv (\tilde{f}_n(\tilde{w}))$ for every $\tilde{w} \in \tilde{\mathfrak A}_n \setminus \{e\}$.
\end{cor}

When $n\leq 3$, it is easy to find a bijection $f_n: \mathfrak A_n \rightarrow \mathfrak B_n$ such that $w\leq f_n(w)$. Also, one can easily
test if a given polynomial is in the $3 \times 3$ immanant space is TNN~\cite{SkanComm}. 
With our base case set to $n=4$, we are now ready to introduce the following recursive algorithm to define $f_n$ for $n=5,6,\dotsc$.

\begin{alg}\label{keyalg} 
Let $\mathfrak A_n$, $\tilde{\mathfrak A}_n$, $\mathfrak B_n$ be defined as above. 

Base Case: Let $f_4$ be the map
\begin{equation}\label{alg base case}
    1234 \longmapsto 1234 \quad 1243 \longmapsto 1432 \quad 2134 \longmapsto 3214 \quad 2143 \longmapsto 3412.
\end{equation}

For $n=5,6,7,\dotsc$ do the following. 
\begin{enumerate}
\item If $n$ is odd, 
\begin{enumerate}
    \item[a.] Write $n-1 = 2k$.
    \item[b.] For each $w\in \mathfrak A_n$, find the unique pair (by Observation~\ref{obs:ins on An}) $(a,p) \in \mathfrak A_{n-1} \times \mathbb N$ such that $w = \ins p (a)$. 
    \item[c.] Define \begin{equation}\label{oddn+1}
    f_n(w) = \inss{2(p-k)-1}(f_{n-1}(a)).
    \end{equation} 
\end{enumerate}
\item If $n$ is even, 
\begin{enumerate}
    \item[a.] Write $n-1 = 2k+1$. 
    \item[b.] Define $\tilde{f}_{n-1} : \tilde{\mathfrak A}_n \rightarrow \mathfrak B_n$ by $f_{n-1}$ as in~\eqref{eq:def tilde f}. 
    \item[c.] For each $u \in \mathfrak A_{n}$, find the unique pair (by Observation~\ref{obs:ins on tildeAn}) $(\tilde{w},q) \in \tilde{\mathfrak A}_n \times \mathbb N$ such that $u=\ins q (\tilde{w})$. 
    \item[d.] Define \begin{equation}\label{evenn+1}
        f_n(u) = \inss{2(q-k-1)}(\tilde{f}_{n-1}(\tilde{w})).
    \end{equation}
\end{enumerate}
\end{enumerate}

\end{alg}

As an example, let us look at the function $f_5$ defined by Algorithm~\ref{keyalg}. Now, we can define $f_5$ from the base case~\eqref{alg base case}. In this case, we have $n=5$ and $k=2$. For example, to find the image of $w=21534 \in \mathfrak A_5$ under $f_5$, we have the unique pair $(a, p)=(2134, 3)$, i.e., $21534=\ins 3(2134)$. By~\eqref{oddn+1}, we have
\begin{equation*}
f_5(21534) = \inss {2(3-2)-1}(f_4(2134)) = \inss 1(3214) = 52341.\end{equation*}
Repeat the process for every element in $\mathfrak A_5$, and we get the map $f_5$ as follows. 
\begin{align*}
    12534 &\longmapsto 52143 & 12354 &\longmapsto 12543 & 12345 &\longmapsto 12345 \\ 
    12543 &\longmapsto 54123 & 12453 &\longmapsto 14523 & 12435 &\longmapsto 14325 \\
    {\color{red}21534} & \longmapsto \color{red} 52341 & 21354 &\longmapsto 32541 & 21345 &\longmapsto 32145 \\
    21543 &\longmapsto 54321 & 21453 &\longmapsto 34521 & 21435 &\longmapsto 34125 .
\end{align*}
To define $f_6$, we need to find $\tilde{f}_5$ from $f_5$. By~\eqref{eq:def tilde f}, we get $\tilde{f}_5$ as follows. 
\begin{align*}
    23145 &\longmapsto 32541 & 21345 &\longmapsto 32145 & 12345 &\longmapsto 12345 \\ 
    32145 &\longmapsto 34521 & 31245 &\longmapsto 34125 & 13245 &\longmapsto 14325 \\
    23154 &\longmapsto 52341 & 21354 &\longmapsto 52143 & 12354 &\longmapsto 12543 \\
    32154 &\longmapsto 54321 & 31254 &\longmapsto 54123 & 13254 &\longmapsto 14523. 
\end{align*}
Now, we can define $f_6$ from $\tilde{f}_5$. In this case, $n=6$ and $k=2$. For example, to find the image of $u=231564 \in \mathfrak A_6$ under $f_6$, we have the unique pair $(\tilde{w}, q)=(23154,5)$, i.e., $231564 = \ins 5(23154)$. From the above definition of $\tilde{f}_5$, we have $\tilde{f}_5(23154) = 52341$. Then, by~\eqref{evenn+1}, we have 
\begin{equation*}
    f_6(231564) = \inss {2(5-2-1)}(\tilde{f}_5(23154)) = \inss {4}(52341) = 523614.
\end{equation*}
Repeat the process for every element in $\mathfrak A_6$, and we get the map $f_6$ as follows. 
\begin{align*}
    231645 &\longmapsto 365214 & 231465 &\longmapsto 325614  & 231456 &\longmapsto 325416 \\ 
    321645 &\longmapsto 365412 & 321465 &\longmapsto 345612 & 321456 &\longmapsto 345216 \\ 
    231654 &\longmapsto 563214 & {\color{red} 231564} &\longmapsto \color{red} 523614 & 231546 &\longmapsto 523416 \\ 
    321654 &\longmapsto 563412 & 321564 &\longmapsto 543612 & 321546 &\longmapsto 543216 \\ 
    213645 &\longmapsto 361254 & 213465 &\longmapsto 321654 & 213456 &\longmapsto 321456 \\
    312645 &\longmapsto 361452 & 312465 &\longmapsto 341652 & 312456 &\longmapsto 341256 \\
    213654 &\longmapsto 561234 & 213564 &\longmapsto 521634 & 213546 &\longmapsto 521436 \\
    312654 &\longmapsto 561432 & 312564 &\longmapsto 541632 & 312546 &\longmapsto 541236 \\
    123645 &\longmapsto 163254 & 123465 &\longmapsto 123654 & 123456 &\longmapsto 123456 \\
    132645 &\longmapsto 163452 & 132465 &\longmapsto 143652 & 132456 &\longmapsto 143256 \\
    123654 &\longmapsto 165234 & 123564 &\longmapsto 125634 & 123546 &\longmapsto 125436 \\
    132654 &\longmapsto 165432 & 132564 &\longmapsto 145632 & 132546 &\longmapsto 145236. 
\end{align*}

To show Algorithm~\ref{keyalg} is well-defined, it suffices to show that $f_n$ is bijective for $n=5,6,7,\dotsc$.

\begin{lem}\label{l:even to odd}
    In Algorithm~\ref{keyalg}, if $n-1=2k$ is even and $f_{n-1}$ is a bijection, then $f_{n}$ is a bijection. 
\end{lem}
\begin{proof}
    Suppose $n-1=2k$ is even and let $f_{n-1}$ be a bijection in Algorithm~\ref{keyalg}. Let $u\in \mathfrak B_n$. Then, by Observation~\ref{obs:ins on Bn}, there is a unique pair $(b,q) \in \mathfrak B_{n-1} \times \mathbb N$ such that $u = \inss{q}(b)$. Since $f_{n-1}$ is a bijection, we can find $f_{n-1}^{-1}(b) \in \mathfrak A_{n-1}$. 
    We define $g_n: \mathfrak B_n \rightarrow \mathfrak A_n$ by 
    \begin{equation}\label{odd inverse f}
        g_n(u) = \ins{k+\frac{q+1}{2}}(f_{n-1}^{-1}(b)). 
    \end{equation}
    Next, we claim that $g_n$ is the inverse function of $f_n$. 

    To see this, let $w \in \mathfrak A_n$ be arbitrary. By~\eqref{oddn+1}, \begin{equation*}
        f_n(w) = \inss{2(p-k)-1} (f_{n-1}(a)), 
    \end{equation*}where $(a,p) \in \mathfrak A_{n-1} \times \mathbb N$ with $\ins{p} (a) = w$. To find $g_n(f_n(w))$, we need to find the unique pair $(b,q)$ in~\eqref{odd inverse f}. We see that $(b,q) = (f_{n-1}(a) , 2(p-k)-1)$. Hence, we have \begin{equation*}
        g_n(f_n(w)) = \ins{k+\frac{2(p-k)-1 +1}{2}}(f_{n-1}^{-1} (f_{n-1}(a))) = \ins{p}(a) = w. 
    \end{equation*} 
    On the other hand, let $u\in \mathfrak B_n$ be arbitrary, then \begin{equation*}
        g_n(u) = \ins{k+\frac{q+1}{2}}(f_{n-1}^{-1}(b)), 
    \end{equation*}where $(b,q) \in \mathfrak B_{n-1} \times \mathbb N$ with $u = \inss{q}(b)$. To find $f_n(g_n(u))$, we need to find the unique pair $(a,p)$ in~\eqref{oddn+1}. We see that $(a,p) = (f_{n-1}^{-1}(b), k+ \frac{q+1}{2})$. Hence by~\eqref{oddn+1}, we have \begin{equation*}
        f_n(g_n(u)) = \inss{2(k+\frac{q+1}{2} -k)-1}(f_{n-1}(f_{n-1}^{-1}(b))) = \inss{q} (b) = u. 
    \end{equation*}Therefore, $g_n$ is the inverse function of $f_n$. 
\end{proof}

Continuing the example from earlier, now to find the image of $u=52341 \in \mathfrak B_5$ under $f_5^{-1}$, we have the unique pair $(b,q) = (3214,1)$, i.e., $52341 = \inss 1 (3214)$. By~\eqref{odd inverse f}, we have \begin{equation*}
    f_5^{-1}(52341) = \ins{2+\frac{1+1}{2}} f_{4}^{-1} (3214) = \ins 3 (2134) = 21534. 
\end{equation*}

\begin{lem}\label{l:odd to even}
    In Algorithm~\ref{keyalg}, if $n-1=2k+1$ is odd and $f_{n-1}$ is a bijection, then $f_{n}$ is a bijection. 
\end{lem}
\begin{proof}
Suppose $n-1=2k+1$ is odd and let $f_{n-1}$ be a bijection in Algorithm~\ref{keyalg}. By Observation~\ref{o:well-define tilde f}, $\tilde{f}_{n-1}: \tilde{\mathfrak A}_{n-1} \rightarrow \mathfrak B_{n-1}$ is a bijection. Let $u \in \mathfrak B_n$. Then, by Observation~\ref{obs:ins on Bn}, there is a unique pair $(b,p) \in \mathfrak B_{n-1} \times \mathbb N$ such that $u = \inss p (b)$. Since $\tilde{f}_{n-1}$ is a bijection, we can find $\tilde{f}_{n-1}^{-1}(b) \in \tilde{\mathfrak A}_{n-1}$. We define $g_n: \mathfrak B_n \rightarrow \mathfrak A_n$ by \begin{equation}\label{even inverse f}
    g_n(u) = \ins{k+1+\frac{p}{2}}(\tilde{f}_{n-1}^{-1}(b)).
\end{equation}Next, we claim that $g_n$ is the inverse function of $f_n$. 

To see this, let $w\in \mathfrak A_n$ be arbitrary. By~\eqref{evenn+1}, \begin{equation*}
    f_n(w) = \inss{2(q-k-1)}(\tilde{f}_{n-1}(\tilde{w})), 
\end{equation*}where $(\tilde{w},q)\in \tilde{\mathfrak A}_n \times \mathbb N$ with $w = \ins{q} (\tilde{w})$. To find $g_n(f_n(w))$, we need to find the unique pair $(b,p)$ in~\eqref{even inverse f}. We see that $(b,p) =(\tilde{f}_{n-1}(\tilde{w}),2(q-k-1))$. Hence by~\eqref{even inverse f}, we have 
\begin{equation*}
    g_n(f_n(w)) = \ins{k+1+\frac{2(q-k-1)}{2}}(\tilde{f}_{n-1}^{-1}(\tilde{f}_{n-1}(\tilde{w}))) = \ins{q} (\tilde{w}) = w. 
\end{equation*}
On the other hand, let $u \in \mathfrak B_n$ be arbitrary, then \begin{equation*}
    g_n(u) = \ins{k+1+\frac{p}{2}}(\tilde{f}_{n-1}^{-1}(b)),
\end{equation*}where $(b,p) \in \mathfrak B_{n-1} \times \mathbb N$ with $u=\inss p (b)$. To find $f_n(g_n(u))$, we need to find the unique pair $(\tilde{w},q)$ in~\eqref{evenn+1}. We see that $(\tilde{w},q) = (\tilde{f}_{n-1}^{-1}(b),k+1+\frac{p}{2})$. Hence, by~\eqref{evenn+1}, we have 
\begin{equation*}
    f_n(g_n(u)) = \inss{2(k+1+\frac{p}{2} -k -1)}(\tilde{f}_{n-1}(\tilde{f}_{n-1}^{-1}(b))) = \inss{p}(b) = u. 
\end{equation*}
Therefore, $g_n$ is the inverse function of $f_n$. 
\end{proof}

For example, to find the image of $u=523614 \in \mathfrak B_6$ under $f_6^{-1}$, we have the unique pair $(b,p) = (52341,4)$, i.e., $523614 = \inss{4}(52341)$. By~\eqref{even inverse f}, we have 
\begin{equation*}
    f_6^{-1}(523614) = \ins{2+1+\frac{4}{2}} (\tilde{f}_5^{-1}(52341)) = \ins{5}(23154) = 231564.
\end{equation*}

Then, we have the immediate corollary. 

\begin{cor}
    Algorithm~\ref{keyalg} is well-defined. 
\end{cor}
\begin{proof}
    As the base case, one can observe that $f_4$ is a bijection. Then inductively from Lemma~\ref{l:even to odd} and Lemma~\ref{l:odd to even}, we have the desired result. 
\end{proof}

We aim to build a bijection $\mathfrak A_n \rightarrow \mathfrak B_n$ such that $w \leq f_n(w)$ for all $w\in \mathfrak A_n$. It's easy to see that the identity element $e=123\cdots n \in \sn$ is in both $\mathfrak A_n$ and $\mathfrak B_n$. Therefore, the identity element has to be sent to the identity under $f_n$, and we have the next observation. 

\begin{obs}\label{o:identity to identity}
    The bijection $f_n$ defined in Algorithm~\ref{keyalg} sends identity to identity. 
\end{obs}
\begin{proof}
    When $n=4$, we see that $1234 \longmapsto 1234$. We assume that $f_{n-1}$ sends identity to identity, and consider $f_n$. Let $e_n \defeq 123\cdots n$ be the identity element in $\sn$. If $n$ is odd, then $e_n = \ins{n} (e_{n-1})$. Write $n-1 = 2k$, by~\eqref{oddn+1}, we have
    \begin{equation*}
        f_{n}(e_n) = \inss{2(n-k)-1} (f_{n-1}(e_{n-1})) = \inss{n} (e_{n-1}) = e_n. 
    \end{equation*}
    If $n$ is even, then $e_n = \ins n (e_{n-1})$. Since $f_{n-1}$ sends identity to identity, by~\eqref{eq:def tilde f} the bijection $\tilde{f}_{n-1}$ sends identity to identity. Write $n-1 = 2k+1$, by~\eqref{evenn+1}, we have 
    \begin{equation}
        f_n(e_n) = \inss{2(n-k-1)}(\tilde{f}_{n-1}(e_{n-1})) = \inss{n}(e_{n-1}) = e_n.
    \end{equation}
    Therefore, $f_n$ sends identity to identity. 
\end{proof}

By the definition of $\tilde{f}_n$ in~\eqref{eq:def tilde f}, we have the immediate corollary. 

\begin{cor}\label{c:identity to identity}
    The bijection $\tilde{f}_n$, where $f_n$ is defined in Algorithm~\ref{keyalg}, sends identity to identity. 
\end{cor}

It is well-known that the Bruhat lattice is graded by the number of inversions~\eqref{eq:invdef}. We have seen that $f_n$ sends the identity element to the identity element. For $w \in\mathfrak A_n \setminus \{e\}$, $f_n(w)$ is located higher than $w$ in the Bruhat lattice of $\sn$.

\begin{thm}\label{t:invineq}
    In Algorithm \ref{keyalg}, for all $w\in \mathfrak A_n \setminus \{e\}$, we have $\inv (w) < \inv (f_n(w))$.
\end{thm}
\begin{proof}

    Let $w \in \mathfrak A_n$ not equal to identity. One can check that $f_4$ satisfies $\inv (w) < \inv (f_4(w))$ for all non-identity elements in $\mathfrak A_4$.
    
    Now, we are going to use induction. We assume that $\inv (w) < \inv (f_{n-1}(w))$ for all non-identity $v \in \mathfrak A_{n-1}$ as our induction hypothesis, and consider the map $f_{n}$. 

    Case 1: Fix $n-1 = 2k$ and let $v \in \mathfrak A_n$ not equal to identity. There exists a unique pair $(a,p) \in \mathfrak A_{n-1} \times \mathbb N$ such that $v = \ins{p} (a)$. By~\eqref{oddn+1}, $f_n(v)=\inss{2(p-k)-1}(f_{n-1}(a))$. By Lemma~\ref{l:inss inv}, we have $\inv (f_n(v)) = \inv (\inss{2(p-k)-1} (f_{n-1}(a)))$ is at least 
    \begin{equation*}
         \inv (f_{n-1}(a)) + \frac{n-1 -(2(p-k-1)-1)+1}{2} = \inv (f_{n-1}(a)) + n-p,   
    \end{equation*} with strict inequality when $f_{n-1}(a)$ equals identity, i.e., $a$ equals identity by Observation~\ref{o:identity to identity}.  
    If $a$ is not identity, by our induction hypothesis, we have \begin{equation*}
        \inv (f_n(v)) \geq \inv(f_{n-1}(a)) +n -p > \inv (a) +n-1-p+1. 
    \end{equation*} If $a$ is identity, we have \begin{equation*}
        \inv (f_n(v)) > \inv(f_{n-1}(a)) +n -p = \inv (a) +n-1-p+1. 
        \end{equation*}
    By Lemma~\ref{l:ins inv}, we see that \begin{equation*}
        \inv (a) +n-1-p+1 = \inv(\ins{p}(a)) = \inv (v). 
    \end{equation*}Therefore, $\inv (f_n(v)) > \inv (v)$.

    Case 2: Fix $n-1 = 2k+1$ and let $u \in \mathfrak A_n$ not equal to identity. There exists a unique pair $(\tilde{u},q) \in \tilde{\mathfrak A}_{n-1} \times \mathbb N$ such that $u = \ins q (\tilde{u})$. By~\eqref{evenn+1}, $f_n(u) = \inss{2(q-k-1)}(\tilde{f}_{n-1}(\tilde{u}))$. By Lemma~\ref{l:inss inv}, $\inv (f_n(u))$ is at least 
    \begin{equation*}
        \inv (\tilde{f}_{n-1}(\tilde{u}))+ \frac{n-1 - 2(q-k-1) +1}{2} = \inv (\tilde{f}_{n-1}(\tilde{u})) +n-q, 
    \end{equation*}with strict inequality when $\tilde{f}_{n-1}(\tilde{u})$ equals identity, i.e.,$\tilde{u}$ equals identity by Corollary~\ref{c:identity to identity}. If $\tilde{u}$ is not identity, by Corollary~\ref{c:inv tilde f} and our induction hypothesis, we have
    \begin{equation*}
        \inv (\tilde{f}_{n-1} (\tilde{u})) > \inv (\tilde{u}). 
    \end{equation*}Hence, we have \begin{equation*}
        \inv (f_n(u)) \geq \inv (\tilde{f}_{n-1}(\tilde{u})) +n-q > \inv (\tilde{u}) +n-q.
    \end{equation*}
    If $\tilde{u}$ is identity, we have 
    \begin{equation*}
        \inv (f_n(u)) > \inv (\tilde{f}_{n-1}(\tilde{u})) +n-q = \inv (\tilde{u}) +n-1-q+1.
    \end{equation*} 
    By Lemma~\ref{l:ins inv}, we have \begin{equation*}
         \inv (\tilde{u}) +n-1-q+1 = \inv (\ins q (\tilde{u})) = \inv (u). 
    \end{equation*}Therefore, $\inv (f_n(u)) > \inv (u)$. 
\end{proof}

We know that $\inv (v) < \inv (w)$ is a necessary condition of $v < w$ in the Bruhat order. Hence,we have the immediate corollary. 

\begin{cor}
    In Algorithm~\ref{keyalg}, for all $w\in \mathfrak A_n \setminus\{e\}$, we have $w \ngeq f_n(w)$ under the Bruhat order. 
\end{cor}

The next theorem is due to Ehresmann \cite{Ehresmann}, which is a way to compare two elements in $\sn$ under the Bruhat order. It is also called the tableau criterion. 

\begin{thm}[Tableau Criterion]\label{TabCriterion}
    Let $w,u \in \sn$. Define $w_{i,k}$ to be the $i$-th element in the increasing rearrangement of $w_1,w_2,\dotsc, w_k$ and similarly define $u_{i,k}$. Then $w \leq u$ under the Bruhat order if and only if $w_{i,k} \leq u_{i,k}$ as an integer for all $k\in [1,n]$ and $1\leq i \leq k$. 
\end{thm}

For example, we can use Theorem \ref{TabCriterion} to compare $2143 $ and $3241$ in the Bruhat order. When $k=1,2,3,4$, we have the following rearrangements of $2143$ and $3241$. 
\begin{center}
\begin{tabularx}{0.8\textwidth} { 
  | >{\centering\arraybackslash}X 
  | >{\centering\arraybackslash}X 
  | >{\centering\arraybackslash}X | }
 \hline
  & $2143$ & $3241$ \\
 \hline
 $k=1$  & $2$  & $2$  \\ 
\hline
$k=2$  & $1,2$  & $2,3$  \\ 
\hline
$k=3$  & $1,2,4$  & $2,3,4$  \\ 
\hline
$k=4$  & $1,2,3,4$  & $1,2,3,4$  \\ 
\hline
 
\end{tabularx}
\end{center}
One can see that for each $k$, the $i$-th letter of the rearrangement of $2143$ is less than or equal to the $i$-th letter of the rearrangement of $3241$. Therefore, $2143 \leq 3241$ in the Bruhat order.

By the Tableau Criterion, we can compare every $w\in \mathfrak A_n$ and its image $f_n(w) \in\mathfrak B_n$ under Algorithm~\ref{keyalg}. Then, we have the following theorem.  

\begin{thm}\label{t:special n}
    In Algorithm~\ref{keyalg}, for $n\leq 13$, we have $w\leq f_n(w)$ under the Bruhat order for all $w\in \mathfrak A_n$. 
\end{thm}

\begin{proof}
    This is done by brute force, with MATLAB code in the appendix. 
\end{proof}

We are going to extend Theorem~\ref{t:special n} to all $n$ in the next section.

\section{Bruhat Order Comparison for the Bijection $f_n$}\label{s:proof of balance case}

In this section, we are going to prove that $w \leq f_n (w)$ under the Bruhat order for all $w \in \mathfrak A_n$, where $f_n$ is defined in Algorithm~\ref{keyalg}. 

Now, we consider the case when $n=2k$. In this case, both families of
permutations under consideration can be naturally parametrized by pairs
of permutations in $\mathfrak S_k$. More precisely, an element of the Young
subgroup preserving the two blocks
\begin{equation*}
[1,k]
\qquad\text{and}\qquad
[k+1,\ldots,2k]
\end{equation*}
is determined by the order of the entries within each block. Thus, for
$\sigma,\tau\in \mathfrak S_k$, we write
\begin{equation}
A_k(\sigma,\tau) = \sigma_1\cdots \sigma_k\,(k+\tau_1)\cdots(k+\tau_k) \in \mathfrak A_{2k}.
\end{equation}

Similarly, a parity-alternating permutation is determined by the
relative order of its odd entries and its even entries. Accordingly, for
$\alpha,\beta\in \mathfrak S_k$, we write
\begin{equation}
B_k(\alpha,\beta) = (2\alpha_1-1)(2\beta_1)\cdots(2\alpha_k-1)(2\beta_k) \in \mathfrak B_{2k}.
\end{equation}
These parametrizations will allow us to describe the map between the two
families in terms of permutations in $\mathfrak S_k$. 

The parametrizations $A_k$ and $B_k$ also behave naturally under the
$RU$-operation. In both cases, applying $RU$ reverses the roles of the
two component permutations and applies $RU$ to each of them. We record
these identities in the following lemma.

\begin{lem}\label{l:RU on A B}
Let $\sigma,\tau,\alpha,\beta\in\mathfrak S_k$. Then we have
\begin{equation}
A_k(\sigma,\tau)^{RU}
=
A_k(\tau^{RU},\sigma^{RU}),
\label{eq:RU on A}
\end{equation}
and
\begin{equation}
B_k(\alpha,\beta)^{RU}
=
B_k(\beta^{RU},\alpha^{RU}).
\label{eq:RU on B}
\end{equation}
\end{lem}

\begin{proof}
By definition,
\begin{equation*}
A_k(\sigma,\tau)
=
\sigma_1\cdots\sigma_k
(k+\tau_1)\cdots(k+\tau_k).
\end{equation*}
Applying $RU$, we first reverse the order of the entries and then
replace each entry $x$ by $2k+1-x$. Thus,
\begin{equation*}
A_k(\sigma,\tau)^{RU}
={}
(k+1-\tau_k)\cdots(k+1-\tau_1)
(2k+1-\sigma_k)\cdots(2k+1-\sigma_1).
\end{equation*}
Since
$\tau^{RU}
=
(k+1-\tau_k)\cdots(k+1-\tau_1)$
and
$\sigma^{RU}
=
(k+1-\sigma_k)\cdots(k+1-\sigma_1)$,
we have
\begin{equation*}
2k+1-\sigma_i
=
k+(k+1-\sigma_i).
\end{equation*}
Therefore,
\begin{equation*}
A_k(\sigma,\tau)^{RU}
=
A_k(\tau^{RU},\sigma^{RU}),
\end{equation*}
which proves \eqref{eq:RU on A}.

Similarly,
$B_k(\alpha,\beta)
=
(2\alpha_1-1)(2\beta_1)\cdots
(2\alpha_k-1)(2\beta_k)$.
Applying $RU$ gives
\begin{equation*}
B_k(\alpha,\beta)^{RU}
={}
\bigl(2(k+1-\beta_k)-1\bigr)
\bigl(2(k+1-\alpha_k)\bigr)\cdots
\bigl(2(k+1-\beta_1)-1\bigr)
\bigl(2(k+1-\alpha_1)\bigr).
\end{equation*}
Since
$\beta^{RU}
=
(k+1-\beta_k)\cdots(k+1-\beta_1)$
and
$\alpha^{RU}
=
(k+1-\alpha_k)\cdots(k+1-\alpha_1)$,
we have 
\begin{equation*}
B_k(\alpha,\beta)^{RU} = B_k(\beta^{RU},\alpha^{RU}).
\end{equation*}
\end{proof}

For example, let $k=3$, $\sigma=123$, and $\tau=231$. Then
\begin{equation*}
A_3(123,231)=123564.
\end{equation*}
Applying $RU$, we obtain
\begin{equation*}
123564^{RU}=312456.
\end{equation*}
On the other hand,
\begin{equation*}
\tau^{RU}=312,
\qquad
\sigma^{RU}=123,
\end{equation*}
and hence
\begin{equation*}
A_3(\tau^{RU},\sigma^{RU})
=
A_3(312,123)
=
312456.
\end{equation*}

Similarly, let $\alpha=231$ and $\beta=132$. Then
\begin{equation*}
B_3(231,132)=325614,
\end{equation*}
and
\begin{equation*}
325614^{RU}=361254.
\end{equation*}
Since
\begin{equation*}
\beta^{RU}=213,
\qquad
\alpha^{RU}=312,
\end{equation*}
we also have
\begin{equation*}
B_3(\beta^{RU},\alpha^{RU})
=
B_3(213,312)
=
361254.
\end{equation*}
Notice that in both identities above, the two component permutations
are interchanged after applying $RU$.

For $d\in[k+1]$, define a minimum-insertion map
\begin{equation*}
\begin{aligned}
I_d^-:\mathfrak{S}_k &\to \mathfrak{S}_{k+1},\\
w=w_1\cdots w_k
&\mapsto
(w_1+1)\cdots(w_{d-1}+1)\,1\,
(w_d+1)\cdots(w_k+1).
\end{aligned}
\end{equation*}
Thus, $I_d^-$ inserts the new minimum $1$ in position $d$, after
increasing each original entry by $1$.

For convenience, we also write
\begin{equation*}
I_h^+ := \ins {h},
\end{equation*}
where $\operatorname{ins}_h$ is the insertion map defined in Section~\ref{s:A gen on all}.

The insertion maps interact with the $RU$-operation in a simple way.
We record the following two identities for later use.

\begin{lem}\label{l:ins RU identity}
For $w\in\mathfrak{S}_k$, and $d,h\in[k+1]$, we have 
\begin{gather*}
\bigl(I_d^-(w)\bigr)^{RU}
=
I_{k+2-d}^+\bigl(w^{RU}\bigr),\ \text{and} \ \bigl(I_h^+(w)\bigr)^{RU}
=
I_{k+2-h}^-\bigl(w^{RU}\bigr).
\end{gather*}
\end{lem}

\begin{proof}
The first identity follows directly from the definition of $RU$: reversing the
positions sends the entry inserted in position $d$ to position
$k+2-d$, while complementing the entries sends the new minimum $1$
to the new maximum $k+1$. The proof of the second identity is similar. 
\end{proof}

Using the parametrizations $A_k$ and $B_k$ above, we may equivalently
encode the map $f_{2k}$ defined in Algorithm~\ref{keyalg} as a map
\begin{equation*}
F_k:S_k\times S_k\longrightarrow S_k\times S_k.
\end{equation*}
More precisely, for $\sigma,\tau,\alpha,\beta\in S_k$, we define
\begin{equation}\label{def:F_k}
F_k(\sigma,\tau)=(\alpha,\beta)
\quad\Longleftrightarrow\quad
f_{2k}\bigl(A_k(\sigma,\tau)\bigr)=B_k(\alpha,\beta).
\end{equation}
Thus, $F_k$ records the action of $f_{2k}$ in terms of the two
$S_k$-coordinates determined by the parametrizations $A_k$ and $B_k$.

For example, let $k=3$ and take $\sigma=\tau=231$. Then
\begin{equation*}
A_3(231,231)=231564.
\end{equation*}
From the definition of $f_6$, we have
\begin{equation*}
f_6(231564)=523614.
\end{equation*}
On the other hand,
\begin{equation*}
523614=B_3(321,132).
\end{equation*}
Therefore,
\begin{equation*}
F_3(231,231)=(321,132).
\end{equation*}
Indeed, the odd entries of $523614$ occur in the order $5,3,1$,
corresponding to $321$, while the even entries occur in the order
$2,6,4$, corresponding to $132$. 

Although our main focus is the case $n=2k$, the recursive definition
of $f_{2k}$ naturally involves permutations in $\mathfrak{S}_{2k-1}$.
We therefore introduce a corresponding parametrization for
parity-alternating permutations of odd size. For
$\alpha\in \mathfrak{S}_k$ and $\beta\in \mathfrak{S}_{k-1}$, define
\begin{equation}\label{def:B circ}
B_k^\circ(\alpha,\beta)
=
(2\alpha_1-1)(2\beta_1)
\cdots
(2\alpha_{k-1}-1)(2\beta_{k-1})
(2\alpha_k-1).
\end{equation}
Thus, $B_k^\circ(\alpha,\beta)\in \mathfrak B_{2k-1}$ is the
parity-alternating permutation whose odd entries are ordered according
to $\alpha$ and whose even entries are ordered according to $\beta$. 

There are some properties involving $B_k^{\circ}$ and maps we defined earlier.

\begin{lem}\label{l:B circ with inss}
The insertion-swapping maps satisfy the following identities.

\begin{enumerate}[(i)]
    \item Let $w,v\in\mathfrak S_{k-1}$ and $s\in[k]$. Then
    \begin{equation*}
        \inss{2s-1}\bigl(B_{k-1}(w,v)\bigr)
        =
        B_k^\circ\bigl(I_s^+(w),v\bigr).
    \end{equation*}

    \item Let $w\in\mathfrak S_k$, $v\in\mathfrak S_{k-1}$, and
    $h\in[k]$. Then
    \begin{equation*}
        \inss{2h}\bigl(B_k^\circ(w,v)\bigr)
        =
        B_k\bigl(w,I_h^+(v)\bigr).
    \end{equation*}
\end{enumerate}
\end{lem}

\begin{proof}
For part~(i), we have
\begin{equation*}
B_{k-1}(w,v)
=
(2w_1-1)(2v_1)\cdots
(2w_{s-1}-1)(2v_{s-1})
(2w_s-1)(2v_s)\cdots
(2w_{k-1}-1)(2v_{k-1}).
\end{equation*}
By the definition of $\inss{2s-1}$,
\begin{align*}
\inss{2s-1}\bigl(B_{k-1}(w,v)\bigr)
={}&
(2w_1-1)(2v_1)\cdots
(2w_{s-1}-1)(2v_{s-1})(2k-1)\\
&\quad
(2v_s)(2w_s-1)
(2v_{s+1})(2w_{s+1}-1)\cdots
(2v_{k-1})(2w_{k-1}-1).
\end{align*}
Since
\begin{equation*}
I_s^+(w)
=
w_1\cdots w_{s-1}\,k\,w_s\cdots w_{k-1},
\end{equation*}
the right-hand side is precisely
\begin{equation*}
B_k^\circ\bigl(I_s^+(w),v\bigr).
\end{equation*}

For part~(ii), we have
\begin{equation*}
B_k^\circ(w,v)
=
(2w_1-1)(2v_1)\cdots
(2w_{h-1}-1)(2v_{h-1})
(2w_h-1)(2v_h)\cdots
(2w_k-1).
\end{equation*}
Applying $\inss{2h}$ gives
\begin{align*}
\inss{2h}\bigl(B_k^\circ(w,v)\bigr)
={}&
(2w_1-1)(2v_1)\cdots
(2w_{h-1}-1)(2v_{h-1})(2w_h-1)(2k)\\
&\quad
(2w_{h+1}-1)(2v_h)\cdots
(2w_k-1)(2v_{k-1}).
\end{align*}
Since
\begin{equation*}
I_h^+(v)
=
v_1\cdots v_{h-1}\,k\,v_h\cdots v_{k-1},
\end{equation*}
this is exactly
\begin{equation*}
B_k\bigl(w,I_h^+(v)\bigr).
\end{equation*}
\end{proof}

For example, let $k=3$, $w=21$, and $v=12$. For part~(i), take
$s=1$. Then
\begin{equation*}
B_2(21,12)=3214
\end{equation*}
and
\begin{equation*}
I_1^+(21)=321.
\end{equation*}
Hence
\begin{equation*}
\inss{1}\bigl(B_2(21,12)\bigr)
=
\inss{1}(3214)
=
52341
=
B_3^\circ(321,12).
\end{equation*}

For part~(ii), take $w=321$, $v=12$, and $h=2$. We have
\begin{equation*}
B_3^\circ(321,12)=52341
\end{equation*}
and
\begin{equation*}
I_2^+(12)=132.
\end{equation*}
Therefore,
\begin{equation*}
\inss{4}\bigl(B_3^\circ(321,12)\bigr)
=
\inss{4}(52341)
=
523614
=
B_3(321,132).
\end{equation*}

\begin{lem}\label{l:B circ with RU}
Let $w\in\mathfrak{S}_k$ and $v\in\mathfrak{S}_{k-1}$. Then
\begin{equation*}
\bigl(B_k^\circ(w,v)\bigr)^{RU}
=
B_k^\circ\bigl(w^{RU},v^{RU}\bigr).
\end{equation*}
\end{lem}

\begin{proof}

By definition, we have
\begin{align*}
B_k^\circ(w,v)
&=
(2w_1-1)(2v_1)\cdots
(2w_{k-1}-1)(2v_{k-1})(2w_k-1), \\
w^{RU}
&=
(k+1-w_k)\cdots(k+1-w_1), \\
v^{RU}
&=
(k-v_{k-1})\cdots(k-v_1).
\end{align*}

Applying $RU$ on $\bigl(B_k^\circ(w,v)\bigr)$, we have 
\begin{align*}
\bigl(B_k^\circ(w,v)\bigr)^{RU}
={}&
\bigl(2(k+1-w_k)-1\bigr)
\bigl(2(k-v_{k-1})\bigr)\cdots\\
&\qquad
\bigl(2(k+1-w_2)-1\bigr)
\bigl(2(k-v_1)\bigr)
\bigl(2(k+1-w_1)-1\bigr),
\end{align*}
which equals 
$B_k^\circ\bigl(w^{RU},v^{RU}\bigr)$ by definition of $B_k^\circ$.

\end{proof}

For example, let $k=3$, $w=231$, and $v=21$. Then
\begin{equation*}
B_3^\circ(231,21)=34521.
\end{equation*}
Applying $RU$ gives
\begin{equation*}
34521^{RU}=54123.
\end{equation*}
On the other hand,
\begin{equation*}
231^{RU}=312,
\qquad
21^{RU}=21,
\end{equation*}
and hence
\begin{equation*}
B_3^\circ(231^{RU},21^{RU})
=
B_3^\circ(312,21)
=
54123.
\end{equation*}
Thus,
\begin{equation*}
\bigl(B_3^\circ(231,21)\bigr)^{RU}
=
B_3^\circ(231^{RU},21^{RU}),
\end{equation*}
as stated in the lemma. In contrast to \eqref{eq:RU on B}, the two components are not
interchanged here. 

To analyze the maps $F_k$ recursively, we first describe how they
interact with the two insertion operations introduced above. This
recursive relation is naturally coupled with a symmetry property of
$F_k$, and the two statements will be proved simultaneously by
induction. 

\begin{prop}\label{p:Ek Dk}
The maps $F_k$ satisfy the following two properties.

\emph{$(E_k)$} Let $k\geq 3$, and suppose that
\begin{equation*}
F_{k-1}(\sigma,\tau)=(\alpha,\beta),
\end{equation*}
where $\sigma,\tau,\alpha,\beta\in S_{k-1}$. Then, for $d,h\in[k]$,
\begin{equation*}
F_k\bigl(I_d^-(\sigma),I_h^+(\tau)\bigr)
=
\bigl(I_d^-(\beta),I_h^+(\alpha)\bigr).
\end{equation*}

\emph{$(D_k)$} Let $k\geq 2$, and suppose that
\begin{equation*}
F_k(\sigma,\tau)=(\alpha,\beta).
\end{equation*}
Then
\begin{equation*}
F_k\bigl(\tau^{RU},\sigma^{RU}\bigr)
=
\bigl(\beta^{RU},\alpha^{RU}\bigr).
\end{equation*}
\end{prop}

\begin{proof}
We prove $(E_k)$ and $(D_k)$ by simultaneous induction. The induction proceeds according to the following scheme:
\begin{equation*}
D_2
\;\Longrightarrow\;
E_3
\;\Longrightarrow\;
D_3
\;\Longrightarrow\;
E_4
\;\Longrightarrow\;
D_4
\;\Longrightarrow\;
\cdots .
\end{equation*}
More precisely,
we first verify $(D_2)$. Then, for each $k\geq 3$, assuming
$(D_{k-1})$, we first prove $(E_k)$ and then use $(E_k)$ together with
$(D_{k-1})$ to prove $(D_k)$.

It follows directly from the definition of $f_4$ that $F_2$ is the
identity map on $\mathfrak{S}_2\times\mathfrak{S}_2$. Hence $(D_2)$
holds.

We now assume $(D_{k-1})$ and prove $(E_k)$. Suppose that
\begin{equation*}
F_{k-1}(\sigma,\tau)=(\alpha,\beta),
\end{equation*}
and let $d,h\in[k]$. Set
\begin{equation*}
W=A_k\bigl(I_d^-(\sigma),I_h^+(\tau)\bigr).
\end{equation*}
The entry $2k$ occurs in position $k+h$ of $W$. Removing this entry
gives a permutation $\widetilde W\in\widetilde{\mathfrak A}_{2k-1}$
such that
$W=\ins{k+h}(\widetilde W).$
Therefore, by the definition of $f_{2k}$,
\begin{equation*}
f_{2k}(W)
=
\inss{2(k+h-(k-1)-1)}\bigl(\tilde f_{2k-1}(\widetilde W)\bigr) = \inss{2h}\bigl(\tilde f_{2k-1}(\widetilde W)\bigr).
\end{equation*}

Since $W=A_k\bigl(I_d^-(\sigma),I_h^+(\tau)\bigr)$, removing the entry $2k$ gives
\begin{equation*}
\widetilde W
=
(\sigma_1+1)\cdots(\sigma_{d-1}+1)\,1\,
(\sigma_d+1)\cdots(\sigma_{k-1}+1)
(k+\tau_1)\cdots(k+\tau_{k-1}).
\end{equation*}
Applying $RU$ to $\widetilde W$, the entry $1$ becomes the new maximum
$2k-1$ and moves from position $d$ to position $2k-d$. Moreover,
\begin{gather*}
\tau^{RU}
=
(k-\tau_{k-1})\cdots(k-\tau_1), \ \text{and} \ \sigma^{RU}
=
(k-\sigma_{k-1})\cdots(k-\sigma_1).
\end{gather*}
Hence
\begin{equation*}
\widetilde W^{RU}
=
\ins{2k-d}
\bigl(A_{k-1}(\tau^{RU},\sigma^{RU})\bigr).
\end{equation*}

By the induction hypothesis $(D_{k-1})$,
\begin{equation*}
F_{k-1}(\tau^{RU},\sigma^{RU})
=
(\beta^{RU},\alpha^{RU}),
\end{equation*}
and by the definition of $F_{k-1}$, we have
\begin{equation*}
f_{2k-2}
\bigl(A_{k-1}(\tau^{RU},\sigma^{RU})\bigr)
=
B_{k-1}(\beta^{RU},\alpha^{RU}).
\end{equation*}

Set $s=k+1-d$. Applying the odd step in the definition of
$f_{2k-1}$, we obtain
\begin{equation*}
f_{2k-1}(\widetilde W^{RU})
=
\inss{2(2k-d-(k-1))-1}
\bigl(B_{k-1}(\beta^{RU},\alpha^{RU})\bigr)
=
\inss{2s-1}
\bigl(B_{k-1}(\beta^{RU},\alpha^{RU})\bigr),
\end{equation*}
which equals
\begin{equation*}
B_k^\circ
\bigl(I_s^+(\beta^{RU}),\alpha^{RU}\bigr)
\end{equation*} by Lemma~\ref{l:B circ with inss}. 

By Lemma~\ref{l:ins RU identity}, we have $I_s^+(\beta^{RU}) =
\bigl(I_d^-(\beta)\bigr)^{RU}$. Therefore, 
\begin{equation*}
f_{2k-1}(\widetilde W^{RU})
=
B_k^\circ
\bigl((I_d^-(\beta))^{RU},\alpha^{RU}\bigr).
\end{equation*}
By the definition of $\tilde{f}_{2k-1}$ and Lemma~\ref{l:B circ with RU}, we have
\begin{equation*}
\tilde {f}_{2k-1}(\widetilde W) = \bigr(f_{2k-1}(\widetilde W^{RU})\bigr)^{RU} = \left[
B_k^\circ
\bigl((I_d^-(\beta))^{RU},\alpha^{RU}\bigr)
\right]^{RU} =
B_k^\circ
\bigl(I_d^-(\beta),\alpha\bigr).
\end{equation*}
Consequently, by Lemma~\ref{l:B circ with inss}, we have
\begin{equation*}
f_{2k}(W)=
\inss{2h}
\left(
B_k^\circ\bigl(I_d^-(\beta),\alpha\bigr)
\right)=
B_k
\bigl(I_d^-(\beta),I_h^+(\alpha)\bigr).
\end{equation*}
Therefore, by the definition of $F_k$, we conclude that
\begin{equation*}
F_k
\bigl(I_d^-(\sigma),I_h^+(\tau)\bigr)
=
\bigl(I_d^-(\beta),I_h^+(\alpha)\bigr),
\end{equation*}
which proves $(E_k)$.

We next assume $(E_{k})$ together with $(D_{k-1})$
to prove $(D_k)$. Suppose that
\begin{equation*}
F_k(\sigma,\tau)=(\alpha,\beta).
\end{equation*}
Let $d$ be the position of $1$ in $\sigma$, and let $h$ be the position
of $k$ in $\tau$. There are unique
$\sigma',\tau'\in\mathfrak{S}_{k-1}$ such that
\begin{equation*}
\sigma=I_d^-(\sigma'),
\qquad
\tau=I_h^+(\tau').
\end{equation*}
Write
\begin{equation*}
F_{k-1}(\sigma',\tau')=(\alpha',\beta').
\end{equation*}
Applying $(E_k)$, we have
\begin{equation*}
\alpha=I_d^-(\beta'),
\qquad
\beta=I_h^+(\alpha').
\end{equation*}

By $(D_{k-1})$,
\begin{equation*}
F_{k-1}
\bigl((\tau')^{RU},(\sigma')^{RU}\bigr)
=
\bigl((\beta')^{RU},(\alpha')^{RU}\bigr).
\end{equation*}
Lemma~\ref{l:ins RU identity} give
\begin{equation*}
\tau^{RU}
=
I_{k+1-h}^-\bigl((\tau')^{RU}\bigr),
\qquad
\sigma^{RU}
=
I_{k+1-d}^+\bigl((\sigma')^{RU}\bigr).
\end{equation*}
Applying $(E_k)$ again, $F_k(\tau^{RU},\sigma^{RU})$ equals
\begin{equation*}
F_k\left(
I_{k+1-h}^-\bigl((\tau')^{RU}\bigr),
I_{k+1-d}^+\bigl((\sigma')^{RU}\bigr)
\right)
=
\left(
I_{k+1-h}^-\bigl((\alpha')^{RU}\bigr),
I_{k+1-d}^+\bigl((\beta')^{RU}\bigr)
\right).
\end{equation*}
Using Lemma~\ref{l:ins RU identity} once more,
\begin{equation*}
I_{k+1-h}^-\bigl((\alpha')^{RU}\bigr)
=\bigr(I_h^+(\alpha')\bigr)^{RU}
=\beta^{RU},
\qquad
I_{k+1-d}^+\bigl((\beta')^{RU}\bigr)
= \bigr(I_d^-(\beta')\bigr)^{RU}
=\alpha^{RU}.
\end{equation*}

Therefore,
\begin{equation*}
F_k(\tau^{RU},\sigma^{RU})
=
(\beta^{RU},\alpha^{RU}),
\end{equation*}
which proves $(D_k)$ and completes the induction.
\end{proof}

We next recall the rank-matrix criterion for the Bruhat order. For
$w=w_1\cdots w_n\in\mathfrak{S}_n$ and $a,b\in[n]$, define
\begin{equation*}
r_w(a,b)
=
\left|\{\,i\in[a]\mid w_i\leq b\,\}\right|.
\end{equation*}
For $a=0$ or $b=0$, we use the convention
\begin{equation*}
r_w(0,b)=r_w(a,0)=0.
\end{equation*}
The rank matrix of $w$ is
\begin{equation*}
R_w
=
\bigl(r_w(a,b)\bigr)_{1\leq a,b\leq n}.
\end{equation*}

Thus, $r_w(a,b)$ counts the number of entries among the first $a$
positions of $w$ whose values are at most $b$. 

The following well-known rank-matrix criterion~\cite[Theorem~2.1.5]{BBCoxeter} characterizes the
Bruhat order.

\begin{thm}[Rank-matrix criterion for the Bruhat order]\label{t:rank-matrix}
Let $u,w\in\mathfrak{S}_n$. Then
\begin{equation*}
u\leq w
\quad\Longleftrightarrow\quad
r_u(a,b)\geq r_w(a,b)
\end{equation*}
for all $a,b\in[n]$.
\end{thm}

For example, consider $u=2143$ and $w=3241$. Their rank matrices are
\begin{equation*}
R_u=
\begin{bmatrix}
0&1&1&1\\
1&2&2&2\\
1&2&2&3\\
1&2&3&4
\end{bmatrix},
\qquad
R_w=
\begin{bmatrix}
0&0&1&1\\
0&1&2&2\\
0&1&2&3\\
1&2&3&4
\end{bmatrix}.
\end{equation*}
Since every entry of $R_u$ is greater than or equal to the
corresponding entry of $R_w$, the rank-matrix criterion gives
$2143\leq 3241$. 

Let $w\in\mathfrak{A}_{2k}$. Write
\begin{equation*}
w=A_k(\sigma,\tau)
\end{equation*}
for some $\sigma,\tau\in\mathfrak{S}_k$. Suppose that
\begin{equation*}
f_{2k}(w)=B_k(\alpha,\beta),
\end{equation*}
where $\alpha,\beta\in\mathfrak{S}_k$. Equivalently, by the definition
of $F_k$,
\begin{equation*}
F_k(\sigma,\tau)=(\alpha,\beta).
\end{equation*}
Thus, to prove
\begin{equation*}
w\leq f_{2k}(w),
\end{equation*}
it is enough to prove
\begin{equation*}
A_k(\sigma,\tau)\leq B_k(\alpha,\beta).
\end{equation*}
By Theorem~\ref{t:rank-matrix}, it therefore suffices to
show that
\begin{equation}\label{eq:rank-matrix}
r_{A_k(\sigma,\tau)}(a,b)
\geq
r_{B_k(\alpha,\beta)}(a,b)
\end{equation}
for all $a,b\in[2k]$. 
We will verify \eqref{eq:rank-matrix} by considering the four regions determined
by whether $a$ and $b$ are at most or greater than $k$. 

To compare the rank matrices of $A_k(\sigma,\tau)$ and
$B_k(\alpha,\beta)$, it will be useful to express the rank function of
$B_k(\alpha,\beta)$ directly in terms of the rank functions of
$\alpha$ and $\beta$. The following lemma gives this decomposition.

\begin{lem}\label{l:rank B}
Let $\alpha,\beta\in\mathfrak S_k$ and
\begin{equation*}
u=B_k(\alpha,\beta)\in\mathfrak B_{2k}.
\end{equation*}
Then, for all $a,b\in[2k]$,
\begin{equation*}
r_u(a,b)
=
r_\alpha\left(
\ceil[\Big]{\frac{a}{2}},
\ceil[\Big]{\frac{b}{2}}
\right)
+
r_\beta\left(
\floor[\Big]{\frac{a}{2}},
\floor[\Big]{\frac{b}{2}}
\right).
\end{equation*}
\end{lem}

\begin{proof}
Among the first $a$ positions of $u=B_k(\alpha,\beta)$, there are
$\ceil{\tfrac{a}{2}}$ odd positions and
$\floor{\tfrac{a}{2}}$ even positions.

The entries in the odd positions are
\begin{equation*}
2\alpha_1-1,\ldots,
2\alpha_{\ceil{\tfrac{a}{2}}}-1.
\end{equation*}
For each $i$, we know that
\begin{equation*}
2\alpha_i-1\leq b
\quad\Longleftrightarrow\quad
\alpha_i\leq \ceil[\Big]{\frac{b}{2}}.
\end{equation*}
Hence the contribution from the odd positions is $r_\alpha\left(
\ceil{\tfrac{a}{2}},
\ceil{\tfrac{b}{2}}
\right)$.

Similarly, the entries in the even positions are
\begin{equation*}
2\beta_1,\ldots,
2\beta_{\floor{\tfrac{a}{2}}}.
\end{equation*}
For each $j$, we know that
\begin{equation*}
2\beta_j\leq b
\quad\Longleftrightarrow\quad
\beta_j\leq \floor[\Big]{\frac{b}{2}}.
\end{equation*}
Hence the contribution from the even positions is
$r_\beta\left(
\floor{\tfrac{a}{2}},
\floor{\tfrac{b}{2}}
\right)$.

Adding the two contributions gives
\begin{equation*}
r_u(a,b)
=
r_\alpha\left(
\ceil[\Big]{\frac{a}{2}},
\ceil[\Big]{\frac{b}{2}}
\right)
+
r_\beta\left(
\floor[\Big]{\frac{a}{2}},
\floor[\Big]{\frac{b}{2}}
\right).
\end{equation*}
\end{proof}

By Theorem~\ref{t:rank-matrix}, it remains to compare the
rank functions of $A_k(\sigma,\tau)$ and $B_k(\alpha,\beta)$ at every
pair $(a,b)\in[2k]\times[2k]$. We divide this comparison into four
regions according to whether $a$ and $b$ are at most $k$ or greater
than $k$. Before we begin with the northwest quadrant, where $a,b\in[k]$. We need some auxiliary lemmas. 

\begin{lem}\label{l:rank insertion}
Let $w\in\mathfrak S_{k-1}$.

\begin{enumerate}[(i)]
    \item For $d\in[k]$ and $a,b\in[k]$,
    \begin{equation}
    r_{I_d^-(w)}(a,b)
    =
    \begin{cases}
    r_w(a,b-1), & d>a,\\[4pt]
    1+r_w(a-1,b-1), & d\leq a.
    \end{cases}
    \label{eq:rank minimum insertion}
    \end{equation}

    \item For $h\in[k]$, $a\in[k]$, and $b\in[k-1]$,
    \begin{equation}
    r_{I_h^+(w)}(a,b)
    =
    \begin{cases}
    r_w(a,b), & h>a,\\[4pt]
    r_w(a-1,b), & h\leq a.
    \end{cases}
    \label{eq:rank maximum insertion}
    \end{equation}
\end{enumerate}
Here we use the convention $r_w(a,0)=0$.
\end{lem}

\begin{proof}
For \eqref{eq:rank minimum insertion}, recall that $I_d^-(w)$ is
obtained by increasing every entry of $w$ by $1$ and inserting $1$ in
position $d$. If $d>a$, the inserted $1$ does not occur among the first
$a$ positions, and an entry of $I_d^-(w)$ is at most $b$ precisely when
the corresponding entry of $w$ is at most $b-1$. Hence
\begin{equation*}
r_{I_d^-(w)}(a,b)=r_w(a,b-1).
\end{equation*}
If $d\leq a$, the inserted $1$ contributes one to the rank, while the
remaining $a-1$ entries contribute $r_w(a-1,b-1)$. Thus
\begin{equation*}
r_{I_d^-(w)}(a,b)
=
1+r_w(a-1,b-1).
\end{equation*}

For \eqref{eq:rank maximum insertion}, $I_h^+(w)$ is obtained by
inserting the new maximum $k$ in position $h$. Since $b\leq k-1$, the
inserted entry $k$ does not contribute to the rank. If $h>a$, the first
$a$ entries are unchanged, while if $h\leq a$, only $a-1$ entries of
$w$ occur among the first $a$ positions. Therefore,
\begin{equation*}
r_{I_h^+(w)}(a,b)
=
\begin{cases}
r_w(a,b), & h>a,\\[4pt]
r_w(a-1,b), & h\leq a.
\end{cases}
\end{equation*}
\end{proof}

\begin{lem}\label{l:Rk}
Let $\sigma,\tau,\alpha,\beta\in\mathfrak S_k$, and suppose that
\begin{equation*}
F_k(\sigma,\tau)=(\alpha,\beta).
\end{equation*}
Then, for all $a,b\in[k]$ and $0\leq p\leq a$,
\begin{equation}
r_\sigma(a,b)
\geq
r_\alpha\left(
p,
\ceil[\Big]{\frac{b}{2}}
\right)
+
r_\beta\left(
a-p,
\floor[\Big]{\frac{b}{2}}
\right).
\tag{$R_k$}
\end{equation}
\end{lem}

\begin{proof}
Before we start, we call $a,b,p$ the first, second, third parameter of $(R_k)$. Now, we proceed by induction on $k$. We use the convention
$r_w(0,b)=r_w(a,0)=0$. With this convention, the statement is also
trivial when one of the parameters occurring in the induction
hypothesis is zero. 

For $k=2$, by Algorithm~\ref{keyalg}, the map $F_2$ is the identity map on
$\mathfrak S_2\times\mathfrak S_2$. Thus $\alpha=\sigma$ and
$\beta=\tau$. If $b=1$, then
\begin{equation*}
r_\alpha(p,1)+r_\beta(a-p,0)
=
r_\sigma(p,1)
\leq
r_\sigma(a,1).
\end{equation*}
If $b=2$, then
\begin{equation*}
r_\alpha(p,1)+r_\beta(a-p,1)
\leq
p+(a-p)
=
a
=
r_\sigma(a,2).
\end{equation*}
Hence $(R_2)$ holds.

Now let $k\geq 3$ and assume $(R_{k-1})$. Suppose that
\begin{equation*}
F_k(\sigma,\tau)=(\alpha,\beta).
\end{equation*}
Let $d$ be the position of $1$ in $\sigma$, and let $h$ be the position
of $k$ in $\tau$. Then there exist
$\sigma',\tau'\in\mathfrak S_{k-1}$ such that
\begin{equation*}
\sigma=I_d^-(\sigma'),
\qquad
\tau=I_h^+(\tau').
\end{equation*}
Write
\begin{equation*}
F_{k-1}(\sigma',\tau')=(\alpha',\beta').
\end{equation*}
By $(E_k)$ in Proposition~\ref{p:Ek Dk},
\begin{equation*}
\alpha=I_d^-(\beta'),
\qquad
\beta=I_h^+(\alpha').
\end{equation*}

Fix $a,b\in[k]$ and $0\leq p\leq a$. Define
\begin{equation*}
\delta=\mathbf{1}_{d\leq a},
\qquad
\eta=\mathbf{1}_{d\leq p},
\qquad
\varepsilon=\mathbf{1}_{h\leq a-p}.
\end{equation*}
Since $p\leq a$, we always have $\eta\leq\delta$.

By \eqref{eq:rank minimum insertion} and
\eqref{eq:rank maximum insertion}, we have

\begin{gather}
r_\sigma(a,b)
=
\delta+r_{\sigma'}(a-\delta,b-1),\label{eq:rank function with delta}\\
r_\alpha\left(p,\ceil[\Big]{\frac{b}{2}}\right)
=
\eta+
r_{\beta'}\left(
p-\eta,
\ceil[\Big]{\frac{b}{2}}-1
\right),\label{eq:rank function alpha}\\
r_\beta\left(a-p,\floor[\Big]{\frac{b}{2}}\right)
=
r_{\alpha'}\left(
a-p-\varepsilon,
\floor[\Big]{\frac{b}{2}}
\right).\label{eq:rank function beta}
\end{gather}

Now, set $q=a-p-\varepsilon$.

We first consider the case $(\delta,p,\varepsilon)\neq(1,0,0)$.
Then, we have $0\leq q\leq a-\delta$.
Indeed, $q\geq0$ by the definition of $\varepsilon$, while
$q\leq a-\delta$ is equivalent to $p+\varepsilon\geq\delta$, which
fails only when $(\delta,p,\varepsilon)=(1,0,0)$.

Note that $(a-\delta)-q = (a-\delta)-(a-p-\varepsilon) =p+\varepsilon -\delta$. By the definition of $\delta$, we have \begin{equation*}
a-\delta \leq
\begin{cases}
k - \delta = k -1, &\text{when} \ a = k,\\[4pt]
k-1-\delta \leq k-1, &\text{when} \ a < k. 
\end{cases}
\end{equation*}
We also note that $b-1 \leq k-1$. 
Applying $(R_{k-1})$ with parameters $a-\delta,b-1,q$ and using
\begin{equation*}
\ceil[\Big]{\frac{b-1}{2}}
=
\floor[\Big]{\frac{b}{2}},
\qquad
\floor[\Big]{\frac{b-1}{2}}
=
\ceil[\Big]{\frac{b}{2}}-1,
\end{equation*}
we obtain
\begin{equation*}
r_{\sigma'}(a-\delta,b-1)
\geq{}
r_{\alpha'}\left(
q,
\floor[\Big]{\frac{b}{2}}
\right) +
r_{\beta'}\left(
p+\varepsilon-\delta,
\ceil[\Big]{\frac{b}{2}}-1
\right).
\end{equation*}
Comparing this inequality with \eqref{eq:rank function with delta}, we have
\begin{equation}\label{eq:NW first ineq}
r_\sigma(a,b)
\geq{}
r_{\alpha'}\left(
q,
\floor[\Big]{\frac{b}{2}}
\right)+
\delta+
r_{\beta'}\left(
p+\varepsilon-\delta,
\ceil[\Big]{\frac{b}{2}}-1
\right).
\end{equation}

Next, we claim that
\begin{equation}\label{eq:NW prop claim}
\delta+
r_{\beta'}\left(
p+\varepsilon-\delta,
\ceil[\Big]{\frac{b}{2}}-1
\right)
\geq
\eta+
r_{\beta'}\left(
p-\eta,
\ceil[\Big]{\frac{b}{2}}-1
\right).
\end{equation}
To see this, we first note that $\delta \geq \eta$ by above definition. Then, it remains to compare the other two terms. 
If $p-\eta\leq p+\varepsilon-\delta$, \eqref{eq:NW prop claim} follows from the monotonicity of the rank function in its first argument. Otherwise, we have $p-\eta > p+\varepsilon-\delta$, i.e., $\eta < \delta -\varepsilon$. Hence 
\begin{equation*}
\delta=1,
\qquad
\eta=\varepsilon=0,
\end{equation*}
and the claim~\eqref{eq:NW prop claim} follows from $r_{\beta'}(x+1,y)\leq r_{\beta'}(x,y)+1$ for all (possible) $x$ and $y$. 

Comparing the right-hand-side of \eqref{eq:NW prop claim} with \eqref{eq:rank function alpha} and by \eqref{eq:rank function beta}, we have 
\begin{equation*}
r_\sigma(a,b)
\geq{}
r_{\alpha'}\left(
q,
\floor[\Big]{\frac{b}{2}}
\right)
+
\eta+
r_{\beta'}\left(
p-\eta,
\ceil[\Big]{\frac{b}{2}}-1
\right)
={}
r_\beta\left(
a-p,\floor[\Big]{\frac{b}{2}}
\right)
+
r_\alpha\left(
p,\ceil[\Big]{\frac{b}{2}}
\right).
\end{equation*}

It remains to consider the boundary case $(\delta,p,\varepsilon)=(1,0,0)$.
Since $p=0$, by the definition of $\eta$, we have $\eta=0$. Since $\varepsilon=0$ and $p=0$, we have $h>a$. Therefore, \eqref{eq:rank function with delta} and \eqref{eq:rank function beta} become
\begin{equation*}
r_\sigma(a,b)
=
1+r_{\sigma'}(a-1,b-1)
\end{equation*}
and
\begin{equation*}
r_\beta\left(a,\floor[\Big]{\frac{b}{2}}\right)
=
r_{\alpha'}\left(a,\floor[\Big]{\frac{b}{2}}\right).
\end{equation*}
Indeed, $r_{\alpha'}\left(a,\floor*{\tfrac{b}{2}}\right)$ is well-defined since $h > a$ implies $a \leq k-1$. 
Applying $(R_{k-1})$ 
such that the third parameter equals $a-1$,
we obtain
\begin{equation*}
r_{\sigma'}(a-1,b-1)
\geq
r_{\alpha'}\left(
a-1,\floor[\Big]{\frac{b}{2}}
\right) + r_{\beta'}\left(
0,
\ceil[\Big]{\frac{b}{2}}-1
\right) = r_{\alpha'}\left(
a-1,\floor[\Big]{\frac{b}{2}}
\right).
\end{equation*}
Since increasing the first argument of a rank function by one increases
its value by at most one, we have
\begin{equation*}
1+
r_{\alpha'}\left(
a-1,\floor[\Big]{\frac{b}{2}}
\right)
\geq
r_{\alpha'}\left(
a,\floor[\Big]{\frac{b}{2}}
\right).
\end{equation*}
Hence, by \eqref{eq:rank function with delta}
\begin{equation*}
r_\sigma(a,b)
\geq
r_{\alpha'}\left(
a,\floor[\Big]{\frac{b}{2}}
\right).
\end{equation*}
Since $\beta = I_h^+(\alpha')$ and $h>a$, inserting the maximum entry at position $h$ does not change the rank function. We have 
\begin{equation*}
   r_{\alpha'}\left(
a,\floor[\Big]{\frac{b}{2}}
\right) = r_{I_h^+(\alpha')}\left(
a,\floor[\Big]{\frac{b}{2}}
\right) = r_{\beta}\left(
a,\floor[\Big]{\frac{b}{2}}
\right).
\end{equation*}
Hence, we have 
\begin{equation*}
    r_\sigma(a,b) \geq r_{\beta}\left(
a,\floor[\Big]{\frac{b}{2}}
\right)
\end{equation*}
Since
$r_\alpha\left(
0,\ceil{\tfrac{b}{2}}
\right)=0$, this is precisely $(R_k)$ when $p=0$.

Thus $(R_k)$ holds, completing the induction.
\end{proof}

The estimate $(R_k)$ in Lemma~\ref{l:Rk} is stronger than what is needed for the northwest
quadrant. Indeed, the latter is obtained by choosing
$p=\ceil{\tfrac{a}{2}}$ in $(R_k)$. This immediately gives the following corollary.

\begin{cor}[Northwest quadrant]\label{c:NW quadrant}
Let $\sigma,\tau,\alpha,\beta\in\mathfrak S_k$, and suppose that
\begin{equation*}
F_k(\sigma,\tau)=(\alpha,\beta).
\end{equation*}
Then, for all $a,b\in[k]$,
\begin{equation*}
r_{A_k(\sigma,\tau)}(a,b)
\geq
r_{B_k(\alpha,\beta)}(a,b).
\end{equation*}
\end{cor}

\begin{proof}
Since $a,b\in[k]$, by the definition of $A_k$,
\begin{equation*}
r_{A_k(\sigma,\tau)}(a,b)=r_\sigma(a,b).
\end{equation*}
Applying $(R_k)$ in Lemma~\ref{l:Rk} with
$p=\ceil{\tfrac{a}{2}}$,
we obtain
\begin{equation*}
r_\sigma(a,b)
\geq{}
r_\alpha\left(
\ceil[\Big]{\frac{a}{2}},
\ceil[\Big]{\frac{b}{2}}
\right)
+
r_\beta\left(
\floor[\Big]{\frac{a}{2}},
\floor[\Big]{\frac{b}{2}}
\right),
\end{equation*}
where we have used
\begin{equation*}
a-\ceil[\Big]{\frac{a}{2}}
=
\floor[\Big]{\frac{a}{2}}.
\end{equation*}
By Lemma~\ref{l:rank B}, the right-hand side is precisely
\begin{equation*}
r_{B_k(\alpha,\beta)}(a,b).
\end{equation*}
Hence
\begin{equation*}
r_{A_k(\sigma,\tau)}(a,b)
\geq
r_{B_k(\alpha,\beta)}(a,b),
\end{equation*}
as desired.
\end{proof}

The northwest quadrant has now been settled. The two mixed quadrants are
considerably simpler: in these regions, the desired rank inequalities
follow directly from the block structure of $A_k(\sigma,\tau)$. We record
these two cases together in the following lemma.

\begin{lem}[Mixed quadrants]\label{l:mixed quadrants}
Let $\sigma,\tau,\alpha,\beta\in\mathfrak S_k$, and suppose that
\begin{equation*}
F_k(\sigma,\tau)=(\alpha,\beta).
\end{equation*}
Then
\begin{equation*}
r_{A_k(\sigma,\tau)}(a,b)
\geq
r_{B_k(\alpha,\beta)}(a,b)
\end{equation*}
whenever either $a\leq k<b$ or $b\leq k<a$.
\end{lem}

\begin{proof}
First suppose that $a\leq k<b$. Since the first $a$ entries of
$A_k(\sigma,\tau)$ all belong to $[k]$, they are all at most $b$.
Therefore, $r_{A_k(\sigma,\tau)}(a,b)=a$.

On the other hand, for any permutation $w$ we have
$r_w(a,b)\leq a$. Hence
\begin{equation*}
r_{A_k(\sigma,\tau)}(a,b)
=
a
\geq
r_{B_k(\alpha,\beta)}(a,b).
\end{equation*}

Now suppose that $b\leq k<a$. Among the first $k$ entries of
$A_k(\sigma,\tau)$, each element of $[k]$ occurs exactly once, while
all entries after position $k$ are greater than $k$. Thus exactly $b$
of the first $a$ entries are at most $b$, and hence
$r_{A_k(\sigma,\tau)}(a,b)=b$.
Since $r_w(a,b)\leq b$ for every permutation $w$, we obtain
\begin{equation*}
r_{A_k(\sigma,\tau)}(a,b)
=
b
\geq
r_{B_k(\alpha,\beta)}(a,b).
\end{equation*}
\end{proof}

The northwest and the two mixed quadrants have now been settled. To
handle the remaining southeast quadrant, we will relate the rank
function in the southeast region to the rank function of the
$RU$-transform in the northwest region. We first record the following
general identity.

\begin{lem}\label{l:rank RU}
Let $w\in\mathfrak S_{2k}$. Then, for all $a,b\in[2k]$,
\begin{equation}
r_w(a,b)
=
a+b-2k
+
r_{w^{RU}}(2k-a,2k-b).
\label{eq:rank RU}
\end{equation}
Here we use the convention
$r_w(0,b)=r_w(a,0)=0$.
\end{lem}

\begin{proof}
Among the $2k$ entries of $w$, exactly $2k-b$ are greater than $b$.
Let $N$ denote the number of such entries occurring among the first
$a$ positions. Then
\begin{equation*}
r_w(a,b)=a-N.
\end{equation*}

Write
$w^{RU}=v_1\cdots v_{2k}$. 
By the definition of $RU$,
$v_j=2k+1-w_{2k+1-j}$
for every $j\in[2k]$.

Now consider an entry $w_i$ of $w$ with $i>a$ and $w_i>b$.
Under $RU$, this entry corresponds to the entry
\begin{equation*}
v_{2k+1-i}=2k+1-w_i
\end{equation*}
of $w^{RU}$. Since $i>a$, we have
\begin{equation*}
2k+1-i\leq 2k-a,
\end{equation*}
and since $w_i>b$, we have
\begin{equation*}
2k+1-w_i\leq 2k-b.
\end{equation*}
Thus, each entry of $w$ greater than $b$ and occurring after position
$a$ corresponds to an entry of $w^{RU}$ among its first $2k-a$
positions whose value is at most $2k-b$.

Conversely, suppose that $v_j$ is among the first $2k-a$ positions of
$w^{RU}$ with $v_j\leq 2k-b$. The corresponding entry of $w$
is
\begin{equation*}
w_{2k+1-j}=2k+1-v_j.
\end{equation*}
Since $j\leq 2k-a$, we have
\begin{equation*}
2k+1-j>a,
\end{equation*}
and since $v_j\leq 2k-b$, we have
\begin{equation*}
2k+1-v_j>b.
\end{equation*}
Therefore, this correspondence is a bijection between the entries of
$w$ greater than $b$ occurring after position $a$ and the entries of
$w^{RU}$ among its first $2k-a$ positions whose values are at most
$2k-b$.

Hence the number of entries of $w$ greater than $b$ occurring after
position $a$ is
\begin{equation*}
r_{w^{RU}}(2k-a,2k-b).
\end{equation*}
Since the total number of entries of $w$ greater than $b$ is $2k-b$,
we obtain
\begin{equation*}
N
=
2k-b-r_{w^{RU}}(2k-a,2k-b).
\end{equation*}
Substituting this into $r_w(a,b)=a-N$ gives
\begin{equation*}
r_w(a,b)=
a+b-2k+r_{w^{RU}}(2k-a,2k-b).
\end{equation*}
\end{proof}

We now return to the southeast quadrant. Lemma~\ref{l:rank RU} allows us to transfer the desired inequality to the northwest
quadrant, while the symmetry $(D_k)$ ensures that the corresponding
$RU$-transformed permutations are still related by $F_k$. We can
therefore complete the final quadrant comparison.

\begin{lem}[Southeast quadrant]\label{l:SE quadrant}
Let $\sigma,\tau,\alpha,\beta\in\mathfrak S_k$, and suppose that
\begin{equation*}
F_k(\sigma,\tau)=(\alpha,\beta).
\end{equation*}
Then, for all $a,b\in[k+1,2k]$, we have 
\begin{equation*}
r_{A_k(\sigma,\tau)}(a,b)
\geq
r_{B_k(\alpha,\beta)}(a,b).
\end{equation*}
\end{lem}

\begin{proof}
Fix $a,b\in[k+1,2k]$, and set
\begin{equation*}
\bar a=2k-a,
\qquad
\bar b=2k-b.
\end{equation*}
Then we have $0\leq \bar a,\bar b\leq k-1$.

By Lemma~\ref{l:RU on A B}, we have
\begin{equation*}
A_k(\sigma,\tau)^{RU}
=
A_k(\tau^{RU},\sigma^{RU}), \ \text{and} \
B_k(\alpha,\beta)^{RU}
=
B_k(\beta^{RU},\alpha^{RU}).
\end{equation*}
Moreover, by $(D_k)$ in Proposition~\ref{p:Ek Dk},
\begin{equation*}
F_k(\tau^{RU},\sigma^{RU})
=
(\beta^{RU},\alpha^{RU}).
\end{equation*}

If $\bar a,\bar b\geq 1$, Corollary~\ref{c:NW quadrant} applied to
\begin{equation*}
F_k(\tau^{RU},\sigma^{RU})
=
(\beta^{RU},\alpha^{RU})
\end{equation*}
gives
\begin{equation*}
r_{A_k(\tau^{RU},\sigma^{RU})}(\bar a,\bar b)
\geq
r_{B_k(\beta^{RU},\alpha^{RU})}(\bar a,\bar b).
\end{equation*}

If either $\bar a=0$ or $\bar b=0$, both sides are zero.
Thus, in all cases,
\begin{equation}\label{eq:SE quadrant RU}
r_{A_k(\sigma,\tau)^{RU}}(\bar a,\bar b)
\geq
r_{B_k(\alpha,\beta)^{RU}}(\bar a,\bar b).
\end{equation}

Applying \eqref{eq:rank RU} to $A_k(\sigma,\tau)$, we obtain
\begin{equation*}
r_{A_k(\sigma,\tau)}(a,b)
=
a+b-2k
+
r_{A_k(\sigma,\tau)^{RU}}(\bar a,\bar b).
\end{equation*}
Similarly,
\begin{equation*}
r_{B_k(\alpha,\beta)}(a,b)
=
a+b-2k
+
r_{B_k(\alpha,\beta)^{RU}}(\bar a,\bar b).
\end{equation*}
Therefore, by \eqref{eq:SE quadrant RU}, we have 
\begin{align*}
r_{A_k(\sigma,\tau)}(a,b)
&=
a+b-2k
+
r_{A_k(\sigma,\tau)^{RU}}(\bar a,\bar b)\\
&\geq
a+b-2k
+
r_{B_k(\alpha,\beta)^{RU}}(\bar a,\bar b)\\
&=
r_{B_k(\alpha,\beta)}(a,b).
\end{align*}
This proves the southeast quadrant.
\end{proof}

Combining the northwest, mixed, and southeast quadrant comparisons, we
immediately obtain the following corollary. 

\begin{cor}\label{c:rank comparison even}
Let $\sigma,\tau,\alpha,\beta\in\mathfrak S_k$, and suppose that
\begin{equation*}
F_k(\sigma,\tau)=(\alpha,\beta).
\end{equation*}
Then, for all $a,b\in[2k]$, we have
\begin{equation*}
r_{A_k(\sigma,\tau)}(a,b)
\geq
r_{B_k(\alpha,\beta)}(a,b).
\end{equation*}
\end{cor}

The preceding rank comparisons establish the required inequality at
every entry of the rank matrix. We can therefore apply the rank-matrix
criterion to obtain the even case of the main result.

\begin{thm}\label{thm:even case}
Let $n=2k$ ($k\geq 2$). Then, for every $w\in\mathfrak A_{2k}$,
\begin{equation*}
w\leq f_{2k}(w)
\end{equation*}
under the Bruhat order.
\end{thm}

\begin{proof}
This follows immediately from Corollary~\ref{c:rank comparison even}
and Theorem~\ref{t:rank-matrix}.
\end{proof}

The odd case follows directly from the even case. More precisely, the
result for $n=2k$ implies the corresponding result for $n=2k-1$. 

\begin{cor}\label{c:odd case}
Let $n=2k-1$ ($k\geq 3$). Then, for every $w\in\mathfrak A_{2k-1}$,
\begin{equation*}
w\leq f_{2k-1}(w)
\end{equation*}
under the Bruhat order.
\end{cor}

\begin{proof}
Let $w\in\mathfrak A_{2k-1}$. Since
$w^{RU}\in\widetilde{\mathfrak A}_{2k-1}$, we have
\begin{equation*}
\ins{2k}(w^{RU})\in\mathfrak A_{2k}.
\end{equation*}
By the definition of $\widetilde f_{2k-1}$,
\begin{equation*}
\widetilde f_{2k-1}(w^{RU})
=
\left[f_{2k-1}\bigl((w^{RU})^{RU}\bigr)\right]^{RU}
=
f_{2k-1}(w)^{RU}.
\end{equation*}
Since $\inss{2k}$ inserts $2k$ at the last position, no swapping occurs,
and hence $\inss{2k}=\ins{2k}$. Therefore, by the definition of
$f_{2k}$,
\begin{equation*}
f_{2k}\bigl(\ins{2k}(w^{RU})\bigr)
=
\ins{2k}\bigl(f_{2k-1}(w)^{RU}\bigr).
\end{equation*}

By Theorem~\ref{thm:even case},
\begin{equation*}
\ins{2k}(w^{RU})
\leq
\ins{2k}\bigl(f_{2k-1}(w)^{RU}\bigr).
\end{equation*}
Both permutations have the common last entry $2k$. By the tableau criterion (Theorem~\ref{TabCriterion}), deleting this common last entry gives
\begin{equation*}
w^{RU}\leq f_{2k-1}(w)^{RU}.
\end{equation*}
Since $RU$ preserves the Bruhat order (Proposition~\ref{upside-down reverse}) and is an involution, we conclude
that
\begin{equation*}
w\leq f_{2k-1}(w).
\end{equation*}
\end{proof}

Combining the even and odd cases, we obtain the main result of this section.

\begin{thm}\label{t:Bruhat balance case}
Let $f_n$ be the map defined in Algorithm~\ref{keyalg}. For every $n\geq 4$ and every $w\in\mathfrak A_n$, we have
\begin{equation*}
w\leq f_n(w)
\end{equation*}
under the Bruhat order.
\end{thm}

\begin{proof}
    For even \(n\), the conclusion follows from Theorem~\ref{thm:even case}. For odd \(n\ge5\), it follows from Corollary~\ref{c:odd case}.
\end{proof}

By Theorem~\ref{pairing test} and Theorem~\ref{t:Bruhat balance case}, we immediately get the balanced case of~\eqref{eq:intromain}. 

\begin{thm}
    Let $A$ be an $n \times n$ TNN matrix. Then the inequality 
    \begin{equation*}
        \perm(A_{[1,h],[1,h]}) \perm(A_{[h+1,n],[h+1,n]}) \;\geq\;
    \perm(A_{[1,n] \cap 2\mathbb Z, [1,n] \cap 2\mathbb Z})
    \perm(A_{[1,n] \ssm 2\mathbb Z, [1,n] \ssm 2\mathbb Z})
    \end{equation*}
    holds when $h = \floor*{\tfrac{n}{2}}$ or $\ceil*{\tfrac{n}{2}}$. 
\end{thm}

\section{Permanental Inequalities for Arbitrary $h$}\label{s:all h}

In Sections~\ref{s:A gen on all} and \ref{s:proof of balance case}, we proved the permanental inequality~\eqref{eq:intromain} for arbitrary totally nonnegative matrices in the balanced case
$h=\floor*{\frac{n}{2}}$
or
$h=\ceil*{\frac{n}{2}}$.
The proof relied on the bijection $f_n:\mathfrak A_n\to\mathfrak B_n$ and the result that
\begin{equation*}
w\leq f_n(w)
\end{equation*}
under the Bruhat order for every $w\in\mathfrak A_n$ (Theorem~\ref{t:Bruhat balance case}).

In this section, we extend this result to every $h\in[1,n-1]$. When $h$ is not balanced, the two families of permutation monomials occurring in \eqref{eq:evenoddintro} no longer have the same cardinality, so a bijection analogous to $f_n$ cannot exist. Instead, we will construct an injection from $\mathfrak B_n$ into the Young subgroup corresponding to the partition $[1,h]\mid[h+1,n]$. The main idea is to embed the unbalanced problem into a larger balanced one by adjoining a fixed initial segment, apply the inverse of the bijection in the balanced case, and then return to the original symmetric group. A stability property of the maps $F_k$ from Section~\ref{s:proof of balance case} will ensure that the added entries remain fixed. 

For $h\in[n-1]$, define a Young subgroup
\begin{equation}\label{def:gen Young subgroup}
\mathfrak A_{n,h}
:=
\mathfrak S_{1,\ldots,h}\times \mathfrak S_{h+1,\ldots,n}.
\end{equation}
Thus, the permutation monomials appearing in
\begin{equation*}
\operatorname{per}(X_{[1,h],[1,h]}) \cdot
\operatorname{per}(X_{[h+1,n],[h+1,n]})
\end{equation*}
are indexed by $\mathfrak A_{n,h}$. When $h=\floor*{n/2}$, this agrees with the set $\mathfrak A_n$ defined in Section~\ref{s:A gen on all}. 

To carry out this embedding argument, we need to understand how the maps $F_k$ (defined in~\eqref{def:F_k}) behave when some initial entries of the component permutations are fixed. We therefore introduce the following notation for permutations with a prescribed fixed initial segment.

For $0\leq r\leq k$, define 
\begin{equation}\label{def:fix initial symm group}
\mathfrak S_k^{[r]}
:=
\{w\in\mathfrak S_k\mid w_i=i\text{ for all }i\in[r]\}.
\end{equation}
Thus, $\mathfrak S_k^{[r]}$ consists of the permutations in $\mathfrak S_k$ that fix the initial segment $[r]$. In particular,
\begin{equation*}
\mathfrak S_k^{[0]}=\mathfrak S_k,
\qquad
\mathfrak S_k^{[k]}=\{12\cdots k\}=e.
\end{equation*}
We will refer to $\mathfrak S_k^{[r]}$ as the fixed-prefix subset of $\mathfrak S_k$ of length $r$.

The insertion relations in Proposition~\eqref{p:Ek Dk} impose a useful restriction on the behavior of $F_k$ when its image has a fixed initial segment. The following proposition records this fixed-prefix property, which will be the key to embedding an unbalanced problem into a balanced one.

\begin{prop}[Fixed-prefix property]\label{p:fixed prefix}
Let $k\geq 2$ and $0\leq m\leq k-1$. Let
$\sigma,\tau,\alpha,\beta\in\mathfrak S_k$, such that
\begin{equation*}
F_k(\sigma,\tau)=(\alpha,\beta).
\end{equation*}
If
$\alpha\in\mathfrak S_k^{[\ceil*{\tfrac{m}{2}}]}$
and
$\beta\in\mathfrak S_k^{[\floor*{\tfrac{m}{2}}]}$,
then we have
\begin{equation*}
\sigma\in\mathfrak S_k^{[m]}.
\end{equation*}
\end{prop}

\begin{proof}
We proceed by induction on $k$. When $k=2$, the case $m=0$ is immediate. For $m=1$, the conclusion follows from the fact that $F_2$ is the identity map on $\mathfrak S_2\times\mathfrak S_2$.

Now let $k\geq 3$, and assume that the proposition holds for $k-1$. Fix $0\leq m\leq k-1$. The case $m=0$ is immediate, so suppose that $m\geq 1$.

Let $d$ be the position of $1$ in $\sigma$, and let $h$ be the position of $k$ in $\tau$. There exist unique $\sigma',\tau'\in\mathfrak S_{k-1}$ such that
\begin{equation*}
\sigma=I_d^-(\sigma'),
\qquad
\tau=I_h^+(\tau').
\end{equation*}
Write
\begin{equation*}
F_{k-1}(\sigma',\tau')=(\alpha',\beta').
\end{equation*}
By $(E_k)$ in Proposition~\ref{p:Ek Dk},
\begin{equation*}
\alpha=I_d^-(\beta'),
\qquad
\beta=I_h^+(\alpha').
\end{equation*}

Since $m\geq 1$, we have $\ceil*{\tfrac{m}{2}}\geq 1$. Thus, $\alpha_1=1$. Since $I_d^-$ inserts the new minimum $1$ at position $d$, it follows that $d=1$. Therefore,
\begin{equation*}
\sigma=I_1^-(\sigma'),
\qquad
\alpha=I_1^-(\beta').
\end{equation*}
Since
\begin{equation*}
\alpha\in\mathfrak S_k^{[\ceil*{\tfrac{m}{2}}]},
\end{equation*}
we obtain
\begin{equation*}
\beta'\in
\mathfrak S_{k-1}^{[\ceil*{\tfrac{m}{2}}-1]}
=
\mathfrak S_{k-1}^{[\floor*{\tfrac{m-1}{2}}]}.
\end{equation*}

We next claim that
\begin{equation*}
h>\floor*{\frac{m}{2}}.
\end{equation*}
To see this, suppose that $h\leq\floor*{\tfrac{m}{2}}$. Since $I_h^+$ inserts the new maximum $k$ in position $h$, we have
$\beta_h=k$.
On the other hand,
$\beta\in\mathfrak S_k^{[\floor*{\tfrac{m}{2}}]}$
implies that $\beta_h=h$. This is impossible, since
\begin{equation*}
h\leq\floor*{\frac{m}{2}}\leq m\leq k-1.
\end{equation*}
Hence we have 
$h>\floor*{\tfrac{m}{2}}$.

It follows that the insertion of $k$ in $I_h^+(\alpha')$ occurs after the first $\floor*{\tfrac{m}{2}}$ positions. Since
\begin{equation*}
\beta=I_h^+(\alpha') \
\text{and} \
\beta\in\mathfrak S_k^{[\floor*{\tfrac{m}{2}}]},
\end{equation*}
we obtain
\begin{equation*}
\alpha'\in
\mathfrak S_{k-1}^{[\floor*{\tfrac{m}{2}}]}
=
\mathfrak S_{k-1}^{[\ceil*{\tfrac{m-1}{2}}]}.
\end{equation*}

Therefore, the induction hypothesis, applied to
\begin{equation*}
F_{k-1}(\sigma',\tau')=(\alpha',\beta')
\end{equation*}
with $m-1$ in place of $m$, gives
\begin{equation*}
\sigma'\in\mathfrak S_{k-1}^{[m-1]}.
\end{equation*}
Finally, since
$\sigma=I_1^-(\sigma')$,
we conclude that
$\sigma\in\mathfrak S_k^{[m]}$.
\end{proof}

We now use the fixed-prefix property of $F_k$ to embed an unbalanced problem into a larger balanced one. We first define the embedding that adjoins a fixed initial segment.

Fix $h\in[n-1]$ such that
$h\leq \floor*{\tfrac{n}{2}}$,
and set
\begin{equation*}
k=n-h,
\qquad
\ell=n-2h.
\end{equation*}
Then we have
\begin{equation*}
n=2k-\ell,
\qquad
h=k-\ell,
\end{equation*}
and $0\leq \ell\leq k-1$.

Define the map
\begin{equation}\label{def:embedding map}
\begin{aligned}
\iota_\ell:\mathfrak S_{2k-\ell}
&\to
\mathfrak S_{2k},\\
w=w_1\cdots w_{2k-\ell}
&\mapsto
12\cdots \ell
(w_1+\ell)(w_2+\ell)\cdots(w_{2k-\ell}+\ell).
\end{aligned}
\end{equation}
Thus, $\iota_\ell$ is obtained by adjoining the fixed initial segment $12\cdots\ell$ and increasing every original entry by $\ell$. 

\begin{obs}\label{o:iota bijection}
The map $\iota_\ell$ restricts to a bijection
\begin{equation*}
\iota_\ell:
\mathfrak S_{2k-\ell}
\longrightarrow
\mathfrak S_{2k}^{[\ell]}.
\end{equation*}
\end{obs}

\begin{proof}
By definition, every permutation in the image of $\iota_\ell$ fixes the first $\ell$ entries, so
\begin{equation*}
\iota_\ell(\mathfrak S_{2k-\ell})
\subseteq
\mathfrak S_{2k}^{[\ell]}.
\end{equation*}
Conversely, let $v\in\mathfrak S_{2k}^{[\ell]}$. Deleting the initial segment $12\cdots\ell$ and subtracting $\ell$ from each remaining entry gives a unique permutation
$w\in\mathfrak S_{2k-\ell}$
such that
$\iota_\ell(w)=v$.
Hence $\iota_\ell$ is a bijection from $\mathfrak S_{2k-\ell}$ to $\mathfrak S_{2k}^{[\ell]}$.
\end{proof}

For example, let $n=6$ and $h=2$. Then $k=4$ and $\ell=2$, so
\begin{equation*}
\iota_2:\mathfrak S_6\to\mathfrak S_8.
\end{equation*}
For instance,
$\iota_2(143256)=12365478$.
Here, the entries of $143256$ are increased by $2$, and the new entries $1$ and $2$ form a fixed initial segment. Also, we have $\iota_2^{-1}(12365478) = 143256$. 

The embedding $\iota_\ell$ is also compatible with the Bruhat order. 

\begin{lem}\label{l:iota Bruhat}
Let $v,w\in\mathfrak S_{2k-\ell}$. Then
$v\leq w$
if and only if
$\iota_\ell(v)\leq\iota_\ell(w)$.
\end{lem}

\begin{proof}
For $r\in[2k-\ell]$, let $v_{i,r}$ and $w_{i,r}$ denote the $i$-th entries in the increasing rearrangements of the first $r$ entries of $v$ and $w$, respectively. By the definition of $\iota_\ell$, the increasing rearrangements of the first $\ell+r$ entries of $\iota_\ell(v)$ and $\iota_\ell(w)$ are
\begin{equation*}
1,\ldots,\ell,
v_{1,r}+\ell,\ldots,v_{r,r}+\ell
\end{equation*}
and
\begin{equation*}
1,\ldots,\ell,
w_{1,r}+\ell,\ldots,w_{r,r}+\ell,
\end{equation*}
respectively. Therefore, by the tableau criterion (Theorem~\ref{TabCriterion}),
$v\leq w$
if and only if
$\iota_\ell(v)\leq\iota_\ell(w)$.
\end{proof}

We next describe how the embedding $\iota_\ell$ acts on $\mathfrak A_{n,h}$ and $\mathfrak B_n$. We begin with $\mathfrak A_{n,h}$.

\begin{lem}\label{l:iota A}
Let $h\in[n-1]$ with
$h\leq\floor*{\tfrac{n}{2}}$,
and set $k=n-h$, $\ell=n-2h$.
Then 
\begin{equation*}
\iota_\ell\left(\mathfrak A_{n,h}\right)
=
\left\{
A_k(\sigma,\tau)
\mid
\sigma\in\mathfrak S_k^{[\ell]},
\ \tau\in\mathfrak S_k
\right\}.
\end{equation*}
\end{lem}

\begin{proof}
Let $w\in\mathfrak A_{n,h}$. Since $h=k-\ell$ and $n-h=k$, the first $k-\ell$ entries of $w$ permute $[1,k-\ell]$, while the remaining $k$ entries permute $[k-\ell+1,n]$. Hence, in
\begin{equation*}
\iota_\ell(w)
=
12\cdots\ell\,
(w_1+\ell)\cdots(w_n+\ell),
\end{equation*}
the entries
$(w_1+\ell)\cdots(w_h+\ell)$
permute $[\ell+1,\ell +h] = [\ell+1,k]$. Therefore, the first $k$ entries of $\iota_\ell(w)$ permute $[k]$ and its first $\ell$ entries are $12\cdots\ell$. The last $k$ entries permute $[k+1,2k]$. Therefore, there exist
\begin{equation*}
\sigma\in\mathfrak S_k^{[\ell]},
\qquad
\tau\in\mathfrak S_k
\end{equation*}
such that
\begin{equation*}
\iota_\ell(w)=A_k(\sigma,\tau).
\end{equation*}

Conversely, let $A_k(\sigma,\tau)$ satisfy
\begin{equation*}
\sigma\in\mathfrak S_k^{[\ell]},
\qquad
\tau\in\mathfrak S_k.
\end{equation*}
Since $\sigma$ fixes the first $\ell$ entries, deleting the initial segment $12\cdots\ell$ from $A_k(\sigma,\tau)$ and subtracting $\ell$ from each remaining entry produces a permutation whose first $k-\ell=h$ entries permute $[h]$, while the remaining $k=n-h$ entries permute $[h+1,n]$. Therefore, the resulting permutation lies in $\mathfrak A_{n,h}$.
Hence
$A_k(\sigma,\tau)
\in
\iota_\ell\left(\mathfrak A_{n,h}\right)$,
which proves the result.
\end{proof}

The preceding lemma describes the image of $\mathfrak A_{n,h}$ under $\iota_\ell$. We now turn to $\mathfrak B_n$. 

\begin{lem}\label{l:iota B}
Let $h\in[n-1]$ with
$h\leq\floor*{\tfrac{n}{2}}$,
and set $k=n-h$, $\ell=n-2h$.
Then
\begin{equation*}
\iota_\ell(\mathfrak B_n)
=
\left\{
B_k(\alpha,\beta)
\mid
\alpha\in\mathfrak S_k^{[\ceil*{\frac{\ell}{2}}]},
\ \beta\in\mathfrak S_k^{[\floor*{\frac{\ell}{2}}]}
\right\}.
\end{equation*}
\end{lem}

\begin{proof}
Let $u=u_1\cdots u_n\in\mathfrak B_n$. Since $u_i$ and $i$ have the same parity for every $i\in[n]$, the integers $u_i+\ell$ and $i+\ell$ also have the same parity. Together with the fixed initial segment $12\cdots\ell$, this shows that
\begin{equation*}
\iota_\ell(u)\in\mathfrak B_{2k}.
\end{equation*}
Thus, there exist $\alpha,\beta\in\mathfrak S_k$ such that
$\iota_\ell(u)=B_k(\alpha,\beta)$.

Among the first $\ell$ entries of $\iota_\ell(u)$, the odd entries are
\begin{equation*}
1,3,\ldots,2\ceil*{\tfrac{\ell}{2}}-1,
\end{equation*}
while the even entries are
\begin{equation*}
2,4,\ldots,2\floor*{\tfrac{\ell}{2}}.
\end{equation*}
By the definition of $B_k(\alpha,\beta)$, it follows that
\begin{equation*}
\alpha\in\mathfrak S_k^{[\ceil*{\frac{\ell}{2}}]},
\qquad
\beta\in\mathfrak S_k^{[\floor*{\frac{\ell}{2}}]}.
\end{equation*}
Hence
\begin{equation*}
\iota_\ell(\mathfrak B_n)
\subseteq
\left\{
B_k(\alpha,\beta)
\mid
\alpha\in\mathfrak S_k^{[\ceil*{\ell/2}]},
\ \beta\in\mathfrak S_k^{[\floor*{\ell/2}]}
\right\}.
\end{equation*}

Conversely, let
$B_k(\alpha,\beta)$
satisfy
\begin{equation*}
\alpha\in\mathfrak S_k^{[\ceil*{\frac{\ell}{2}}]},
\qquad
\beta\in\mathfrak S_k^{[\floor*{\frac{\ell}{2}}]}.
\end{equation*}
Then the first $\ell$ entries of $B_k(\alpha,\beta)$ are $12\cdots\ell$. Deleting these entries and subtracting $\ell$ from every remaining entry produces a permutation $u\in\mathfrak S_n$. Since both the positions and the corresponding values are shifted by $\ell$, their parities continue to agree. Therefore,
\begin{equation*}
u\in\mathfrak B_n.
\end{equation*}
By construction,
\begin{equation*}
\iota_\ell(u)=B_k(\alpha,\beta),
\end{equation*}
which proves the reverse inclusion.
\end{proof}

The preceding two lemmas, together with Proposition~\ref{p:fixed prefix}, show that the inverse of the bijection in the balanced case preserves the embedded families in the following sense.

\begin{cor}\label{c:embedded inverse}
Let $h\in[n-1]$ with
$h\leq\floor*{\frac{n}{2}}$,
and set $k=n-h$, $\ell=n-2h$.
Then
\begin{equation*}
f_{2k}^{-1}\left(\iota_\ell(\mathfrak B_n)\right)
\subseteq
\iota_\ell\left(\mathfrak A_{n,h}\right).
\end{equation*}
\end{cor}

\begin{proof}
Let $u\in\mathfrak B_n$. By Lemma~\ref{l:iota B}, there exist
\begin{equation*}
\alpha\in\mathfrak S_k^{[\ceil*{\frac{\ell}{2}}]},
\qquad
\beta\in\mathfrak S_k^{[\floor*{\frac{\ell}{2}}]}
\end{equation*}
such that
$\iota_\ell(u)=B_k(\alpha,\beta)$.
Write
\begin{equation*}
f_{2k}^{-1}\left(\iota_\ell(u)\right)
=
A_k(\sigma,\tau).
\end{equation*}
By the definition of $F_k$, we have
\begin{equation*}
F_k(\sigma,\tau)=(\alpha,\beta).
\end{equation*}
Therefore, Proposition~\ref{p:fixed prefix} gives
\begin{equation*}
\sigma\in\mathfrak S_k^{[\ell]}.
\end{equation*}
By Lemma~\ref{l:iota A}, it follows that
\begin{equation*}
A_k(\sigma,\tau)
\in
\iota_\ell\left(\mathfrak A_{n,h}\right).
\end{equation*}
Hence
\begin{equation*}
f_{2k}^{-1}\left(\iota_\ell(u)\right)
\in
\iota_\ell\left(\mathfrak A_{n,h}\right),
\end{equation*}
which proves the result.
\end{proof}

Corollary~\ref{c:embedded inverse} shows that the inverse of the bijection in the balanced case maps $\iota_\ell(\mathfrak B_n)$ into $\iota_\ell(\mathfrak A_{n,h})$. Therefore, we may return to the original symmetric group by applying $\iota_\ell^{-1}$. %This allows us to define the map that will be used to compare the permutation monomials in the general case.

For $h\leq\floor*{\tfrac{n}{2}}$, set
\begin{equation*}
k=n-h,
\qquad
\ell=n-2h.
\end{equation*}
Define
\begin{equation*}
\begin{aligned}
g_{n,h}:\mathfrak B_n
&\to
\mathfrak A_{n,h},\\
u
&\mapsto
\iota_\ell^{-1}\left(
f_{2k}^{-1}\left(\iota_\ell(u)\right)
\right).
\end{aligned}
\end{equation*}
By Corollary~\ref{c:embedded inverse}, the map $g_{n,h}$ is well-defined. 

For example, let $n=5$ and $h=2$. Then $k=3$ and $\ell=1$. Consider
\begin{equation*}
u=52341\in\mathfrak B_5.
\end{equation*}
Applying the embedding $\iota_1$, we obtain
\begin{equation*}
\iota_1(u)=163452.
\end{equation*}
From the table of $f_6$ in Section~\ref{s:A gen on all}, we have
\begin{equation*}
f_6(132645)=163452.
\end{equation*}
Therefore,
\begin{equation*}
f_6^{-1}\left(\iota_1(u)\right)=132645.
\end{equation*}
Applying $\iota_1^{-1}$ gives
\begin{equation*}
g_{5,2}(52341)
=
\iota_1^{-1}(132645)
=
21534.
\end{equation*}
Thus,
\begin{equation*}
52341
\xmapsto{\iota_1}
163452
\xmapsto{f_6^{-1}}
132645
\xmapsto{\iota_1^{-1}}
21534.
\end{equation*}
In particular, $21534\in\mathfrak A_{5,2}$. 

The preceding example illustrates the construction of $g_{n,h}$. We now establish its basic properties, beginning with injectivity.

\begin{lem}\label{l:g injective}
Let $h\in[n-1]$ with
$h\leq\floor*{\frac{n}{2}}$.
Then the map
\begin{equation*}
g_{n,h}:\mathfrak B_n\to\mathfrak A_{n,h}
\end{equation*}
is injective.
\end{lem}

\begin{proof}
We prove it by the definition of an injection. 
Let $u,v\in\mathfrak B_n$, and suppose that
\begin{equation*}
g_{n,h}(u)=g_{n,h}(v).
\end{equation*}
By the definition of $g_{n,h}$,
\begin{equation*}
\iota_\ell^{-1}
\left(
f_{2k}^{-1}\left(\iota_\ell(u)\right)
\right)
=
\iota_\ell^{-1}
\left(
f_{2k}^{-1}\left(\iota_\ell(v)\right)
\right).
\end{equation*}
Since $\iota_\ell$ is a bijection onto $\mathfrak S_{2k}^{[\ell]}$, it follows that
\begin{equation*}
f_{2k}^{-1}\left(\iota_\ell(u)\right)
=
f_{2k}^{-1}\left(\iota_\ell(v)\right).
\end{equation*}
Since $f_{2k}$ is a bijection, we obtain
\begin{equation*}
\iota_\ell(u)=\iota_\ell(v).
\end{equation*}
Finally, since $\iota_\ell$ is injective, we conclude that
\begin{equation*}
u=v.
\end{equation*}
Hence $g_{n,h}$ is injective.
\end{proof}

Injectivity alone is not sufficient for our purpose. To compare the corresponding permutation monomials using Theorem~\ref{pairing test}, we also need to compare each $u\in\mathfrak B_n$ with its image under $g_{n,h}$ in the Bruhat order. The following lemma establishes the required comparison.

\begin{lem}\label{l:g Bruhat}
Let $h\in[n-1]$ with
$h\leq\floor*{\frac{n}{2}}$.
Then, for every $u\in\mathfrak B_n$, we have
\begin{equation*}
g_{n,h}(u)\leq u
\end{equation*}
under the Bruhat order.
\end{lem}

\begin{proof}
Set
\begin{equation*}
k=n-h,
\qquad
\ell=n-2h.
\end{equation*}
Let $u\in\mathfrak B_n$, and write
\begin{equation*}
v=f_{2k}^{-1}\left(\iota_\ell(u)\right).
\end{equation*}
By the definition of $g_{n,h}$,
\begin{equation*}
v=\iota_\ell\left(g_{n,h}(u)\right).
\end{equation*}
Since $v\in\mathfrak A_{2k}$, Theorem~\ref{t:Bruhat balance case} gives
\begin{equation*}
v\leq f_{2k}(v)=\iota_\ell(u).
\end{equation*}
Hence
\begin{equation*}
\iota_\ell\left(g_{n,h}(u)\right)
\leq
\iota_\ell(u).
\end{equation*}
Since $n = 2k-\ell$, applying Lemma~\ref{l:iota Bruhat}, we immediately obtain 
\begin{equation*}
    g_{n,h}(u)
\leq
u,
\end{equation*}
as desired. 
\end{proof}

We have established a Bruhat-order property of $g_{n,h}$ when
$h\leq \floor*{\frac{n}{2}}$. 
To treat the remaining values of $h$, we use the $RU$-symmetry introduced in Section~\ref{s:A gen on all}. We first record how $RU$ acts on the sets $\mathfrak A_{n,h}$.

\begin{lem}\label{l:RU A general}
Let $h\in[n-1]$. Then the map $RU$ restricts to a bijection
\begin{equation*}
\mathfrak A_{n,h}
\longrightarrow
\mathfrak A_{n,n-h}.
\end{equation*}
\end{lem}

\begin{proof}
Let $w=w_1\cdots w_n\in\mathfrak A_{n,h}$. Then we have 
\begin{equation*}
\{w_1,\ldots,w_h\}=[h], 
\ \text{and} \
\{w_{h+1},\ldots,w_n\}=[h+1,n].
\end{equation*}
Write
\begin{equation*}
w^{RU}
=
(n+1-w_n)\cdots(n+1-w_1).
\end{equation*}
The first $n-h$ entries of $w^{RU}$ are obtained from
\begin{equation*}
w_{h+1},\ldots,w_n
\end{equation*}
by reversing their order and replacing each entry $x$ by $n+1-x$. Since
\begin{equation*}
\{w_{h+1},\ldots,w_n\}=[h+1,n],
\end{equation*}
these entries permute $[n-h]$. Similarly, the last $h$ entries of $w^{RU}$ permute
$[n-h+1,n]$.
Therefore,
\begin{equation*}
w^{RU}\in\mathfrak A_{n,n-h}.
\end{equation*}

Since $RU$ is an involution, the reverse inclusion follows in the same way. Hence $RU$ restricts to a bijection between $\mathfrak A_{n,h}$ and $\mathfrak A_{n,n-h}$.
\end{proof}

We now extend the definition of $g_{n,h}$ to the remaining values of $h$. Suppose that
\begin{equation*}
h>\floor*{\frac{n}{2}}.
\end{equation*}
Then
\begin{equation*}
n-h\leq\floor*{\frac{n}{2}},
\end{equation*}
so the map
\begin{equation*}
g_{n,n-h}:\mathfrak B_n\to\mathfrak A_{n,n-h}
\end{equation*}
has already been defined. Since $RU$ restricts to an involution on $\mathfrak B_n$ and, by Lemma~\ref{l:RU A general}, gives a bijection from $\mathfrak A_{n,n-h}$ to $\mathfrak A_{n,h}$, we define
\begin{equation*}
\begin{aligned}
g_{n,h}:\mathfrak B_n
&\to
\mathfrak A_{n,h},\\
u
&\mapsto
\left(g_{n,n-h}\left(u^{RU}\right)\right)^{RU}.
\end{aligned}
\end{equation*}
Thus, the maps $g_{n,h}$ are now defined for every $h\in[n-1]$. 

The construction above gives the following result for arbitrary $h\in[n-1]$.

\begin{prop}\label{p:g general}
Let $h\in[n-1]$. Then the map
\begin{equation*}
g_{n,h}:\mathfrak B_n\to\mathfrak A_{n,h}
\end{equation*}
is injective. Moreover, for every $u\in\mathfrak B_n$, we have
\begin{equation*}
g_{n,h}(u)\leq u
\end{equation*}
under the Bruhat order.
\end{prop}

\begin{proof}
If
$h\leq\floor*{\frac{n}{2}}$,
the result follows from Lemmas~\ref{l:g injective} and~\ref{l:g Bruhat}.

Now suppose that
$h>\floor*{\frac{n}{2}}$.
Then $n-h\leq\floor*{\frac{n}{2}}$, so $g_{n,n-h}$ is injective. Since $RU$ is an involution on $\mathfrak B_n$ and on $\mathfrak S_n$, the definition
\begin{equation*}
g_{n,h}(u)
=
\left(g_{n,n-h}\left(u^{RU}\right)\right)^{RU}
\end{equation*}
immediately implies that $g_{n,h}$ is injective.

Furthermore, by Lemma~\ref{l:g Bruhat},
\begin{equation*}
g_{n,n-h}\left(u^{RU}\right)
\leq
u^{RU}.
\end{equation*}
Since $RU$ preserves the Bruhat order, applying $RU$ to both sides gives
\begin{equation*}
\left(g_{n,n-h}\left(u^{RU}\right)\right)^{RU}
\leq
u.
\end{equation*}
By the definition of $g_{n,h}$, this is
$g_{n,h}(u)\leq u$.
Hence the result holds for all $h\in[n-1]$.
\end{proof}

The preceding proposition allows us to prove the permanental inequality for arbitrary $h$.

\begin{thm}\label{t:general h TNN}
Let $A$ be an $n\times n$ TNN matrix. Then, for every $h\in[n-1]$,
\begin{equation*}
\operatorname{per}(A_{[1,h],[1,h]})
\operatorname{per}(A_{[h+1,n],[h+1,n]})
\geq
\operatorname{per}(A_{[1,n]\cap 2\mathbb Z,[1,n]\cap 2\mathbb Z})
\operatorname{per}(A_{[1,n]\setminus 2\mathbb Z,[1,n]\setminus 2\mathbb Z}).
\end{equation*}
\end{thm}

\begin{proof}
For $w\in\mathfrak S_n$, write
\begin{equation*}
x_w=x_{1,w_1}\cdots x_{n,w_n}.
\end{equation*}
Since $g_{n,h}:\mathfrak B_n\to\mathfrak A_{n,h}$ is injective, we have
\begin{equation*}
\sum_{w\in\mathfrak A_{n,h}}x_w
-
\sum_{u\in\mathfrak B_n}x_u
=
\sum_{u\in\mathfrak B_n}
\left(x_{g_{n,h}(u)}-x_u\right)
\quad+
\sum_{w\in\mathfrak A_{n,h}\setminus g_{n,h}(\mathfrak B_n)}
x_w.
\end{equation*}
By Proposition~\ref{p:g general}, we have
$g_{n,h}(u)\leq u$
for every $u\in\mathfrak B_n$. Hence, by Theorem~\ref{pairing test}, each polynomial
$x_{g_{n,h}(u)}-x_u$
is TNN. The remaining monomials are nonnegative on every TNN matrix. Therefore, the entire difference is nonnegative on every TNN matrix, which proves the result.
\end{proof}

\section{Open Problems}\label{s:open}

The results of the preceding sections suggest several directions for further investigation.

One possible direction is to seek higher-order analogs of
Theorem~\ref{t:main}. 
It would be interesting to determine whether a
generalization of the merge function defined in Section~\ref{s:merge}
can be used to prove the following conjecture.
\begin{conj}
    Let $A$ be the antiadjacency matrix of an $n$-element unit interval order. Define sets 
    \begin{equation*}
    I= [1,n]_0^3, \quad J = [1,n]_1^3, \quad K = [1,n]_2^3. 
    \end{equation*} 
    Then, for all $1\leq h_1 < h_2 \leq n$, we have 
    \begin{equation*}
        \perm(A_{[1,h_1],[1,h_1]}) \perm(A_{[h_1,h_2],[h_1,h_2]}) \perm(A_{[h_2,n],[h_2,n]}) \geq \perm(A_{I,I}) \perm(A_{J,J}) \perm(A_{K,K}). 
    \end{equation*}
\end{conj}

Theorem~\ref{t:maintwo} cannot, in its stated form, be extended to all
totally nonnegative matrices.
It would nevertheless be interesting to characterize the totally
nonnegative matrices for which its inequality continues to hold.

\begin{prob}
    What kind of TNN matrices satisfy the inequality in Theorem~\ref{t:maintwo}? 
\end{prob}

The results of Section~\ref{s:A gen on J} suggest asking whether there is
an analogous higher-order inequality for arbitrary totally nonnegative
matrices.

\begin{prob}
Determine whether the inequalities in Section~\ref{s:A gen on J},
involving products of permanents indexed by residue classes modulo
$s$, admit analogs for arbitrary totally nonnegative matrices.
\end{prob}

More generally, motivated by the determinant--permanent correspondence
discussed in the introduction, it would be interesting to characterize
the pairs of subsets that give rise to inequalities whose directions
are reversed when determinants are replaced by permanents.
\begin{prob}
Decide which sets $L,K \subseteq [n]$ satisfy 

\begin{equation*}
\perm(A_{L,L})\perm(A_{\overline{L},\overline{L}}) \geq \perm(A_{K,K})\perm(A_{\overline{K},\overline{K}})
    \end{equation*} and \begin{equation*}
\det(A_{L,L})\det(A_{\overline{L},\overline{L}}) \leq \det(A_{K,K})\det(A_{\overline{K},\overline{K}})
    \end{equation*} for all TNN matrices. 
\end{prob}

The maps $f_n$ and $g_{n,h}$ introduced in Sections~\ref{s:A gen on all}
and~\ref{s:all h} may also be of independent combinatorial interest.
Although they are constructed here in order to prove permanental
inequalities, their defining properties suggest several natural
questions about their interaction with Bruhat order.

\begin{prob}
Investigate the order-theoretic structure of the maps $f_n$ and
$g_{n,h}$. In particular, determine whether they admit a uniform
description, whether they preserve Bruhat order, and describe the
Bruhat intervals determined by $w$ and its image under these maps.
\end{prob}

\section*{Acknowledgements}
The authors would like to thank Tommy Parisi for providing several useful examples and Nut Suebthawinkul for testing examples via MATLAB. The authors are especially grateful to Jules Falcone for a helpful and illuminating email exchange. We also thank Jonathan Bloom, Matt Jiang, Nicole Newfield, Hanqing Qi, Daniel Soskin, and Fangzheng Yu from Skandera's research course for their careful reading and helpful suggestions.

\section{Appendix}

Here is a MATLAB code to generate examples from Theorem~\ref{t:Bruhat balance case}. 
\url{https://github.com/SihongPan/General-TNN-Matrices-Per-Top-Bottom--Per-Odd-Even-Test/blob/main/General_TNN_Matrices_Per_Ineq.m}

\bibliography{skan}
\end{document}